\documentclass{article}
\usepackage[a4paper, total={6in, 8in}]{geometry}
\usepackage{amsfonts,amsmath,amsthm,graphicx,mathrsfs}
\usepackage{verbatim}
\usepackage[toc,page]{appendix}
\usepackage{subcaption}
\usepackage[hidelinks]{hyperref}
\numberwithin{equation}{section}

\newcommand{\matdev}{\partial^{\bullet}}
\newcommand{\normdev}{\partial^{\circ}}
\newcommand{\pioladev}{\partial^{*}}
\newcommand{\ddt}[1]{\frac{\partial #1}{\partial t}}
\newcommand{\gradg}{\nabla_{\Gamma}}
\newcommand{\D}{\underline D}

\newcommand{\lapg}{\Delta_{\Gamma}}
\newcommand{\divg}{\gradg \cdot}
\newcommand{\Ech}{E^{\mathrm{CH}}}
\newcommand{\Echd}{E^{\mathrm{CH},\delta}}
\newcommand{\Sperp}{\mathcal{S}^{\perp}}
\newcommand{\PK}{P_{\mathcal{K}}}
\newcommand{\PKperp}{P_{\mathcal{K}^\perp}}

\newcommand{\Xint}[1]{\mathchoice
	{\XXint\displaystyle\textstyle{#1}}%
	{\XXint\textstyle\scriptstyle{#1}}%
	{\XXint\scriptstyle\scriptscriptstyle{#1}}%
	{\XXint\scriptscriptstyle\scriptscriptstyle{#1}}%
	\!\int}
\newcommand{\XXint}[3]{{\setbox0=\hbox{$#1{#2#3}{\int}$ }
		\vcenter{\hbox{$#2#3$ }}\kern-.6\wd0}}

\newcommand{\dashint}{\Xint-}
\newcommand{\mval}[2]{\dashint_{#2} #1}
\DeclareMathOperator*{\esssup}{ess\,sup}

\newcommand{\mbf}[1]{\mathbf{#1}}
\newcommand{\mbb}[1]{\mathbb{#1}}
\newcommand{\ut}{\mbf{u}_T}
\newcommand{\hatut}{\hat{\mbf{u}}_T}
\newcommand{\divfree}[1]{\mbf{V}_\sigma(#1)}
\newcommand{\Hdivfree}[1]{\mbf{H}_\sigma(#1)}
\newcommand{\fd}{f^\delta}

\theoremstyle{plain}
\newtheorem{theorem}{Theorem}[section]

\newtheorem{proposition}[theorem]{Proposition}
\newtheorem{lemma}[theorem]{Lemma}

\newtheorem{remark}[theorem]{Remark}
\newtheorem{assumption}[theorem]{Assumption}

\begin{document}

%opening
\title{Navier-Stokes-Cahn-Hilliard equations on evolving surfaces}
\author{Charles M. Elliott\thanks{Mathematics Institute, Zeeman Building, University of Warwick, Coventry CV4 7AL, UK, \href{mailto:c.m.elliott@warwick.ac.uk}{c.m.elliott@warwick.ac.uk}}\\
	Thomas Sales\thanks{Mathematics Institute, Zeeman Building, University of Warwick, Coventry CV4 7AL, UK, \href{mailto:tom.sales@warwick.ac.uk}{tom.sales@warwick.ac.uk}}}
\date{}

\maketitle

\begin{abstract}
	We derive a system of equations which can be seen as an evolving surface version of the diffuse interface ``Model H'' of Hohenberg and Halperin (1977).
	We then consider the well-posedness for the corresponding (tangential) system when one prescribes the evolution of the surface.
	Well-posedness is proved for smooth potentials in the Cahn-Hilliard equation with polynomial growth, and also for a thermodynamically relevant singular potential.

\end{abstract}

\section{Introduction}
The Navier-Stokes-Cahn-Hilliard system on a sufficiently smooth, closed, oriented, evolving hypersurface, $(\Gamma(t))_{t \in [0,T]} \subset \mbb{R}^3$, %is given by
\begin{gather} 
	\rho \matdev \mbf{u} = -\gradg p + pH \boldsymbol{\nu} + \gradg \cdot (2 \eta(\varphi) \mbb{E}(\mbf{u}) ) - \varepsilon \gradg \cdot (\gradg \varphi \otimes \gradg \varphi) +\mathbf{F}\label{NSCH1},\\
	\gradg \cdot \mathbf{u} = 0 \label{NSCH2},\\
	\matdev \varphi = \gradg \cdot (M(\varphi) \gradg \mu),\label{NSCH3}\\
	\mu = -\varepsilon\lapg \varphi + \frac{1}{\varepsilon}F'(\varphi)\label{NSCH4}
\end{gather}
 is  derived in analogy  to the ``Model H'' of Hohenberg and Halperin as in \cite{hohenberg1977theory} and as a thin film limit of the relevant system in a thin evolving Cartesian domain as in \cite{miura2018singular}. In addition we provide a well posedness result in the case of a prescribed evolving surface.\\

Here
\[ \mbb{E}(\mbf{u}) = \frac{1}{2} \left(\gradg \mbf{u} +  (\gradg \mbf{u})^T \right) \]
denotes the rate of strain tensor.
We assume a matched density and  denote the constant density by $\rho$. The velocity of the surface is denoted by $\mbf{u}$ which will often be decomposed as $\mbf{u} = \mbf{u}_\nu + \mbf{u}_T$, where $\mbf{u}_\nu, \mbf{u}_T$ are the normal and tangential components respectively.
The associated pressure is denoted as $p$, $\eta(\cdot)$ is the variable viscosity which depends on $\varphi$, and $\mbf{F}$ is some external force.
We have split the fourth order Cahn-Hilliard equation into two second order equations \eqref{NSCH3}, \eqref{NSCH4} for $\varphi,\mu$ where $\mu$ is the associated chemical potential.
In our derivations we consider a non-constant mobility function $M(\cdot)$, but our analysis will consider a constant mobility $M(\cdot) \equiv 1$.
Lastly, $\boldsymbol{\nu}$ denotes the outer unit normal to $\Gamma(t)$, and $H = \gradg \cdot \boldsymbol{\nu}$ the mean curvature of $\Gamma(t)$ --- with the convention that a sphere has positive mean curvature.
The differential operators used will be discussed below.\\

As in \cite{JanOlsReu18,OlsReuZhi22}, one may be interested in this system where the normal component of the surface evolution is known a priori.
That is, $\mathbf{u}_\nu = V_N \boldsymbol{\nu} $ for some, sufficiently smooth, known function $V_N$.
For our theory, it is sufficient to assume that $V_N$ is a $C^3$ function, and so we are working on a $C^3$ evolving surface, $\Gamma(t)$.
We assume that this surface is such that $|\Gamma(t)| = |\Gamma_0|$, which is equivalent to assuming that
\[ \int_{\Gamma(t)} HV_N = 0, \]
for all $t \in [0,T]$.
In this case one obtains the tangential Navier-Stokes-Cahn-Hilliard system
\begin{gather}
	\rho \left( \mathbb{P} \normdev \ut + (\gradg \ut)\ut + V_N \mathbb{H} \ut - \frac{1}{2} \gradg V_N^2 \right) = -\gradg \tilde{p} + \mbb{P} \gradg \cdot (2 \eta(\varphi) \mbb{E}(\ut) ) + \mu \gradg \varphi + \mbf{F}_T\label{TNSCH1},\\
	\gradg \cdot \ut = -H V_N \label{TNSCH2},\\
	\normdev \varphi + \gradg \varphi \cdot \ut = \gradg \cdot (M(\varphi) \gradg \mu),\label{TNSCH3}\\
	\mu = -\varepsilon\lapg \varphi + \frac{1}{\varepsilon}F'(\varphi).\label{TNSCH4}
\end{gather}
Here $\tilde{p} = p + \frac{\varepsilon}{2}|\gradg \varphi|^2 + \frac{1}{\varepsilon} F(\varphi)$ is a modified pressure.
In both systems the pressure, and modified pressure, are unknown due to the incompressibility constraint.\\

Variants of this model have been considered on (mainly stationary) surfaces in \cite{arroyo2009relaxation,bachini2023interplay, olshanskii2022comparison,palzhanov2021decoupled,sun2022modeling,yang2020phase}.
The main focus in the existing literature is on the derivation and numerical simulation of such a system, but there has been little consideration for the well-posedness thus far.
As such, our work considers the well-posedness of a somewhat simpler model (surface evolution notwithstanding) which still captures the main features of the coupling of the Navier-Stokes equations with the Cahn-Hilliard equation.
Lastly we note that the model in \cite{bachini2023derivation} considers the influence of a physically relevant bending energy, and the model in \cite{palzhanov2021decoupled} consider a variable density --- in accordance with the derivation in \cite{abels2012thermodynamically}.\\

{\textbf{Some geometric differential notation}}
The evolving surface, $\Gamma(t)$, is assumed to be sufficiently smooth with a normal, $\boldsymbol{\nu}$, and normal velocity, $V_N \boldsymbol{\nu}$.
Geometric quantities and differential operators  are defined by
\begin{gather*}
	\mbb{P} = \mbb{I} - \boldsymbol{\nu} \otimes \boldsymbol{\nu}, \\
	\gradg \phi = \mbb{P} \nabla \phi^e,\quad(\gradg \phi)_i=\D_i\phi,\quad
	\gradg \mbf{v} = \mbb{P} \nabla \mbf{v}^e \mbb{P},\quad(\gradg \mbf{v})_{i,j}:=\D_jv_i\\
(\gradg  \mbf{v}\cdot \mbf{w})_i=\D_jv_iw_j,\quad	\gradg \cdot \mbf{v} = \text{tr}(\gradg \mbf{v}),\quad \mbb{H} = \gradg \boldsymbol{\nu},
\end{gather*}
where we have used Einstein summation convention.
Here $\phi$ and $\mbf{v}, \mbf{w}$ are scalar  and vector fields respectively,  $\mbb{I}$ denotes the identity matrix, and $(\cdot)^e$ denotes an extension onto a neighbourhood of $\Gamma(t)$.
These expressions are independent of the specific choice of extension.
Likewise we use the following notation for the normal time derivatives of scalar and vector fields,
\begin{gather*}
\normdev \phi = \ddt{\phi^e} + V_N\nabla \phi^e\cdot \boldsymbol{\nu},\quad	\normdev \mbf{v} = \ddt{\mbf{v}^e} + V_N\nabla \mbf{v}^e\cdot\boldsymbol{\nu}
	%\matdev \mbf{u} = \normdev \mbf{u} + (\gradg \mbf{u}) \mbf{u}_T,
\end{gather*}
which are again independent of choice of extension. Note that this is the time derivative along a trajectory evolving in the normal direction. 
We use $\matdev$ to note the time derivative which follows also the physical tangential flow  $\ut$, 
\begin{gather*}
\matdev \phi = \normdev \phi + \mbf{u}_T\cdot \gradg \phi,\quad	\matdev \mbf{v} =\normdev \mbf{v} + \gradg \mbf{v}\cdot \mbf{u}_T,
\end{gather*}

On the other hand if we wish to use another tangential flow for a velocity field $\mbf{w}=V_N\boldsymbol{\nu}+\mbf{w}_T$ with a  tangential vector field $\mbf{w}_T$, we  write
\[ \partial_{\mbf{w}}^\bullet \phi =  \normdev \phi + \gradg \phi \cdot \mbf{w}_T,\quad \partial_{\mbf{w}}^\bullet \mbf{v} =  \normdev \mbf{v} + \gradg \mbf{v}\cdot \mbf{w}_T. \]
Throughout we will use the convention that a vector quantity will be denoted in bold, e.g. $\mbf{v}$, and a tensor quantity will be denoted in blackboard bold, e.g. $\mbb{P}$.

\subsection*{Applications}

On a stationary Euclidean domain, the Navier-Stokes-Cahn-Hilliard equations has found many applications, for example in studying thermocapillary flows \cite{antanovskii1995phase}, and spinodal decomposition \cite{gurtin1996two}.
For further details and applications we refer the reader to \cite{anderson1998diffuse}.
A more recent application to a modified version of this system has been to the study of chemotaxis, for example in modelling tumour growth \cite{lam2018thermodynamically}.
In this case there are suitable changes to allow for a transfer of mass - adding further complications.\\

Another biological application which has been of interest in recent years is in studying lipid bilayer membranes.
It is known that the curvature of the domain influences the dynamics of lipid membranes and enters through a bending energy, for example the J\"ulicher-Lipowsky energy
\[ E^\text{H}[\varphi] = \int_{\Gamma}\frac{1}{2}\kappa(\varphi)(H-H_0(\varphi))^2, \]
as presented in \cite{julicher1996shape}.
Here $\kappa(\cdot)$ is the bending stiffness, and $H_0(\cdot)$ is the spontaneous curvature which depends on the diffuse interface - an example being $H_0(\varphi) = \Lambda \varphi$, for some curvature coefficient $\Lambda \in \mbb{R}$.
We refer the reader to \cite{deserno2015fluid,mcmahon2005membrane} for a discussion of these mechanisms, and to \cite{hatcher2020phase} for discussion and analysis of diffuse interface models for phase separation on biological membranes.
In \cite{fan2010hydrodynamic} the authors consider a model for the kinetics of a lipid bilayer membrane, coupling the Cahn-Hilliard equation to the stationary Stokes equations in a planar domain.
This model has been extended to surfaces in the recent works \cite{bachini2023derivation,bachini2023interplay} where the authors also include the contributions of the relevant bending energy, as well as considering the full time-dependent Navier-Stokes equations for the hydrodynamics.
We also refer the reader to \cite{arroyo2009relaxation} and the references within.\\

Lastly we mention a different model of interest \cite{barrett2017finite}, in which the authors consider the Cahn-Hilliard equations coupled with the Navier-Stokes equations on a free surface determined by the Navier-Stokes equation in a bulk domain.
This work also proposes semi-discrete and fully-discrete numerical schemes, and contains numerical examples.

\subsection*{Contributions and outline}
The contributions of this paper are to provide two equivalent derivations of a diffuse interface model coupling the Navier-Stokes equations and the Cahn-Hilliard equation on an evolving surface, and to extend existing analysis for the analogous system on a stationary, Euclidean domain to an evolving surface.
The main results are showing existence and uniqueness of weak solutions. The system is derived in Section \ref{derivations}.  Some necessary notation , functional analysis and useful inequalities are provided in Section \ref{Preliminaries}. Statements of existence and uniqueness are provided in Section \ref{Posedness} for both smooth and logarithmic Cahn-Hilliard potentials. Existence is proved in Section \ref{Exist} and uniqueness in Section \ref{Unique}. Existence and uniqueness of a mixed formulation involving the  pressure  is shown in Section \ref{Pressure}. Section \ref{Conc} contains some concluding remarks about future directions.   Appendix \ref{evolvinglaplace}  and Appendix \ref{inversestokes} concern analytic results for the Laplace operator and inverse Stokes-type operator.

\section{Derivation of the surface Navier-Stokes-Cahn-Hilliard system}
\label{derivations}
In this section we provide two derivations \eqref{TNSCH1}-\eqref{TNSCH4} - one by surface balance laws, and the other by considering a thin film limit.

\subsection{Derivation by balance laws}
We follow a similar presentation to that of \cite{gurtin1996two} and derive \eqref{NSCH1}-\eqref{NSCH4} by using a balance of microstresses.
Consider a binary mixture of a fluid, with constituent densities $\rho_1, \rho_2$.
The total density
\[\rho := \rho_1 + \rho_2,\]
is assumed to be constant.
We define $c_i := \frac{\rho_i}{\rho}$ to be the corresponding concentration, so that $c_1 + c_2 = 1$.
Following the assumption of \cite{gurtin1996two}, we assume that the momenta and kinetic energies of the constituent components are negligible when computed relative to the gross motion of the fluid.
As such, we consider the gross velocity, $\mathbf{u}$, in our derivation instead of the velocities of each component.
By considering the total momentum of the fluid in an arbitrary region it is clear to see that
\[ \mathbf{u} = c_1 \mathbf{u}_1 + c_2 \mathbf{u}_2,\]
where $\mathbf{u}_i$ is the velocity of component with density $\rho_i$.
Throughout we consider an arbitrary material portion $\Sigma(t) \subset \Gamma(t)$
whose  boundary, $\partial \Sigma(t)$, moves with conormal material  velocity of the surface fluid $\ut \cdot \boldsymbol{\nu}_\Sigma$, where $\boldsymbol{\nu}_\Sigma$ denotes the unit conormal vector for $\partial \Sigma(t)$ (the outward unit normal vector which is tangential to $\Sigma(t)$).\\

Since $\rho$ is constant, conservation of mass within the material region $\Sigma(t)$ yields
\[0=\frac{d}{dt} \int_{\Sigma(t)} \rho  = \rho \frac{d}{dt} \int_{\Sigma(t)}  1.\]
Applying the transport theorem we obtain that
\[0= \int_{\Sigma(t)} \gradg \cdot \mathbf{u},\]
and since  $\Sigma$ is  arbitrary this yields
\[\gradg \cdot \mathbf{u} = 0,  \ \mathrm{on} \ \Gamma(t).\]
This shows that the material surface $\Gamma(t)$ has the property of \textit{local  inextensiblity}, that is to say $\frac{d}{dt}|\Sigma(t)| = 0$ for all $\Sigma(t) \subset \Gamma(t)$ such that the boundary moves with conormal material velocity $\ut \cdot \boldsymbol{\nu}_\Sigma$.
As a consequence %one finds that the solution of \eqref{NSCH1}-\eqref{NSCH4} is such that
we have the property that the total area is preserved
\[ |\Gamma(t)| = |\Gamma_0|,\]
for all $t \in [0,T]$.\\

For each component, $u_i$, we consider the mass balance
\[\frac{d}{dt} \int_{\Sigma(t)}  c_i = - \int_{\partial \Sigma(t)} \mathbf{q}_i \cdot \boldsymbol{\nu}_\Sigma,\]
where $\mathbf{q}_i$ is some flux vector to be determined.
The normal component of $\mathbf{q}_i$ does not contribute to the flux, and hence we assume $\mathbf{q}_i$ is purely tangential.
Using the transport theorem, along with the incompressibility above we find that
\[\frac{d}{dt} \int_{\Sigma(t)} c_i = \int_{\Sigma(t)} \matdev c_i.\]
Using integration by parts on the boundary integral, one obtains
\[\int_{\partial \Sigma(t)} \mathbf{q}_i \cdot \boldsymbol{\nu}_\Sigma = \int_{\Sigma(t)} \gradg \cdot \mathbf{q}_i - \int_{\Sigma(t)} \mathbf{q}_i \cdot \boldsymbol{\nu} H = \int_{\Sigma(t)} \gradg \cdot \mathbf{q}_i,\]
where $H$ denotes the mean curvature.
Thus we obtain an equation for $u_i$,
\[\matdev c_i = - \gradg \cdot \mbf{q}_i, \quad \text{on } \Gamma(t).\]
We define quantities $\varphi = c_1 - c_2, \mbf{q} = \mbf{q}_1 - \mbf{q}_2$, and observe that
\begin{align}
	\matdev \varphi = - \gradg \cdot \mbf{q}, \label{transporteqn}
\end{align}
where we choose $\mbf{q}$ later.
Since $c_1 + c_2 = 1$, we can revert back to the individual concentrations by
\[c_1 = \frac{1+\varphi}{2}, \quad c_2 = \frac{1-\varphi}{2} .\]

Lastly, one considers the linear momentum balance for the gross momentum, that is
\[\frac{d}{dt} \int_{\Sigma(t)} \rho \mathbf{u} = \int_{\partial \Sigma(t)} \mathbb{T} \boldsymbol{\nu}_{\Sigma}+ \int_{\Sigma(t)} \mbf{F},\]
where $\mathbb{T}$ is the Cauchy stress tensor describing the stresses across the surface.
Integrating by parts together with $\Sigma(t)$ being arbitrary to 
\begin{align}
	\rho \matdev \mbf{u} = \gradg \cdot \mbb{T}+ \mbf{F}, \label{linmombalance}
\end{align}
where again there is no term involving $\boldsymbol{\nu}$ coming from the integration by parts\footnote{However there is, as we emphasise later, still a component of this equation in the normal direction.} as we assume $\mbb{T}$ maps onto the tangent space of $\Gamma(t)$.
Similarly, one can apply standard arguments to show that the balance of angular momentum yields
\begin{align}
	\mbb{T} = \mbb{T}^T. \label{angmombalance}
\end{align}
Next we consider a local dissipation inequality for the energy.
For a region $\Sigma(t)$, the energy is given by
\[\int_{\Sigma(t)} E(\varphi, \gradg \varphi, \mbf{u}) = \int_{\Sigma(t)} E_1(\mbf{u}) + E_2(\varphi,\gradg \varphi)\]
where
\[E_1(\mbf{z}) = \frac{\rho}{2} |\mbf{z}|^2, \quad E_2(\eta, \mbf{z}) = \frac{\varepsilon}{2}|\mbf{z}|^2 + \frac{1}{\varepsilon} F(\eta).\]
We assume, as in \cite{abels2012thermodynamically,gurtin1996two}, that there is a local dissipation inequality given by
\begin{align} \label{dissineq}
	\frac{d}{dt} \int_{\Sigma(t)} E \leq \int_{\partial \Sigma(t)} \mbb{T} \boldsymbol{\nu}_\Sigma \cdot \mbf{u} + \int_{\partial \Sigma(t)} \matdev \varphi \, \boldsymbol{\xi} \cdot \boldsymbol{\nu}_\Sigma - \int_{\partial \Sigma(t)} \mu \, \mbf{q} \cdot \boldsymbol{\nu}_\Sigma+ \int_{\Gamma(t)} \mbf{F} \cdot \mbf{u},
\end{align}
where $\mu$ is the difference of the chemical potentials, $\mu_i$, of each component, and $\boldsymbol{\xi}$ is a stress (which we assume exists as in \cite{abels2012thermodynamically,gurtin1996two}) characterising the microforces across the boundary of a region - and acts only in the tangential direction.
This inequality is understood as being an appropriate form of the second law of thermodynamics.
The boundary terms correspond to the work done by the macroscopic stresses in the fluid, the work done by the microscopic stresses, and the change of potential energy respectively.\\

As $\Sigma$ is arbitrary, we find that on $\Gamma(t)$ one has
\[\matdev E - \gradg \cdot \left( \mbb{T} \mbf{u} \right) - \gradg \cdot \left( \matdev \varphi \boldsymbol{\xi} \right) + \gradg \cdot \left( \mu \mbf{q} \right)  - \mbf{F} \cdot \mbf{u}\leq 0,\]
where we have used \eqref{angmombalance}.
We use \eqref{linmombalance} and the form of $E$ to see that this implies
\begin{multline*}
	\rho \matdev \mbf{u} \cdot \mbf{u} + \varepsilon \gradg \varphi \cdot \matdev \gradg \varphi + \frac{1}{\varepsilon} F'(\varphi)\matdev \varphi - \rho \matdev \mbf{u} \cdot \mbf{u} - \mbb{T} : \gradg \mbf{u} - \gradg \cdot \left( \matdev \varphi \boldsymbol{\xi} \right) + \gradg \cdot \left( \mu \mbf{q} \right)\\
	- \mbf{F} \cdot \mbf{u}\leq 0.
\end{multline*}

The stress $\boldsymbol{\xi}$ is understood to have a microforce balance on an arbitrary region $\Sigma(t)$, given by
\[\int_{\partial\Sigma(t)} \boldsymbol{\xi} \cdot \boldsymbol{\nu}_\Sigma = \int_{\Sigma(t)} \sigma,\]
from which we obtain
\begin{align}
	\gradg \cdot \boldsymbol{\xi} + \sigma = 0, \text{ on } \Gamma(t), \label{microforcebalance}
\end{align}
where $\sigma$ is a scalar function representing the internal forces in the surface.
Lastly we write $\mbb{S} = \mbb{T} + p \mbb{P}$, where $\mbb{P}$ is the projection tensor, and ${p = -\frac{1}{2} \mathrm{tr}(\mbb{T})}$ is the pressure.
We note that $\mbb{S}$ maps onto tangent vectors, is trace-free, and we find
\[\mbb{P} : \gradg \mbf{u} = \mbb{I}: \gradg \mbf{u} - \boldsymbol{\nu} \otimes \boldsymbol{\nu} : \gradg \mbf{u} = \gradg \cdot \mbf{u} = 0.\]
Next by using \eqref{transporteqn}, \eqref{linmombalance}, \eqref{microforcebalance} in \eqref{dissineq} it is straightforward to see that
\[(\varepsilon \gradg \varphi - \boldsymbol{\xi}) \cdot \matdev \gradg \varphi + \left( \sigma + \frac{1}{\varepsilon} F'(\varphi) - \mu \right)\matdev \varphi  - \left(\mbb{S} + \gradg \varphi \otimes \boldsymbol{\xi} \right): \gradg \mbf{u} + \mbf{q} \cdot \gradg \mu  - \mbf{F} \cdot \mbf{u}\leq 0,\]
where we have also used
\[\gradg (\matdev \varphi) = \mbb{P} \matdev \gradg \varphi + (\gradg \mbf{u})^T \gradg \varphi.\]
As noted in \cite{abels2012thermodynamically, gurtin1996two} the quantity
\[ \underbrace{-(\varepsilon \gradg \varphi - \boldsymbol{\xi}) \cdot \matdev \gradg \varphi - \left( \sigma + \frac{1}{\varepsilon} F'(\varphi) - \mu \right)\matdev \varphi  + \left(\mbb{S} + \gradg \varphi \otimes \boldsymbol{\xi} \right): \gradg \mbf{u} - \mbf{q} \cdot \gradg \mu + \mbf{F} \cdot \mbf{u}}_{=: \mathcal{D}},\]
represents the dissipation, and hence the assumed inequality is equivalent to $\mathcal{D} \geq 0$.\\

One then argues as in \cite{abels2012thermodynamically,gurtin1996two} to show that if one allows $\mbb{S}, \mbf{q}, \boldsymbol{\xi}, \sigma$ to depend on $\varphi, \gradg \varphi$, $\mu, \gradg \mu$, $\mbb{E}(\mbf{u})$ arbitrarily then the assumed dissipation inequality can fail to hold.
The argument requires one to assume the presence of general forces and external mass supplies.
In particular one finds that necessarily
\begin{align} \label{microforces}
	\varepsilon \gradg \varphi - \boldsymbol{\xi} = 0, \quad \sigma + \frac{1}{\varepsilon} F'(\varphi) - \mu = 0.
\end{align}
We then assume, as in \cite{caetano2021cahn}, that the mass flux takes the form
\[\mbf{q} = -M(\varphi) \gradg \mu,\]
for some mobility function $M(\cdot)$.
Similarly, motivated by Newton's rheological law we assume
\[\mbb{S} + \varepsilon \gradg \varphi \otimes \gradg \varphi  = 2 \eta(\varphi) \mbb{E}(\mbf{u}),\]
where
\[\mbb{E}(\mbf{u}) = \frac{1}{2} \left(\gradg \mbf{u} +  (\gradg \mbf{u})^T \right)\]
is the rate of strain tensor, and $\eta(\cdot)$ is a variable viscosity (depending on the concentration).
As noted in \cite{abels2012thermodynamically}, $\mbb{S} + \varepsilon \gradg \varphi \otimes \gradg \varphi $ represents the change in energy due to friction in the fluids, and is referred to as the viscous strain tensor.
In summary, this allows us to observe that
\begin{align}
	\mathbb{T} = \underbrace{-p\mathbb{P} + 2 \eta(\varphi) \mathbb{E}(\mathbf{u})}_{\text{Boussinesq-Scriven term}} - \underbrace{\varepsilon \gradg \varphi \otimes \gradg \varphi}_{\text{Korteweg term}}, \label{stresstensor}
\end{align}
where the first term is like the Boussinesq-Scriven ansatz seen in \cite{brandner2022derivations}, but we allow variable viscosity.
Then by combining \eqref{microforcebalance}, \eqref{microforces} we find that
\[\mu = -\varepsilon \lapg \varphi + \frac{1}{\varepsilon} F'(\varphi) ,\]
and similarly combining \eqref{transporteqn}, and the assumption on $\mbf{q}$ we obtain
\[\matdev \varphi = \gradg \cdot (M(\varphi) \gradg \mu),\]
which are the relations for $\varphi,\mu$ as in \cite{caetano2021cahn}.
Lastly we observe that
\[\gradg \cdot (p \mbb{P}) = \gradg p - pH\boldsymbol{\nu},\]
and hence by using \eqref{linmombalance}, \eqref{stresstensor} we see
\[\rho \matdev \mbf{u} = -\gradg p + pH \boldsymbol{\nu} + \gradg \cdot (2 \eta(\varphi) \mbb{E}(\mbf{u}) ) - \varepsilon \gradg \cdot (\gradg \varphi \otimes \gradg \varphi) +\mathbf{F}.\]
Lastly by recalling that $\gradg \cdot \mbf{u}= 0$, we observe that we have derived \eqref{NSCH1} - \eqref{NSCH4}.\\

To obtain the tangential Navier-Stokes-Cahn-Hilliard system we assume that the geometric motion of $\Gamma(t)$ is defined by the normal velocity field
$V_N\boldsymbol{\nu}$ so material continuity in the normal direction implies the equation
\[\mathbf{u}\cdot\boldsymbol{\nu} =V_N.\]
To find the unknown $\ut$ one considers the projection of this momentum equation.
For this we note that (as in \cite{OlsReuZhi22})
\begin{gather*}
	\matdev \varphi = \normdev \varphi + \gradg \varphi \cdot \ut,\\
	\mathbb{P}\matdev \mathbf{u} = \mathbb{P} \normdev \mathbf{u}_T + (\gradg \mathbf{u}_T)\mbf{u}_T + V_N \mathbb{H} \mathbf{u}_T - \frac{1}{2} \gradg V_N^2.
\end{gather*}
Similarly we compute
\begin{align*}
	\gradg \cdot (\gradg \varphi \otimes \gradg \varphi))  = \lapg \varphi \gradg \varphi + \frac{1}{2} \gradg |\gradg \varphi|^2 - \left( \gradg \varphi \cdot \mbb{H} \gradg \varphi \right) \boldsymbol{\nu},
\end{align*}
and by using \eqref{NSCH4} we see
\[-\varepsilon \lapg \varphi \gradg \varphi = \left(\mu - \frac{1}{\varepsilon} F'(\varphi)\right)\gradg \varphi.\]
Thus defining the modified pressure to be
\[\tilde{p} = p + \frac{\varepsilon}{2}|\gradg \varphi|^2 + \frac{1}{\varepsilon} F(\varphi),\]
and recalling that $\gradg \cdot \mathbf{u}_T = - HV_N$ we obtain the tangential Navier-Stokes-Cahn-Hilliard equations \eqref{TNSCH1} - \eqref{TNSCH4}.
The obvious modification to the calculations in \cite{OlsReuZhi22} also yields an equation for the normal component
\begin{multline}
	\rho \matdev V_N = -2 \eta(\varphi)\left( \text{tr}(\mbb{H} \gradg \ut ) + \rho V_N \text{tr}(\mbb{H}^2) \right) + \rho \ut \cdot \mbb{H} \ut - \ut \cdot \gradg V_N + pH\\
	- \varepsilon \gradg \varphi \cdot \mbb{H} \gradg \varphi+F_\nu, \label{normalcomponent}
\end{multline}
where $F_\nu=\mathbf{F}\cdot \nu$, which  must be satisfied and  coupled with \eqref{TNSCH1}-\eqref{TNSCH4} .
\begin{remark}
    \begin{enumerate}
        \item \eqref{NSCH1}-\eqref{NSCH4} is a simplified form of the system derived in \cite{bachini2023derivation}, where the authors also consider the effects of bending/friction.
    It is useful to note that these authors consider the modified pressure throughout.
    Hence neglecting the effects of bending/friction terms and changing notation suitably one finds the two systems are identical.
    \item One may also be able to derive a related model by considerations similar to \cite{zimmermann2019isogeometric}.
    We leave this for future work.
    \end{enumerate}
\end{remark}

\subsection{Derivation by a thin film limit}
In this subsection we consider the thin film limit of relevant Navier-Stokes-Cahn-Hilliard equations on an evolving Cartesian domain.
This approach has been considered for the heat equation \cite{miura2017zero}, the Navier-Stokes equations \cite{brandner2022derivations,miura2018singular}, and the Ginzburg-Landau equation \cite{miura2024thin}.
Besides use in derivation of a suitable system of surface Navier-Stokes-Cahn-Hilliard equations there has been interest in using a thin film approximation numerically \cite{yang2020phase} to study the limiting surface equations.\\

We assume throughout that $\Gamma(t)$ does not undergo a change in topology.
Indeed, in the presence of a change in topology the modelling of this phenomenon is different and so one expects the systems \eqref{NSCH1}-\eqref{NSCH4} and \eqref{TNSCH1}-\eqref{TNSCH4} won't necessarily make sense.
We discuss this more in Remark \ref{topology change}.
We also ignore the effect of the external force, $\mbf{F}$, for brevity.\\

As before we still consider a closed oriented evolving surface, $\Gamma(t)$, with a prescribed normal velocity $V_N$.
We define $\Omega_\gamma(t)$ by
\[ \Omega_\gamma(t) := \{ x \in \mbb{R}^3 \mid |d(x,t)| < \gamma \}, \]
where $d(x,t)$ is the signed distance function of $\Gamma(t)$, and $\gamma > 0$ is sufficiently small.
We then define the (noncylindrical) space-time domain
\[ Q_{\gamma,T} := \bigcup_{t \in [0,T]} \Omega_{\gamma}(t) \times \{t\},\]
where we pose our problem.
We consider a Navier-Stokes-Cahn-Hilliard system on $Q_{\gamma,T}$,
\begin{gather}
	\rho\left(\frac{\partial \mbf{u}^\gamma}{\partial t} + (\mbf{u}^\gamma \cdot \nabla) \mbf{u}^\gamma\right) = -\nabla p^\gamma + \nabla \cdot (2\eta(\varphi^\gamma) \mbb{E}_{\Omega} (\mbf{u}^\gamma)) - \varepsilon \nabla \cdot (\nabla \varphi^\gamma \otimes \nabla \varphi^\gamma),\label{bulkeqn1}\\
	\nabla \cdot \mbf{u}^\gamma = 0,\label{bulkeqn2}\\
	\ddt{\varphi^\gamma} - \nabla \cdot \left(M(\varphi^\gamma) \nabla \mu^\gamma\right) + \mbf{u}^\gamma \cdot \nabla \varphi^\gamma = 0,\label{bulkeqn3}\\
	\mu^\gamma = -\varepsilon \Delta \varphi^\gamma + \frac{1}{\varepsilon}F'(\varphi^\gamma),\label{bulkeqn4}
\end{gather}
equipped with boundary conditions
\begin{gather}
	\mbf{u}^\gamma \cdot \boldsymbol{\nu}^\gamma = V_N^\gamma, \label{bulkBC1}\\
	\left[ \mbb{E}_{\Omega}(\mbf{u}^\gamma) \boldsymbol{\nu}^\gamma  \right]_{\text{tan}} = 0, \label{bulkBC2}\\
	\nabla \varphi^\gamma \cdot \boldsymbol{\nu}^\gamma = 0,\label{bulkBC3}\\
	\nabla \mu^\gamma \cdot \boldsymbol{\nu}^\gamma  = 0,\label{bulkBC4}
\end{gather}
on the lateral boundary $\partial_\ell Q_{\gamma,T}$ defined as
\[ \partial_\ell Q_{\gamma,T} := \bigcup_{t \in [0,T]} \partial \Omega_{\gamma}(t) \times \{t\} . \]
Here the normal velocity of the bulk domain, $\Omega_\gamma(t)$, is $V_N^\gamma(x,t) := V_N(\pi(x,t),t)$.
We are using the notation $\mbb{E}_\Omega := \frac{1}{2}(\nabla \mbf{u}^\gamma + \left(\nabla \mbf{u}^\gamma)^T\right)$ for the rate of strain tensor in $\Omega_\gamma$, and $[\cdot]_{\text{tan}}$ denoting the tangential component to $\partial \Omega_\gamma(t)$ of a vector in $\mbb{R}^3$.
The condition \eqref{bulkBC2} is sometimes referred to as the perfect slip condition, and appears as a natural boundary condition.
One may expect different boundary conditions for $\varphi^\gamma, \mu^\gamma$, similar to the Robin-type condition seen for the heat equation in \cite{miura2017zero}.
However the usual Neumann conditions are still sufficient for our setting, and retain the mass conservation property.
To see this we use the Reynolds transport theorem so that
\begin{align*}
	\frac{d}{dt} \int_{\Omega_{\gamma}(t)} \varphi^\gamma  &= \int_{\Omega_{\gamma}(t)} \ddt{\varphi^\gamma} + \int_{\partial \Omega_{\gamma}(t)}\varphi^\gamma V_N^\gamma\\
	&=  \int_{\Omega_\gamma(t)} \left(\Delta \mu^\gamma - \mbf{u}^\gamma \cdot \nabla \varphi^\gamma \right) + \int_{\partial \Omega_{\gamma}(t)}\varphi^\gamma V_N^\gamma\\
	& = \int_{\partial \Omega_{\gamma}(t)}\left(	\nabla \mu^\gamma \cdot \boldsymbol{\nu}^\gamma - \varphi^\gamma \mbf{u}^\gamma \cdot \boldsymbol{\nu}^\gamma + \varphi^\gamma V_N^\gamma \right)\\
	&= \int_{\partial \Omega_{\gamma}(t)} \nabla \mu^\gamma \cdot \boldsymbol{\nu}^\gamma ,
\end{align*}
where we have used \eqref{bulkBC1} for the final equality.
The Neumann condition for $\varphi^\gamma$ follows in the usual way, without any need for Reynolds transport theorem.
We now expand $\varphi^\gamma, \mu^\gamma, \mbf{u}^\gamma, p^\gamma$ in terms of the signed distance function as:
\begin{gather*}
	\varphi^\gamma(x,t) = \varphi^0(\pi(x,t),t) + d(x,t)\varphi^1(\pi(x,t),t) + d(x,t)^2 \varphi^2(\pi(x,t),t) + \mathcal{O}(d(x,t)^3),\\
	\mu^\gamma(x,t) = \mu^0(\pi(x,t),t) + d(x,t)\mu^1(\pi(x,t),t) + d(x,t)^2 \mu^2(\pi(x,t),t) + \mathcal{O}(d(x,t)^3),\\
	\mbf{u}^\gamma(x,t) = \mbf{u}^0(\pi(x,t),t) + d(x,t)\mbf{u}^1(\pi(x,t),t) + d(x,t)^2 \mbf{u}^2(\pi(x,t),t) + \mathcal{O}(d(x,t)^3),\\
	p^\gamma(x,t) = p^0(\pi(x,t),t) + d(x,t)p^1(\pi(x,t),t) + d(x,t)^2 p^2(\pi(x,t),t) + \mathcal{O}(d(x,t)^3),
\end{gather*}
where here $\pi(x,t)$ is the closest point projection of $x$ onto $\Gamma(t)$.
This is uniquely defined on a small tubular neighbourhood, $\mathcal{N}(\Gamma(t))$, of $\Gamma(t)$.
Hence we have a requirement on $\gamma$ being sufficiently small so that
\[ Q_{\gamma,T} \subseteq \bigcup_{t \in [0,T]} \mathcal{N}(\Gamma(t)) \times \{t\} .\]
Before considering the thin film limit we recall some preliminary results.
Firstly, we recall that
\[ \nabla d(x,t) = \boldsymbol{\nu}(\pi(x,t),t), \qquad \ddt{d(x,t)} = -V_N(\pi(x,t),t),\]
and from these one can show the following results.
\begin{lemma}[\cite{miura2018singular}, Lemma 2.7]
	\label{miura differentiation}
	Let $f$ be a scalar of vector valued function on $\mathcal{G}_T$.
	Then the spatial/temporal derivatives of the composite function $f(\pi(x,t),t)$ are such that
	\begin{gather*}
		\nabla (f(\pi(x,t),t)) = \gradg f(\pi(x,t),t) + d(x,t)\left(\mbb{H}\gradg f\right)(\pi(x,t),t) + \mathcal{O}(d(x,t)^2),\\
		\ddt{f(\pi(x,t),t)} = \normdev f(\pi(x,t), t) + d(x,t)((\gradg V_N \cdot \gradg)f)(\pi(x,t),t) + \mathcal{O}(d(x,t)^2),
	\end{gather*}
	for $(x,t) \in Q_{\gamma,T}$.
\end{lemma}
\begin{lemma}[\cite{miura2018singular}, Lemma 2.8]
	\label{miura's lemma}
	Let $\mbb{S}^0, \mbb{S}^1$ be $3 \times 3$ matrix valued functions on $\Gamma(t)$ for each $t \in (0,T)$.
	Then for $x \in \Omega_{\gamma}(t)$ set
	\[\mbb{S}(x) = \mbb{S}^0(\pi(x,t)) + d(x,t) \mbb{S}^1(\pi(x,t)) + \mathcal{O}(d(x,t)^2).\]
	Then we have
	\[ \nabla \cdot \mbb{S}(x) = \divg \mbb{S}^0(\pi(x,t),t) + (\mbb{S}^1(\pi(x,t),t))^T \boldsymbol{\nu}(\pi(x,t),t) + \mathcal{O}(d(x,t)), \]
	for $x \in \Omega_\gamma(t)$.
\end{lemma}
\begin{theorem}
	Let $(\varphi^\gamma, \mu^\gamma, \mbf{u}^\gamma, p^\gamma)$ solve \eqref{bulkeqn1}-\eqref{bulkeqn4} with boundary conditions \eqref{bulkBC1}-\eqref{bulkBC4}.
	Then $(\varphi^0, \mu^0, \mbf{u}^0, p^0, p^1)$ from the corresponding expansion in terms of the signed distance functions solve
	\begin{gather}
		\rho \matdev \mbf{u}^0 = -\gradg p^0 + p^1 \boldsymbol{\nu} + \gradg \cdot (2 \eta(\varphi^0) \mbb{E}(\mbf{u}^0) ) - \varepsilon \gradg \cdot (\gradg \varphi^0 \otimes \gradg \varphi^0) \label{thinfilmlimit1},\\
		\gradg \cdot \mathbf{u}^0 = 0 \label{thinfilmlimit2},\\
		\matdev \varphi^0 = \gradg \cdot (M(\varphi^0) \gradg \mu^0),\label{thinfilmlimit3}\\
		\mu^0 = -\varepsilon\lapg \varphi^0 + \frac{1}{\varepsilon}F'(\varphi^0),\label{thinfilmlimit4}
	\end{gather}
	on $\mathcal{G}_T$, such that $\mbf{u}^0 \cdot \boldsymbol{\nu} = V_N$.
\end{theorem}
\begin{proof}
	We abbreviate $\pi(x,t)$ to $\pi$, and $d(x,t)$ to $d$ throughout this proof.
	Firstly by considering \eqref{bulkBC1} and the expansion for $\mbf{u}^\gamma$ (on the boundary $d = \pm \gamma$) we have\footnote{Here the $\pm \gamma$ corresponds to the boundary of $\Omega_\gamma(t)$ consisting of two disjoint sets, $\{ x \in \mbb{R}^3 \mid d(x,t) = \gamma \},$ and $\{ x \in \mbb{R}^3 \mid d(x,t) = -\gamma \}$.}
	\[\mbf{u}^0(\pi,t)\cdot \boldsymbol{\nu}(\pi,t) \pm \gamma \mbf{u}^1(\pi,t)\cdot \boldsymbol{\nu}(\pi,t) + \gamma^2 \mbf{u}^2(\pi,t)\cdot \boldsymbol{\nu}(\pi,t) +\mathcal{O}(\gamma^3) = V_N(\pi, t),\]
	and so equating terms of order $\gamma^k$ for $k =0,1,2$, one finds that
	\begin{gather*}
		\mbf{u}^0(\pi,t)\cdot \boldsymbol{\nu}(\pi,t) = V_N(\pi, t),\\
		\mbf{u}^1(\pi,t)\cdot \boldsymbol{\nu}(\pi,t) =0,\\
		\mbf{u}^2(\pi,t)\cdot \boldsymbol{\nu}(\pi,t)=0.
	\end{gather*}
	Taking the gradient of $\mbf{u}^\gamma$, we find
	\begin{multline}
		\nabla \mbf{u}^\gamma(x,t) = \gradg \mbf{u}^0(\pi, t) + \boldsymbol{\nu}(\pi, t) \otimes \mbf{u}^1(\pi, t)\\
  + d \left( (\mbb{H}\gradg \mbf{u}^0)(\pi, t) + \gradg \mbf{u}^1(\pi, t) + 2(\boldsymbol{\nu} \otimes \mbf{u}^2)(\pi, t) \right)+ \mathcal{O}(d^2), \label{thinfilmpf1}
	\end{multline}
	and hence taking the trace of the above, using \eqref{bulkeqn2} and $\mbf{u}^1 \cdot \boldsymbol{\nu} = 0$, we obtain
	\[ \nabla \cdot \mbf{u}^\gamma = \divg \mbf{u}^0 + \mathcal{O}(d), \]
	from which the zeroth order terms yield \eqref{thinfilmlimit2}.
	Similar calculations let us verify that
	\begin{multline*}
		\nabla \varphi^\gamma(x,t) = \gradg \varphi^0(\pi, t) + \varphi^1(\pi,t) \boldsymbol{\nu}(\pi, t)\\
		+ d \left( (\mbb{H}\gradg \varphi^0)(\pi, t) + \gradg \varphi^1(\pi, t)+ 2(\varphi^2\boldsymbol{\nu})(\pi, t) \right) + \mathcal{O}(d^2),
	\end{multline*}
	and
	\begin{align}
		\nabla p^\gamma(x,t) = \gradg p^0(\pi, t) + p^1(\pi,t) \boldsymbol{\nu}(\pi, t) + \mathcal{O}(d). \label{thinfilmpf2}
	\end{align}
	Similarly, by considering the transpose of \eqref{thinfilmpf1} one finds
	\begin{align}
		\mbb{E}_{\Omega}(\mbf{u}^\gamma)(x,t) = \mbb{S}^0(\pi,t) + d \mbb{S}^1(\pi, t) + \mathcal{O}(d^2), \label{thinfilmpf3}
	\end{align}
	where
	\begin{gather*}
		\mbb{S}^0= \mbb{E}(\mbf{u}^0) + \frac{\boldsymbol{\nu}\otimes \mbf{u}^1 +  \mbf{u}^1 \otimes \boldsymbol{\nu}}{2},\\
		\mbb{S}^1 = \frac{\mbb{H} \gradg \mbf{u}^0 + (\mbb{H} \gradg \mbf{u}^0)^T}{2} + \mbb{E}(\mbf{u}^1) + \boldsymbol{\nu} \otimes \mbf{u}^2 + \mbf{u}^2 \otimes \boldsymbol{\nu}.
	\end{gather*}
	Now, by our smoothness assumption on $\eta(\cdot)$ we may use Taylor's theorem to write
	\[ \eta(\varphi^\gamma(x,t)) = \eta(\varphi^0(\pi,t)) + d\eta'(\varphi^*(\pi,t))\varphi^1(\pi, t) + \mathcal{O}(d^2), \]
	where $\varphi^*(\pi, t)$ is some function valued between $\varphi^0(\pi, t)$ and $\varphi^\gamma(x,t)$ which arises from the remainder term in Taylor's theorem.
	Using this, we find that
	\[ \eta(\varphi^\gamma(x,t))\mbb{E}_{\Omega}(\mbf{u}^\gamma)(x,t) = \widetilde{\mbb{S}}^0(\pi,t) + d \widetilde{S}^1(\pi, t) + \mathcal{O}(d^2), \]
	where
	\begin{align*}
		\widetilde{\mbb{S}}^0= \eta(\varphi^0)\left(\mbb{E}(\mbf{u}^0) + \frac{\boldsymbol{\nu}\otimes \mbf{u}^1 +  \mbf{u}^1 \otimes \boldsymbol{\nu}}{2}\right),
	\end{align*}
	and
	\begin{multline*}
		\widetilde{\mbb{S}}^1 = \eta'(\varphi^*)\left(\mbb{E}(\mbf{u}^0) + \frac{\boldsymbol{\nu}\otimes \mbf{u}^1 +  \mbf{u}^1 \otimes \boldsymbol{\nu}}{2}\right)\\
		+ \eta(\varphi^0) \left( \frac{\mbb{H} \gradg \mbf{u}^0 + (\mbb{H} \gradg \mbf{u}^0)^T}{2} + \mbb{E}(\mbf{u}^1) + \boldsymbol{\nu} \otimes \mbf{u}^2 + \mbf{u}^2 \otimes \boldsymbol{\nu}\right).
	\end{multline*}
	Hence using Lemma \ref{miura's lemma} we see
	\[ \nabla \cdot (2\eta(\varphi^\gamma) \mbb{E}_{\Omega}(\mbf{u}^\gamma)) = 2\divg \widetilde{\mbb{S}}^0 + 2(\widetilde{\mbb{S}}^1)^T \boldsymbol{\nu} + \mathcal{O}(d),\]
	and so we check which of these terms vanish.
	To do this we firstly note that by rewriting \eqref{bulkBC2} one has
	\[ \mbb{P}(\pi,t ) \mbb{E}_{\Omega}(\mbf{u}^\gamma) \boldsymbol{\nu}(\pi, t) = 0, \quad x \in \partial \Omega_{\gamma}(t), \]
	and so by using \eqref{thinfilmpf3} one finds
	\[ \mbb{P}(\pi,t ) {\mbb{S}}^0(\pi, t) \boldsymbol{\nu}(\pi, t) \pm \gamma \mbb{P}(\pi,t ) {\mbb{S}}^1(\pi, t) \boldsymbol{\nu}(\pi, t) + \mathcal{O}(\gamma^2)  = 0, \]
	and hence
	\begin{gather*}
		\mbb{P}(\pi,t ){\mbb{S}}^0(\pi, t) \boldsymbol{\nu}(\pi, t) =0,\\
		\mbb{P}(\pi,t ) {\mbb{S}}^1(\pi, t) \boldsymbol{\nu}(\pi, t) = 0.
	\end{gather*}
	Then by using the form of $\mbb{S}^0$, and
	\[ (\boldsymbol{\nu} \otimes \mbf{u}^1)\boldsymbol{\nu} = 0, \quad (\mbf{u}^1 \otimes \boldsymbol{\nu}) \boldsymbol{\nu} = \mbf{u}^1, \quad \mbb{P}\mbf{u}^1 = \mbf{u}^1, \]
	in the above yields
	\[ \mbf{u}^1 = -2 \mbb{P}(\pi,t) \mbb{E}(\mbf{u}^0)\boldsymbol{\nu}(\pi,t). \]
	Thus we find
	\begin{align*}
		\widetilde{\mbb{S}}^0 &= \eta(\varphi^0)\mbb{E}(\mbf{u}^0) - \eta(\varphi^0)(\boldsymbol{\nu} \otimes \boldsymbol{\nu})\mbb{E}(\mbf{u}^0)\mbb{P} - \eta(\varphi^0) \mbb{P}\mbb{E}(\mbf{u}^0)(\boldsymbol{\nu} \otimes \boldsymbol{\nu})\\
		& = \eta(\varphi^0)\mbb{E}(\mbf{u}^0) - \eta(\varphi^0)(\boldsymbol{\nu} \otimes \boldsymbol{\nu})\mbb{E}(\mbf{u}^0)(\boldsymbol{\nu} \otimes \boldsymbol{\nu})\\
		& = \eta(\varphi^0)\mbb{E}(\mbf{u}^0),
	\end{align*}
	where we have used $\mbb{P} = \mbb{I} - (\boldsymbol{\nu} \otimes \boldsymbol{\nu})$ and $\gradg \mbf{u}^0 = \mbb{P}\gradg \mbf{u}^0\mbb{P}$.
	Hence we find
	\[ \divg \widetilde{\mbb{S}}^0 = \divg(\eta(\varphi^0) \mbb{E}(\mbf{u}^0)).\]
	We now show that $(\widetilde{\mbb{S}}^1)^T \boldsymbol{\nu}=0$.
	Firstly, we notice that from the above calculations it is clear that
	\[(\widetilde{\mbb{S}}^1)^T \boldsymbol{\nu} = \eta'(\varphi^*)\mbb{E}(\mbf{u}^0)\boldsymbol{\nu} + \eta(\varphi^0) \left( \frac{\mbb{H} \gradg \mbf{u}^0\boldsymbol{\nu} + (\mbb{H} \gradg \mbf{u}^0)^T\boldsymbol{\nu}}{2} + \mbb{E}(\mbf{u}^1)\boldsymbol{\nu} + (\boldsymbol{\nu} \otimes \mbf{u}^2)\boldsymbol{\nu} + (\mbf{u}^2 \otimes \boldsymbol{\nu})\boldsymbol{\nu}\right). \]
	Recalling that
	\[ \mbb{E}(\mbf{u}^0) \boldsymbol{\nu} = 0 = \mbb{E}(\mbf{u}^1) \boldsymbol{\nu}, \quad \mbb{H} \gradg \mbf{u}^0\boldsymbol{\nu} = 0 = (\mbb{H} \gradg \mbf{u}^0)^T\boldsymbol{\nu}, \quad (\boldsymbol{\nu} \otimes \mbf{u}^2)\boldsymbol{\nu}=0, \quad (\mbf{u}^2 \otimes \boldsymbol{\nu})\boldsymbol{\nu} = \mbf{u}^2, \]
	where we have used $\boldsymbol{\nu} \in \ker(\mbb{H})$, and $\mbf{u}^2 \cdot \boldsymbol{\nu}=0,$
	we find
	\[ (\widetilde{\mbb{S}}^1)^T \boldsymbol{\nu} = \eta(\varphi^0)\mbf{u}^2. \]
	Then using the form of $\mbb{S}^1$ in $\mbb{P}(\pi,t ) {\mbb{S}}^1(\pi, t) \boldsymbol{\nu}(\pi, t) = 0$ one finds that
	$\mbf{u}^2 = 0$, and so $(\widetilde{\mbb{S}}^1)^T \boldsymbol{\nu} = 0$.
	Hence
	\begin{align}
		\nabla \cdot (2\eta(\varphi^\gamma) \mbb{E}_{\Omega}(\mbf{u}^\gamma)) = \divg (2\eta(\varphi^0) \mbb{E}(\mbf{u}^0)) + \mathcal{O}(d). \label{thinfilmpf4}
	\end{align}
	
	The tensor product involving $\nabla \varphi^\gamma$ is dealt with similarly, where it is straightforward to see that
	\[ \nabla \varphi^\gamma \otimes \nabla \varphi^\gamma = \gradg \varphi^0 \otimes \gradg \varphi^0 + d (\gradg \varphi^1 \otimes \gradg \varphi^0 + \gradg \varphi^0 \otimes \gradg \varphi^1) + \mathcal{O}(d^2), \]
	and hence by using Lemma \ref{miura's lemma}
	\[ \nabla \cdot (\nabla \varphi^\gamma \otimes \nabla \varphi^\gamma) = \gradg \cdot (\gradg \varphi^0 \otimes \gradg \varphi^0) + (\gradg \varphi^0 \otimes \gradg \varphi^1 + \gradg \varphi^1 \otimes \gradg \varphi^0) \boldsymbol{\nu} + \mathcal{O}(d), \]
	where the latter term clearly vanishes as $\gradg \varphi^0, \gradg \varphi^1$ are tangential.
	Thus one obtains
	\begin{align}
		\varepsilon \nabla \cdot (\nabla \varphi^\gamma \otimes \nabla \varphi^\gamma) = \varepsilon \gradg \cdot (\gradg \varphi^0 \otimes \gradg \varphi^0) +\mathcal{O}(d). \label{thinfilmpf5}
	\end{align}
	For the momentum equation, all that remains to discuss is the time derivative and the advective term.
	The time derivative follows immediately from Lemma \ref{miura differentiation} as
	\begin{align}
		\ddt{\mbf{u}^\gamma}(x,t) = \normdev \mbf{u}^0(\pi,t) - V_N(\pi,t) \mbf{u}^1(\pi,t) + \mathcal{O}(d), \label{thinfilmpf6}
	\end{align}
	which we want to turn into a material time derivative by considering the advective term.
	For the advective term we write $(\mbf{u}^\gamma \cdot \nabla) \mbf{u}^\gamma = (\mbf{u}^\gamma)^T \nabla \mbf{u}^\gamma$.
	Hence from the expansion of $\mbf{u}^\gamma$ and using Lemma \ref{miura differentiation} one finds
	\[ (\mbf{u}^\gamma)^T \nabla \mbf{u}^\gamma = (\mbf{u}^0)^T \gradg \mbf{u}^0 + \mbf{u}^1 (\mbf{u}^0 \cdot \boldsymbol{\nu}) + \mathcal{O}(d) = (\mbf{u}^0 \cdot \gradg ) \mbf{u}^0 + V_N\mbf{u}^1 + \mathcal{O}(d). \]
	Hence using this expression for the advection with \eqref{thinfilmpf6} one finds
	\begin{align}
		\rho\left(\ddt{\mbf{u}^\gamma} + (\mbf{u}^\gamma \cdot \nabla) \mbf{u}^\gamma \right) = \rho \left( \normdev \mbf{u}^0 + (\mbf{u}^0 \cdot \gradg ) \mbf{u}^0 \right)+ \mathcal{O}(d) = \rho\matdev \mbf{u}^0 + \mathcal{O}(d), \label{thinfilmpf7}
	\end{align}
	where here we understand $\matdev$ to mean the derivative along the velocity field given by $V_N \boldsymbol{\nu} + \mbf{u}^0$ - that is the tangential velocity is only considered up to the zeroth order term.
	This point is made clearer by the notation of \cite{miura2018singular} where one would write this as $\partial^{\bullet}_{\mbf{u}^0}$.
	Now by combining \eqref{thinfilmpf2}, \eqref{thinfilmpf4}, \eqref{thinfilmpf5}, \eqref{thinfilmpf7} in \eqref{bulkeqn1} one obtains
	\[ \rho \matdev \mbf{u}^0 = -\gradg p^0 + p^1 \boldsymbol{\nu} + \gradg \cdot (2 \eta(\varphi^0) \mbb{E}(\mbf{u}^0) ) - \varepsilon \gradg \cdot (\gradg \varphi^0 \otimes \gradg \varphi^0) + \mathcal{O}(d), \]
	and as the functions $(\varphi^0, \mu^0, \mbf{u}^0, p^0, p^1)$ are independent of $d$ one obtains \eqref{thinfilmlimit1} from the zeroth order terms.\\
	
	It remains to show that \eqref{bulkeqn3},\eqref{bulkeqn4} give \eqref{thinfilmlimit3}, \eqref{thinfilmlimit4} at zeroth order.
	This is largely the same, so we skim the details.
	The advective term in \eqref{bulkeqn3} is the main point of interest here.
	Using the expansions for $\varphi^\gamma, \mbf{u}^\gamma$ and Lemma \ref{miura differentiation} one finds
	\begin{align}
		\mbf{u}^\gamma \cdot \nabla \varphi^\gamma = \mbf{u}^0 \cdot \gradg \varphi^0 + \mbf{u}^0 \cdot \boldsymbol{\nu} \varphi^1 + \mathcal{O}(d). \label{thinfilmpf8}
	\end{align}
	For the time derivative one uses Lemma \ref{miura differentiation} as before so that
	\[ \ddt{\varphi^\gamma}(x,t) = \normdev \varphi^0(\pi,t) - V_N(\pi,t) \varphi^1(\pi,t) + \mathcal{O}(d). \]
	We combine this with \eqref{thinfilmpf8}, recalling that $\mbf{u}^0 \cdot \boldsymbol{\nu} = V_N$, so that
	\begin{align}
		\ddt{\varphi^\gamma} + \mbf{u}^\gamma \cdot \nabla \varphi^\gamma = \normdev \varphi^0 + \mbf{u}^0 \cdot \gradg \varphi^0 + \mathcal{O}(d) = \matdev \varphi^0 + \mathcal{O}(d). \label{thinfilmpf9}
	\end{align}
	It remains to consider the term $\nabla \cdot (M(\varphi^\gamma) \nabla \mu)$, which is dealt with almost identically to the term $\nabla \cdot (2 \eta(\varphi^\gamma) \mbb{E}_{\Omega}(\mbf{u}^\gamma))$.
	From Lemma \ref{miura differentiation} we see
	\[ \nabla \mu^\gamma(x,t) = \gradg \mu^0(\pi,t) + \mu^1(\pi,t) \boldsymbol{\nu}(\pi,t) + d\left( \gradg \mu^1(\pi,t) + 2 \mu^2(\pi,t) \boldsymbol{\nu}(\pi,t) \right) + \mathcal{O}(d^2), \]
	and by Taylor's theorem (assuming $M(\cdot)$ is sufficiently smooth)
	\[M(\varphi^\gamma(x,t)) = M(\varphi^0(\pi,t)) + dM'(\varphi^*(\pi,t)) \varphi^1(\pi, t) + \mathcal{O}(d^2),\]
	where we have abused notation and reused the $\varphi^*$ for the intermediate point arising in the remainder - which is different from the $\varphi^*$ before, but this does not matter.
	Combining these we find
	\begin{align*}
		M(\varphi^\gamma) \nabla \mu^\gamma &= M(\varphi^0) \gradg \mu^0 + M(\varphi^0) \mu^1 \boldsymbol{\nu}\\
		&+ d \left( M'(\varphi^*)\varphi^1\gradg \mu^0 + M'(\varphi^*)\varphi^1 \mu^1 \boldsymbol{\nu} + M(\varphi^0) \gradg \mu^1 + 2 M(\varphi^0)\mu^2 \boldsymbol{\nu} \right)\\
		&+ \mathcal{O}(d^2).
	\end{align*}
	Considering the gradient of this expression one finds
	\begin{multline*}
		\nabla(M(\varphi^\gamma) \nabla \mu^\gamma) = \gradg (M(\varphi^0) \gradg \mu^0) + \gradg(M(\varphi^0)\mu^1) \otimes \boldsymbol{\nu} + M(\varphi^0)\mu^1 \mbb{H} + M'(\varphi^*)\varphi^1\boldsymbol{\nu} \otimes \gradg \mu^0\\
		+ M'(\varphi^*)\varphi^1 \mu^1 \boldsymbol{\nu} \otimes\boldsymbol{\nu} + M(\varphi^0) \boldsymbol{\nu} \otimes\gradg \mu^1 + 2 M(\varphi^0)\mu^2 \boldsymbol{\nu} \otimes\boldsymbol{\nu} + \mathcal{O}(d), 
	\end{multline*}
	and taking the trace of the above yields
	\begin{multline*}
		\nabla \cdot (M(\varphi^\gamma) \nabla \mu^\gamma) = \divg (M(\varphi^0) \gradg \mu^0) + \gradg(M(\varphi^0)\mu^1) \cdot \boldsymbol{\nu} + M(\varphi^0)\mu^1 H + M'(\varphi^*)\varphi^1\boldsymbol{\nu} \cdot \gradg \mu^0\\
		+ M'(\varphi^*)\varphi^1 \mu^1 \boldsymbol{\nu} \cdot \boldsymbol{\nu} + M(\varphi^0) \boldsymbol{\nu} \cdot \gradg \mu^1 + 2 M(\varphi^0)\mu^2 \boldsymbol{\nu} \cdot \boldsymbol{\nu} + \mathcal{O}(d),
	\end{multline*}
	which simplifies to
	\begin{align}
		\nabla \cdot (M(\varphi^\gamma) \nabla \mu^\gamma) = \divg (M(\varphi^0) \gradg \mu^0) + M(\varphi^0)\mu^1 H + M'(\varphi^*)\varphi^1 \mu^1 + 2 M(\varphi^0)\mu^2 + \mathcal{O}(d). \label{thinfilmpf10}
	\end{align}
	Now from the boundary condition \eqref{bulkBC4} one finds
	\[ 0 = \nabla \mu^\gamma \cdot \boldsymbol{\nu}^\gamma = \gradg \mu^0 \cdot \boldsymbol{\nu} + \mu^1 \boldsymbol{\nu} \cdot \boldsymbol{\nu} \pm \gamma \left( \gradg \mu^1 \cdot \boldsymbol{\nu} + 2\mu^2 \boldsymbol{\nu} \cdot \boldsymbol{\nu} \right) + \mathcal{O}(\gamma^2). \]
	Thus equating terms of the same order in $\gamma$ one finds
	\begin{gather*}
		\gradg \mu^0 \cdot \boldsymbol{\nu} + \mu^1 = 0,\\
		\gradg \mu^1 \cdot \boldsymbol{\nu} + 2\mu^2 = 0,
	\end{gather*}
	from which one concludes $\mu^1 = 0 = \mu^2$ and hence \eqref{thinfilmpf10} becomes
	\begin{align}
		\nabla \cdot (M(\varphi^\gamma) \nabla \mu^\gamma) = \divg (M(\varphi^0) \gradg \mu^0) + \mathcal{O}(d). \label{thinfilmpf11}
	\end{align}
	By combining \eqref{thinfilmpf9}, \eqref{thinfilmpf11} in \eqref{bulkeqn3} one obtains \eqref{thinfilmlimit3} by considering the zeroth order terms.\\
	
	We skip the derivation of \eqref{thinfilmlimit4} from \eqref{bulkeqn4} and \eqref{bulkBC3} as it follows the same arguments as we have used so far.
	The only point worth mentioning is that we assume the potential $F(\cdot)$ is $C^2$ so that one may indeed use Taylor's theorem as we have for previous terms.
\end{proof}
\begin{remark}
	\label{topology change}
	\begin{enumerate}
		\item The key difference between the systems \eqref{NSCH1}-\eqref{NSCH4} and \eqref{thinfilmlimit1}-\eqref{thinfilmlimit4} is that the latter system has two Lagrange multipliers to be determined, $p^0$ which is understood as enforcing the divergence free constraint, and $p^1$ which is understood as enforcing the normal velocity constraint.
		This difference is the same as observed in a comparison of various derivations of the evolving surface Navier-Stokes equations in \cite{brandner2022derivations}.
		Moreover, this only occurs in the normal direction and has no bearing on our following analysis of the tangential system \eqref{TNSCH1}-\eqref{TNSCH4}.
		\item As remarked in \cite{miura2018singular} the equations \eqref{bulkeqn2} and \eqref{thinfilmlimit2} imply\footnote{Here we are using $|\Omega_{\gamma}(t)|$ to denote the $\mathscr{L}^3$ Lebesgue measure of $\Omega_\gamma(t)$, and $|\Gamma(t)|$ the $\mathcal{H}^2$ Hausdorff measure of $\Gamma(t)$.} that
		\begin{gather*}
			\frac{d}{dt} |\Omega_{\gamma}(t)| = 0,\\
			\frac{d}{dt} |\Gamma(t)| = 0,
		\end{gather*}
		respectively.
		However, one may also use a corollary of the coarea formula (see  \cite{evans2015measure} for example) proven in \cite{miura2018singular}, Appendix A, to see that
		\[ |\Omega_\gamma(t)| = \int_{\Omega_\gamma(t)} 1= \int_{-\gamma}^{\gamma} \int_{\Gamma(t)} J(t;x, r),  \]
		where $J(t;x,r)$ is a corresponding Jacobian of the form
		\[ J(t;x,r) = 1 - rH(t;x) + r^2 K(t;x), \]
		where $K(t;x) = \det(\mbb{H}(t;x))$ is the Gaussian curvature of $\Gamma(t)$, we refer to \cite{miura2018singular} for details.
		From this one finds
		\[ |\Omega_\gamma(t)| = 2\gamma |\Gamma(t)| + \frac{2\gamma^3}{3} \int_{\Gamma(t)} K(t;x), \]
		where one finds the mean curvature term vanishes by using the divergence theorem and the fact that $\Gamma(t)$ is closed.
		Hence assuming that $\frac{d}{dt}|\Gamma(t)| = 0$ is not sufficient for the existence of a solution to \eqref{thinfilmlimit1}-\eqref{thinfilmlimit4}, one also requires
		\[ \frac{d}{dt} \int_{\Gamma(t)} K(t;x) = 0.  \]
		As seen for the evolving surface Euler, and Navier-Stokes equations in \cite{miura2018singular} this can be assured by imposing that $\Gamma(t)$ does not change its topology as by the Gauss-Bonnet theorem, see \cite{jost2008riemannian},
		\[ \int_{\Gamma(t)} K(t;x) = 2 \pi \chi(\Gamma(t)), \]
		where $\chi(\Gamma(t))$ is the Euler characteristic of $\Gamma(t)$.
	\end{enumerate}
\end{remark}

\subsection{Equivalence of the derived systems}
In this subsection we discuss the equivalence of the systems \eqref{TNSCH1}-\eqref{TNSCH4} and \eqref{thinfilmlimit1}-\eqref{thinfilmlimit4}.
Firstly, when considering \eqref{TNSCH1}-\eqref{TNSCH4}, one notes that in prescribing the normal velocity one does not have to solve \eqref{normalcomponent} but this equation must still be resolved for the normal component to be given by the prescribed velocity.
Hence there must be some normal force $F_\nu \boldsymbol{\nu}$ such that the normal velocity one would obtain from \eqref{normalcomponent} is the prescribed normal velocity $V_N$.
With this in mind the form of \eqref{NSCH1}-\eqref{NSCH4} with a prescribed normal velocity becomes
\begin{gather}
	\rho \matdev \mbf{u} = -\gradg p + pH \boldsymbol{\nu} + \gradg \cdot (2 \eta(\varphi) \mbb{E}(\mbf{u}) ) - \varepsilon \gradg \cdot (\gradg \varphi \otimes \gradg \varphi) + F_\nu \boldsymbol{\nu}, \label{NSCHequiv1}\\
	\gradg \cdot \mathbf{u} = 0,\label{NSCHequiv2}\\
	\matdev \varphi = \gradg \cdot (M(\varphi) \gradg \mu),\label{NSCHequiv3}\\
	\mu = -\varepsilon\lapg \varphi + \frac{1}{\varepsilon}F'(\varphi).\label{NSCHequiv4}
\end{gather}
Taking the normal component of \eqref{NSCHequiv1} one finds that
\begin{multline}
	F_\nu = \rho \matdev V_N + 2 \eta(\varphi)\left( \text{tr}(\mbb{H} \gradg \ut ) - V_N \text{tr}(\mbb{H}^2) \right) - \rho \ut \cdot \mbb{H} \ut + \rho \ut \cdot \gradg V_N\\
 - pH + \varepsilon \gradg \varphi \cdot \mbb{H} \gradg \varphi, \label{normalcomponent2}
\end{multline}
which one can find directly from after solving \eqref{TNSCH1}-\eqref{TNSCH4}.
This calculation is done in detail for the evolving surface Navier-Stokes equations in \cite{JanOlsReu18}.\\

Similarly, by considering the normal component of \eqref{thinfilmlimit1}-\eqref{thinfilmlimit4} one finds (up to a change of notation)
\begin{align*}
	p^1 = F_\nu + pH.
\end{align*}
Thus one finds that we may express \eqref{NSCHequiv1}-\eqref{NSCHequiv4} as
\begin{gather*}
	\rho \matdev \mbf{u} = -\gradg p + p^1 \boldsymbol{\nu} + \gradg \cdot (2 \eta(\varphi) \mbb{E}(\mbf{u}) ) - \varepsilon \gradg \cdot (\gradg \varphi \otimes \gradg \varphi),\\
	\gradg \cdot \mathbf{u} = 0,\\
	\matdev \varphi = \gradg \cdot (M(\varphi) \gradg \mu),\\
	\mu = -\varepsilon\lapg \varphi + \frac{1}{\varepsilon}F'(\varphi),
\end{gather*}
which is precisely the form of \eqref{thinfilmlimit1}-\eqref{thinfilmlimit4} from the thin film limit (up to a change of notation).
This equivalence will also hold in the presence of some tangential force, $\mbf{F}_T$.

\section{Notation, Function Spaces and Inequalities}
\label{Preliminaries}
\subsection{Notation}
Next we introduce some notation which will be used throughout.
%\begin{notation}
	For a $\mathcal{H}^2 -$measurable set, $X \subset \mathbb{R}^{3}$ and  a function $f \in L^1(X)$, we denote  
the $\mathcal{H}^2$ measure of $X$ and the mean value of $f$ on $X$ by	
	
\[|X| := \mathcal{H}^2(X), \quad \mval{f}{X} := \frac{1}{|X|} \int_X f .\]

	The components of the tangential gradient are denoted by
\[\gradg \phi = (\underline{D}_1 \phi, \underline{D}_2 \phi, \underline{D}_3 \phi ).\]
The (scalar) Sobolev spaces on $\Gamma(t)$ are defined by
\begin{gather*}
	H^{k,p}(\Gamma(t)) = \left\{ \phi \in L^p(\Gamma(t)) \mid \underline{D}_i \phi \in H^{k-1,p}(\Gamma(t)), \ i=1,2,3 \right\},
\end{gather*}
and $H^{0,p}(\Gamma(t)) := L^p(\Gamma(t))$.
We refer the reader to \cite{DecDziEll05} for further details.
We also use the following notation for tangential vector-valued Sobolev spaces,
\begin{gather*}
	\mbf{L}^p(\Gamma(t)) = \left\{\boldsymbol{\phi} \in L^p(\Gamma(t))^3 \mid \boldsymbol{\phi} \cdot \boldsymbol{\nu} = 0 \text{ almost everywhere}\right\},\\
	\mbf{H}^{k,p}(\Gamma(t)) = \left\{ \boldsymbol{\phi} \in \mbf{L}^p(\Gamma(t)) \mid \underline{D}_i \boldsymbol{\phi} \in \mbf{H}^{k-1,p}(\Gamma(t)) \ i=1,2,3 \right\},
\end{gather*}
where $\underline{D}_i \boldsymbol{\phi}$ denotes the `$i$'th column of $\gradg \boldsymbol{\phi}$, and $\mbf{H}^{0,p}(\Gamma(t)) := \mbf{L}^p(\Gamma(t))$.
As is standard, in the case $p=2$ we omit the $p$, and write $\mbf{H}^k(\Gamma(t))$.
Similarly we write $\mbf{H}^{-1}(\Gamma(t))$ for the dual space of $\mbf{H}^1(\Gamma(t))$.\\

\subsection{Pushforward map and compatible time dependent spaces}
From our assumptions we obtain the existence of a $C^3$ diffeomorphism
\[\Phi_t^n : \Gamma_0 \rightarrow \Gamma(t),\]
which is defined as $\Phi_t^n(x_0) = x(t)$, where $x(t)$ solves
\[ \frac{dx}{dt} = V_N(x(t),t) \boldsymbol{\nu}(x(t),t), \quad x(0)= x_0. \]
We denote the corresponding inverse as $\Phi_{-t}^n: \Gamma(t) \rightarrow \Gamma_0$.
We then may use the framework established in \cite{AlpCaeDju23,AlpEllSti15a}, where we use the normal pushforward map defined by $\Phi_t^n \phi = \phi \circ \Phi_t^n$, and the pullback $\Phi_{-t}^n \psi = \psi \circ \Phi_{-t}^n$, for some functions $\phi, \psi$ on $\Gamma_0$ and $\Gamma(t)$ respectively.\\

It can then be shown that we have compatibility of the pairs $(H^{k,p}(\Gamma(t)), \Phi_t^n)$ and  $(\mbf{H}^{k,p}(\Gamma(t)), \Phi_t^n)$ for $k = 0,1,2$, $p \in [1,\infty]$, in the sense of \cite{AlpCaeDju23}.
However, a known issue with the associated pushforward map is that is doesn't necessarily preserve the divergence free properties of solenoidal vector fields on $\Gamma_0$.
We remedy this by using the Piola transform as in \cite{DjuGraHer23,OlsReuZhi22}.

\subsection{The Piola transform and time differentiation}
It is clear that the differentials $D \Phi_t^n(p) : T_p \Gamma_0 \rightarrow T_{\Phi_t^n(p)} \Gamma(t)$ are invertible.
We introduce the notation $J(p,t) = \det( D \Phi_t^n(p))$, $J^{-1}(x,t) = \det( D \Phi_{-t}^n(x)) = J(t, \Phi_{-t}^n x)^{-1}$, $\mbb{D}(p,t) =  D \Phi_t^n(p) \mbb{P}(p,0)$, and $\mbb{D}^{-1}(x,t) =  D \Phi_{-t}^n(x) \mbb{P}(x,t)$.
These matrices are such that $\mbb{D}\mbb{D}^{-1} = \mbb{D}^{-1} \mbb{D} = \mbb{P}$.
One then defines the operator
\[\mbb{A}(p,t) := J^{-1}(\Phi_t^n(p), t) \mbb{D}(p,t) + \boldsymbol{\nu}(\Phi_t^n(p),t)\otimes\boldsymbol{\nu}(p,0),\]
for $p \in \Gamma_0$, $t \in [0,T]$.
One can readily observes that
\[\mbb{A}(p,t)|_{T_p\Gamma_0} : T_p\Gamma_0 \rightarrow T_{\Phi_t^n(p)} \Gamma(t), \quad \mbb{A}(p,t)|_{T_p \Gamma_0^\perp} : T_p\Gamma_0^\perp \rightarrow T_{\Phi_t^n(p)} \Gamma(t)^\perp. \]
We then define the Piola pushforward map, for a vector field $\boldsymbol{\psi}$ on $\Gamma_0$ as
\[\mathcal{P}_t \boldsymbol{\psi}(x) = \mbb{A}(\Phi_{-t}^n(x),t)\boldsymbol{\psi}(\Phi_{-t}^n(x)),\]
where it is known that for some sufficiently smooth, tangential vector field, $\boldsymbol{\psi}$, on $\Gamma_0$, then $\divg \boldsymbol{\psi} = 0$ almost everywhere $\Gamma_0$ if, and only if, $\divg \mathcal{P}_t \boldsymbol{\psi} = 0$ almost everywhere on $\Gamma(t)$.\\

One similarly defines an inverse operator $\mbb{A}^{-1}$ by
\[\mbb{A}^{-1}(x,t) := J(\Phi_{-t}^n(x), t) \mbb{D}^{-1}(x,t) + \boldsymbol{\nu}(\Phi_{-t}^n(x),0)\otimes\boldsymbol{\nu}(x,t), \]
for $x \in \Gamma(t)$, $t \in [0,T]$.
The Piola pullback is defined as one would expect, 
\[\mathcal{P}_{-t} \boldsymbol{\psi}(p) = \mbb{A}^{-1}(\Phi_{t}^n(p),t)\boldsymbol{\psi}(\Phi_{t}^n(p)),\]
for a vector field $\boldsymbol{\psi}$ on $\Gamma(t)$.
As one would expect (and hope) this is such that $\divg \boldsymbol{\psi} = 0$ almost everywhere on $\Gamma(t)$ if, and only if, $\divg \mathcal{P}_{-t} \boldsymbol{\psi} = 0$ almost everywhere on $\Gamma_0$.
We now recall the following result.
\begin{lemma}[\cite{OlsReuZhi22}, Lemma 3.1]
We have that $\mbb{D}, \mbb{A} \in C^2(\Gamma_0 \times [0,T])$, and $\mbb{D}^{-1}, \mbb{A}^{-1} \in C^2(\mathcal{G}_T)$, and are hence uniformly bounded in space and time.
\end{lemma}
We note that we have improved regularity compared to the result in \cite{OlsReuZhi22}, as we assume $\Phi_t^n$ are $C^3$-diffeomorphisms instead of $C^2$.\\

This result is then used to show compatibility, in the sense of \cite{AlpEllSti15a}, of the pairs $(\mbf{H}^k(t), \mathcal{P}_t)$ ($k=0,1,2$), and the divergence free space $(\divfree{t}, \mathcal{P}_t)$, which we discuss later.
We refer the reader to \cite{OlsReuZhi22} for details.
With this compatibility of spaces in hand one can refer to a derivative associated to the Piola transform pushforward/pullback maps, as in the sense of \cite{AlpEllSti15a}, defined by
\[\pioladev \boldsymbol{\psi} = \mathcal{P}_t \left( \frac{d}{dt} \mathcal{P}_{-t} \boldsymbol{\psi} \right).\]
We refer to this as the strong Piola derivative.
The utility of this choice of derivative is that for a sufficiently smooth vector field, $\boldsymbol{\psi}$, on $\mathcal{G}_T$ we have
\[\boldsymbol{\psi}(t) \cdot \boldsymbol{\nu} = 0\Rightarrow \pioladev \boldsymbol{\psi}\cdot \boldsymbol{\nu} = 0, \qquad \divg \boldsymbol{\psi}(t) = 0 \Rightarrow \divg \pioladev \boldsymbol{\psi}(t) = 0.\]
The corresponding weak Piola derivative is defined in the same way as the weak material/normal derivative (see \cite{AlpEllSti15a}).
The normal derivative, $\normdev$, and the Piola derivative, $\pioladev$, are related through the following.
\begin{lemma}[\cite{OlsReuZhi22}, Lemma 3.6]
	For sufficiently smooth $\boldsymbol{\psi}$ we have
	\begin{gather}
	\normdev \boldsymbol{\psi} = \pioladev \boldsymbol{\psi} - \mbb{A}(\normdev \mbb{A}^{-1}) \boldsymbol{\psi}, \label{derivatives1}\\
	\mbb{P} \normdev \boldsymbol{\psi} = \pioladev \boldsymbol{\psi} - \mbb{A}\mbb{P}(\normdev \mbb{A}^{-1}) \boldsymbol{\psi}. \label{derivatives2}
	\end{gather}
\end{lemma}
From here on we define $\bar{\mbb{A}} := \mbb{A}\mbb{P}(\normdev \mbb{A}^{-1}) \in C^1(\mathcal{G}_T)$.
We note that from this lemma we may uncontroversially consider either $\pioladev$ or $\normdev$ when we discuss bounds on the derivative of a vector-valued function.\\

From this compatibility of spaces one may now define the evolving Bochner spaces, $L^p_X$, for $p \in [1,\infty]$ and a family of Banach spaces.
We denote a pushforward/pullback map as ${\Phi}_{-t}$ and $\Phi_t$ respectively, and for our purposes these will be either $\Phi_{-t}^n$ and $\Phi_t^n$ or $\mathcal{P}_{-t}$ and $\mathcal{P}_t$.
The evolving Bochner space $L^p_X$ is
\[L^p_X = \left\{ u : [0,T] \rightarrow \bigcup_{t \in [0,T]} X(t) \times \{t\}, t \mapsto (\bar{u}(t), t) \mid \Phi_{-t} \bar{u} \in L^p(0,T;X(0)) \right\},\]
where we identify $u(t)$ with $\bar{u}(t)$.
This is a Banach space when equipped with norm
\[ \| u \|_{L^p_X} := \begin{cases}
    \left( \int_0^T \| u(t) \|_{X(t)}^p \, dt \right)^{\frac{1}{p}}, & p \in [1,\infty),\\
    \esssup_{t \in [0,T]} \|u(t) \|_{X(t)}, & p = \infty,
\end{cases} \]
and a Hilbert space for $p = 2$ and $X(t)$ a family of Hilbert spaces.\\

For a family of Hilbert spaces, $X$, we define the evolving Sobolev-Bochner space $H^{1}_{X'}$ to be
\[ H^1_{X'} = \left\{ u \in L^2_{X} \mid \matdev u \in L^2_{X'} \right\}, \]
where $\matdev u$ is the weak material derivative of $u$ associated with the maps $\Phi_{-t}, \Phi_t$.
As in \cite{AlpCaeDju23,AlpEllSti15a} we have identified $L^2_{X'} \cong (L^2_X)'$.
We refer the reader to \cite{AlpCaeDju23, AlpEllSti15a} for further details and properties of these spaces.

\subsection{Preliminary rewriting of the system}

In order to set up the weak formulation we rewrite the system \eqref{TNSCH1}-\eqref{TNSCH4} in such a way that the unknown $\ut$ is divergence free and we may eliminate the pressure.
To do this we consider the unique solution, $\Psi$, of the elliptic PDE
\[-\lapg \Psi(t) = H(t)V_N(t),\]
on $\Gamma(t)$, subject to the constraint $\mval{\Psi}{\Gamma(t)}=0$, for all $t \in [0,T]$.
Note that this is well-defined as $\int_{\Gamma(t)} HV_N = 0$ for all $t \in [0,T]$.
We then define $\widetilde{\ut} = \gradg \Psi$, from which we see that $\widetilde{\ut} \in \mbf{H}^{2,p}(\Gamma(t))$ for all $p \in [1,\infty)$, and $-\gradg \cdot \widetilde{\ut} = HV_N$.
Then defining $\hatut := \ut - \widetilde{\ut}$ we find that

\begin{gather}
	\begin{split}
	\mathbb{P} \normdev \hatut + (\gradg \hatut)\hatut + V_N \mathbb{H} \hatut - \frac{1}{2} \gradg V_N^2 = -\gradg \tilde{p} + \mbb{P} \gradg \cdot (2 \eta(\varphi) \mbb{E}(\hatut) ) + \mu \gradg \varphi + \mbf{F}_T\\
	- \mathfrak{D}_1(\varphi, \ut, \widetilde{\ut})
	\end{split} \label{TNSCH5},\\
	\gradg \cdot\hatut = 0 \label{TNSCH6},\\
	\normdev \varphi + \gradg \varphi \cdot \hatut = \lapg \mu - \mathfrak{D}_2(\varphi, \widetilde{\ut}),\label{TNSCH7}\\
	\mu = -\varepsilon\lapg \varphi +  \frac{1}{\varepsilon}F'(\varphi),\label{TNSCH8}
\end{gather}
where
\begin{multline*}
	\mathfrak{D}_1(\varphi, \ut, \widetilde{\ut}) = \mbb{P} \normdev \widetilde{\ut} +(\gradg \widetilde{\ut})\widetilde{\ut} + \gradg(\widetilde{\ut}) \hatut + \gradg(\hatut)\widetilde{\ut} + V_N \mbb{H} \widetilde{\ut}\\- \mbb{P} \divg (2 \eta(\varphi)\mbb{E}(\widetilde{\ut})),
\end{multline*}
and
\begin{align*}
	\mathfrak{D}_2(\varphi, \widetilde{\ut}) = \gradg \varphi \cdot \widetilde{\ut}.
\end{align*}
This suggests a new ``body force'', $\mbf{B}$, which is defined as
\[\mbf{B} = \mbf{F}_T - (\gradg \widetilde{\ut})\widetilde{\ut} - \mbb{P} \normdev \widetilde{\ut} - V_N \mbb{H} \widetilde{\ut}.\]
From Appendix \ref{evolvinglaplace} we know $\widetilde{\ut} \in C^0_{\mbf{H}^{2,p}}\cap C^1_{\mbf{L}^p}$, for all $p \in [1,\infty)$, and so for $\mbf{F}_T \in L^2_{\mbf{L}^2}$ one can readily show that $\mbf{B} \in L^2_{\mbf{L}^2}$.\\

The above reformulation is formal, but for sufficiently smooth $\widetilde{\ut}$ we find that this holds in a weak setting.
We show the necessary regularity properties of $\Psi$ (and $\widetilde{\ut}$) in Appendix \ref{evolvinglaplace}, and will discuss this later.
From here on we will now denote $\hatut$ as $\ut$ and treat this as the unknown velocity.
This formulation allows us to work in the space of divergence free test functions, as is typical in the analysis of the Navier-Stokes equations.
As such we introduce some notation for a suitable space of divergence free functions,
\[\mbf{V}_\sigma(t) := \left\{ \boldsymbol{\phi} \in \mbf{H}^1(\Gamma(t)) \mid \divg \boldsymbol{\phi} = 0 \right\}.\]
Similarly we define $\mbf{H}_\sigma(t)$ to be the following closure in the $\|\cdot\|_{\mbf{L}^2(\Gamma(t))}$ norm,
\[\mbf{H}_\sigma(t) := \left\{ \boldsymbol{\phi} \in C^1(\Gamma(t))^3 \mid \divg \boldsymbol{\phi} = 0, \boldsymbol{\phi}\cdot \boldsymbol{\nu} = 0 \right\}^{\|\cdot\|_{\mbf{L}^2(\Gamma(t))}}.\]
Moreover, we have compact, dense embeddings
\[ \divfree{t} {\hookrightarrow} \mbf{H}_\sigma(t) {\hookrightarrow} \divfree{t}'.\]

The appropriate weak formulation follows from multiplying by a sufficiently smooth, solenoidal test function, $\boldsymbol{\phi}$, in \eqref{TNSCH5} and a sufficiently smooth test function, $\phi$, in \eqref{TNSCH7}, \eqref{TNSCH8} and integrate over $\Gamma(t)$.
This yields \eqref{weakTNSCH1}-\eqref{weakTNSCH3} below.
Notice now that, by using the divergence theorem, one finds the pressure term vanishes as
 \[0 = \int_{\Gamma(t)} \divg (\tilde{p} \boldsymbol{\phi}) = \int_{\Gamma(t)} \gradg \tilde{p} \cdot \boldsymbol{\phi} + \int_{\Gamma({t})} \tilde{p} \divg \boldsymbol{\phi}.\]

\subsection{Some bilinear and trilinear forms}
\label{bilinear forms}
Here we introduce some bilinear/trilinear forms to be used in our weak formulation later.
\begin{gather*}
	m(t;\phi, \psi) = \int_{\Gamma(t)} \phi \psi,\\
	\mbf{m}(t;\boldsymbol{\phi}, \boldsymbol{\psi}) = \int_{\Gamma(t)} \boldsymbol{\phi}\cdot \boldsymbol{\psi},\\
	m_*(t;\Lambda, \phi) = \langle \Lambda, \phi \rangle_{H^{-1}(\Gamma(t), H^1(\Gamma(t)))},\\
	\mbf{m}_*(t;\boldsymbol{\Lambda}, \boldsymbol{\phi}) = \langle \boldsymbol{\Lambda}, \boldsymbol{\phi} \rangle_{\divfree{t}', \divfree{t}},\\
	a(t;\phi, \psi) = \int_{\Gamma(t)} \gradg \phi \cdot \gradg \psi,\\
	\mbf{a}(t;\boldsymbol{\phi}, \boldsymbol{\psi}) = 2\int_{\Gamma(t)} \mbb{E}(\boldsymbol{\phi}): \mbb{E}(\boldsymbol{\psi}),\\
	\hat{\mbf{a}}(t;\phi, \boldsymbol{\psi},\boldsymbol{\chi}) = 2\int_{\Gamma(t)} \phi \mbb{E}(\boldsymbol{\psi}): \mbb{E}(\boldsymbol{\chi}),\\
	\mbf{c}_1(t;\boldsymbol{\phi},\boldsymbol{\psi},\boldsymbol{\chi}) = \int_{\Gamma(t)} (\gradg \boldsymbol{\phi}) \boldsymbol{\psi} \cdot \boldsymbol{\chi},\\
	\mbf{c}_2(t;\phi, \psi, \boldsymbol{\chi}) = \int_{\Gamma(t)} \phi \gradg \psi \cdot \boldsymbol{\chi},\\
	\mbf{l}(t;\boldsymbol{\phi}, \boldsymbol{\psi}) = \mbf{m}(t;V_N \mbb{H}\boldsymbol{\phi}, \boldsymbol{\psi}),\\
	\mbf{d}_1(t;\boldsymbol{\phi}, \boldsymbol{\psi}) = \mbf{c}_1(t;\boldsymbol{\phi}, \widetilde{\ut}, \boldsymbol{\psi}) + \mbf{c}_1(t;\widetilde{\ut},\boldsymbol{\phi}, \boldsymbol{\psi}),\\
	\mbf{d}_2(t;\phi, \boldsymbol{\psi}) = \hat{\mbf{a}}(t;\phi, \boldsymbol{\widetilde{\ut}},\boldsymbol{\psi}),
\end{gather*}
for sufficiently smooth scalar functions $\phi, \psi$, vector functions $\boldsymbol{\phi}, \boldsymbol{\psi}, \boldsymbol{\chi}$, and linear functionals $\Lambda \in H^{-1}(\Gamma(t)), \boldsymbol{\Lambda} \in \divfree{t}'$.
We will omit the $t$ argument throughout, as above. We note the following antisymmetry properties of the trilinear forms $\mbf{c}_1, \mbf{c}_2$, which one can readily verify by using the divergence theorem,
\begin{gather*}
	\mbf{c}_1(\boldsymbol{\phi},\boldsymbol{\phi},\boldsymbol{\chi}) = -\mbf{c}_1(\boldsymbol{\chi},\boldsymbol{\phi},\boldsymbol{\phi}),\\
	\mbf{c}_1(\boldsymbol{\phi},\boldsymbol{\phi},\boldsymbol{\phi}) =0,\\
	\mbf{c}_2(\phi, \psi, \boldsymbol{\chi}) = -\mbf{c}_2(\psi, \phi, \boldsymbol{\chi}),
\end{gather*}
for $\phi, \psi \in H^1(\Gamma(t))$ and $\boldsymbol{\phi}, \boldsymbol{\chi} \in \divfree{t}$.
We will use these throughout.\\

We relate some of these bilinear forms to the normal derivatives using the transport theorem.
\begin{lemma}[Transport theorem] \label{transport theorem}
	\begin{enumerate}
		\item 	Let $\phi, \psi \in H^1_{H^{-1}} \cap L^2_{L^2}$, then
		\[\frac{d}{dt} m(\phi, \psi) = m_*(\normdev \phi, \psi) + m_*(\normdev \psi, \phi) + m(\phi, \psi HV_N).\]
		Moreover, if we have that $\phi, \psi \in H^1_{H^1}$ then
		\[\frac{d}{dt} a(\phi, \psi) =a(\normdev \phi, \psi) + a(\normdev \phi, \psi) + b(\phi, \psi),\]
		where 
		\[ b(t;\phi, \psi) := \int_{\Gamma(t)} V_N(H\mbb{I} - 2\mbb{H})\gradg \phi \cdot \gradg \psi.\]
		\item 	Let $\boldsymbol{\phi}, \boldsymbol{\psi} \in H^1_{\mbf{V}_\sigma'} \cap L^2_{\mbf{V}_\sigma}$, then
		\[\frac{d}{dt} \mbf{m}(\boldsymbol{\phi}, \boldsymbol{\psi}) = \mbf{m}_*(\normdev \boldsymbol{\phi}, \boldsymbol{\psi}) + \mbf{m}_*(\normdev \boldsymbol{\psi}, \boldsymbol{\phi}) + \mbf{m}(\boldsymbol{\phi}, \boldsymbol{\psi} HV_N).\]
		Moreover, if we have that $\boldsymbol{\phi}, \boldsymbol{\psi} \in H^1_{\mbf{V}_\sigma}$ then
		\[\frac{d}{dt} \mbf{a}(\boldsymbol{\phi}, \boldsymbol{\psi}) =\mbf{a}(\normdev \boldsymbol{\phi}, \boldsymbol{\psi}) + \mbf{a}(\boldsymbol{\phi}, \normdev \boldsymbol{\psi}) + \mbf{b}(\boldsymbol{\phi}, \boldsymbol{\psi}),\]
		where  $\mbf{b}(t;\cdot,\cdot)$, is a uniformly bounded in $t$, bilinear form $\mbf{H}^1(\Gamma(t)) \times \mbf{H}^1(\Gamma(t)) \rightarrow \mbb{R}$.
	\end{enumerate}

\end{lemma}
The relevant form for the bilinear form $\mbf{b}$ can be deduced noting that
\begin{gather*}
    \normdev \gradg \boldsymbol{\phi} = \normdev (\mbb{P} \nabla\boldsymbol{\phi}^e \mbb{P})= \normdev \mbb{P} \nabla \boldsymbol{\phi}^e \mathbb{P} + \mbb{P} \normdev \nabla \boldsymbol{\phi}^e \mbb{P} + \mbb{P} \nabla \boldsymbol{\phi}^e \normdev \mathbb{P},\\
    \normdev \nabla \boldsymbol{\phi}^e = \nabla (\normdev \boldsymbol{\phi}^e) - \nabla V_N \otimes (\nabla \boldsymbol{\phi}^e \boldsymbol{\nu}) - V_N \nabla \boldsymbol{\phi}^e \mbb{H},
\end{gather*}
for sufficiently smooth $\boldsymbol{\phi}$.
We do not give an explicit expression for $\mbf{b}$ as it is sufficiently long, and requires new notation (which would not reappear) to be written succinctly.
We do note that the smoothness assumptions on $\Gamma(t)$ allow one can show a uniform bound in the $\mbf{H}^1$ norm.

\subsection{Inequalities}
We end this section by recalling some useful inequalities.
The following results on Sobolev spaces are proven in \cite{Aub82,hebey2000nonlinear}.
\begin{theorem}
	\label{sobolevembedding}
	\
	\begin{enumerate}
		\item
		{[Poincar\'e inequality]} \ \\
		There exists a constant $C_P > 0$, independent of $t \in [0,T]$, such that for $f \in H^1(\Gamma(t))$ we have
		\[\left\| f - \mval{f}{\Gamma(t)} \right\|_{L^2(\Gamma(t))} \leq C_P \| \gradg f \|_{L^2(\Gamma(t))}.\]
		\item {[Sobolev Embeddings]} 
		\begin{enumerate}
			\item Let $0 \leq l \leq k$ be two integers, and $1\leq p < q$ be two real numbers such that $\frac{1}{q} = \frac{1}{p} - \frac{k-l}{2}$.
			Then we have continuous embedding \[H^{k,p}(\Gamma(t)) \hookrightarrow H^{l,q}(\Gamma(t)).\]
			\item If $\frac{k-r - \alpha}{2} \geq \frac{1}{p}$, where $\alpha \in (0,1)$, then we have continuous embedding
			\[H^{k,p}(\Gamma(t)) \hookrightarrow C^{r + \alpha}(\Gamma(t)).\]
		\end{enumerate}
	\end{enumerate}
	
\end{theorem}
Notice in particular that this implies
\[H^1(\Gamma(t)) \hookrightarrow L^p(\Gamma(t)),\]
for all $p\in[1, \infty)$, and in fact these embeddings are compact.
It can be shown that the operator norm of the above continuous injections are independent of time - which follows from the fact that we consider a compact time interval and sufficiently smooth evolution of $\Gamma(t)$.
We obtain analogous Sobolev embeddings for the spaces $\mbf{H}^{k,p}(\Gamma(t))$.
We also recall the following inequalities which are used throughout the analysis.
\begin{lemma}[Ladyzhenskaya's interpolation inequality, \cite{OlsReuZhi22}, Lemma 3.4]
For $\phi \in H^1(\Gamma(t))$ and $\boldsymbol{\phi} \in \mbf{H}^1(\Gamma(t))$ we have
\begin{gather}
	\| \phi \|_{L^4(\Gamma(t))} \leq C \|\phi\|_{L^2(\Gamma(t))}^{\frac{1}{2}} \|\phi\|_{H^1(\Gamma(t))}^{\frac{1}{2}}, \label{ladyzhenskaya1}\\
	\| \boldsymbol{\phi} \|_{\mbf{L}^4(\Gamma(t))} \leq C \|\boldsymbol{\phi}\|_{\mbf{L}^2(\Gamma(t))}^{\frac{1}{2}} \|\boldsymbol{\phi}\|_{\mbf{H}^1(\Gamma(t))}^{\frac{1}{2}},\label{ladyzhenskaya2}
\end{gather}
for a constant $C$ independent of $t$.
	
\end{lemma}
\begin{lemma}[Korn's inequality, \cite{OlsReuZhi22}, Lemma 3.2]
	For $\boldsymbol{\phi} \in \mbf{H}^1(\Gamma(t))$ we have
	\begin{gather}
		\|\boldsymbol{\phi}\|_{\mbf{H}^1(\Gamma(t))} \leq C\left( \|\boldsymbol{\phi}\|_{\mbf{L}^2(\Gamma(t))} + \|\mbb{E}(\boldsymbol{\phi})\|_{\mbf{L}^2(\Gamma(t))} \right), \label{korn1}
	\end{gather}
	where the constant $C$ is independent of $t$.
\end{lemma}

In order to establish energy estimates we use the following nonlinear generalisation of the Gr\"onwall inequality.
\begin{lemma}[Bihari-LaSalle inequality, \cite{bihari1956generalization}]
\label{biharilasalle}
	Let $X, K : [0,T] \rightarrow \mbb{R}$, be non-negative continuous\footnote{By density this can be shown to extend to $K \in L^1([0,T])$.} functions, $\omega: \mbb{R}^+ \rightarrow \mbb{R}^+$ be a non-decreasing continuous function, and $k \geq 0$.
	Then if
	 \[X(t) \leq k + \int_0^s K(s) \omega(X(s)) \, ds ,\]
	holds for $t \in [0,T]$, and one can choose $y_0 > 0$ such that
    \[\Omega(k) + \int_0^T K(s) \, ds \in \mathrm{dom}(\Omega^{-1}), \quad \text{where} \quad \Omega(y) := \int_{y_0}^y \frac{1}{\omega(s)} \, ds,\]
    then one for $t \in [0,T]$
	\begin{align}
		X(t) \leq \Omega^{-1}\left( \Omega(k) + \int_0^t K(s) \, ds \right). \label{biharilasalle2}
	\end{align}
	
\end{lemma}
In fact the inequality \eqref{biharilasalle2} is independent of choice of $y_0$.
\begin{comment}
\begin{lemma}[\cite{caetano2021cahn}, Lemma A.1.]
    \label{generalised gronwall}
    Let $T \in (0,\infty)$, $t^* \in (0,T)$, and $X: [0,t^*] \rightarrow [0,\infty)$ such that
    \[ X'(t) \leq C\left( C_0 + X(t) + \gamma X(t)^{q+1}, \right)\]
    for $C > 0, C_0 \geq 0, q \in \mbb{N}$ independent of $M, t^*$.
    Then if $\gamma$ is small enough so that
    \[ \gamma e^{TCq}(X(0) + C_0)^q < 1,\]
    then
    \[ X(t) \leq \frac{(X(0) + C_0)e^{CT}}{\left( 1 - \gamma e^{TCq}(X(0) + C_0)^q \right)^{\frac{1}{q}}} - C_0. \]
\end{lemma}
\end{comment}
Finally we recall three results which will be used in proving uniqueness.
\begin{lemma}[\cite{elliott2015evolving}, Lemma 4.3]
	Let $z \in H^{-1}(\Gamma(t))$ be such that $m_*(z,1) = 0$.
	Define the inverse Laplacian $\mathcal{G}z \in H^1(\Gamma(t)) \in H^1(\Gamma(t))$ as the unique solution of
	\[ a(\mathcal{G}z, \phi) = m_*(z,\phi),\]
	for all $\phi \in H^1(\Gamma(t))$.	
	If $z \in H^1_{H^{-1}}$ then
	\[\|\mathcal{G}z||_{ H^1_{H^1}}\leq C\|z\|_{H^1_{H^{-1}}}.\] 
	\end{lemma}
In fact in \cite{elliott2015evolving} it is assumed that $z \in H^1_{H^1}$ but examining the proof it is sufficient to assume $z \in H^1_{H^{-1}}$.

\begin{lemma}[\cite{li2016tropical}, Lemma 2.2]
	\label{gronwalltype}
	Let $m_1, m_2, S$ be non-negative functions on $(0,T)$ such that $m_1, S \in L^1(0,T)$ and $m_2 \in L^2(0,T)$, with $S > 0$ a.e. on $(0,T)$.
	Now suppose $f,g$ are non-negative functions on $(0,T)$, $f$ is absolutely continuous on $[0,T)$, such that\footnote{Here we are using the notation $\log^+(x) := \max(0,\log(x))$.}
	\begin{align}
		f'(t) + g(t) \leq m_1(t)f(t) + m_2(t)\left( f(t)g(t) \log^+\left( \frac{S(t)}{g(t)} \right) \right)^{\frac{1}{2}},
	\end{align}
	holds a.e. on $(0,T)$, and $f(0) = 0$.
	Then $f(t) \equiv 0$ on $[0,T)$.
\end{lemma}
The final result we mention is an evolving surface analogue of the Brezis-Gallou\"et-Wainger inequality (which originates from work on the nonlinear Schr\"odinger equation, see \cite{brezis1979nonlinear}).

\begin{lemma}
	\label{brezisgallouet}
	For $\phi \in H^2(\Gamma(t))$, one has 
	\begin{align}
		\|\phi\|_{L^\infty(\Gamma(t))} \leq C\|\phi\|_{H^1(\Gamma(t))} \left( 1 + \log\left( 1 + \frac{C\|\phi\|_{H^2(\Gamma(t))}}{\|\phi\|_{H^1(\Gamma(t))}} \right)^{\frac{1}{2}} \right),
	\end{align}
	for constants $C$ independent of $t$.
\end{lemma}
\begin{proof}
	From \cite{gorka2008brezis}, Theorem 1.1 we see that for $\phi \in H^2(\Gamma_0)$ one has
	\[ \|\phi\|_{L^\infty(\Gamma_0)} \leq C\|\phi\|_{H^1(\Gamma_0)} \left( 1 + \log\left( 1 + \frac{\|\phi\|_{H^2(\Gamma_0)}}{\|\phi\|_{H^1(\Gamma_0)}} \right)^{\frac{1}{2}} \right), \] 
	where examining the proof one finds that our assumption that $\Gamma_0$ is $C^3$ is sufficient.
	To see that one can choose the constant independent of time we observe that $\| \phi \|_{L^\infty(\Gamma(t))} = \| \Phi_{-t}^n \phi \|_{L^\infty(\Gamma_0)}$.
	Hence by pulling back to $\Gamma_0$ and using the above inequality, the compatibility of the pairs $(H^k(\Gamma(t)), \Phi_t^n)$ in the sense of \cite{AlpEllSti15a}, and the monotonicity of $x \mapsto \log(1 + x)^\frac{1}{2}$, it is clear that the inequality holds on $\Gamma(t)$ with constants independent of $t$.
    We note that this also introduces a constant into the logarithmic term.
\end{proof}

\section{Weak formulation and well posedness theorems}\label{Posedness}
We are now in a position to discuss the weak formulation.
For simplicity we assume constant mobility, $M(\cdot) \equiv 1$ and a scaling such that $\rho = 1$.

\subsection{Regular potential}
We firstly consider a smooth potential, $F$, under the same assumptions as in \cite{caetano2021cahn}.
That is, we assume $F(r) = F_1(r) + F_2(r)$ for $F_1,F_2 \in C^2(\mathbb{R})$ such that
\begin{enumerate}
	\item $F(r) \geq \beta$,
	\item $F_1 \geq 0$ is convex,
	\item $\exists q \in [1, \infty)$ such that $|F_1'(r)| \leq \alpha |r|^q + \alpha$,
	\item $|F_1'(r)| + |rF_1'(r)| \leq \alpha F_1(r) + \beta$,
	\item $ |F_2'(r)| \leq \alpha |r| + \alpha$,
\end{enumerate}
where $\alpha$ denotes some non-negative constant, and $\beta$ some real constant.
We also assume the viscosity function, $\eta(\cdot)$ is Lipschitz continuous.
A typical example of a viscosity function is
\[ \eta(r) = \eta_1\frac{(1+r)}{2} + \eta_2 \frac{(1-r)}{2}, \quad r \in [-1,1], \]
for two positive constants $\eta_1, \eta_2$, which can then be suitably extended to a $C^2$, Lipschitz continuous function on $\mbb{R}$.
The previous section then allows the following weak formulation, using the notation introduced in Section \ref{bilinear forms}.\\

Given initial data $\varphi_0 \in H^1(\Gamma_0), \mbf{u}_{T,0} \in \Hdivfree{0}$, find $\varphi \in H^{1}_{H^{-1}} \cap L^2_{H^1}, \mu \in L^2_{H^1}, \ut \in H^1_{\mbf{V}_\sigma'} \cap L^2_{\mbf{V}_\sigma}$ such that
\begin{gather}
	\begin{split}
		\mbf{m}_*(\normdev \ut, \boldsymbol{\phi}) + \hat{\mbf{a}}(\eta(\varphi), \ut, \boldsymbol{\phi}) + \mbf{c}_1(\ut,\ut, \boldsymbol{\phi}) + \mbf{l}(\ut, \boldsymbol{\phi}) + \mbf{d}_1(\ut, \boldsymbol{\phi})+\mbf{d}_2(\eta(\varphi), \boldsymbol{\phi})\\
		= \mbf{m}(\mbf{B}, \boldsymbol{\phi}) +  \mbf{c}_2(\mu,\varphi,\boldsymbol{\phi})
	\end{split},\label{weakTNSCH1}\\
	m_*(\normdev \varphi, \phi) + a(\mu, \phi) + \mbf{c}_2(\phi,\varphi,\ut) + \mbf{c}_2(\phi,\varphi, \widetilde{\ut}) = 0,\label{weakTNSCH2}\\
	m(\mu,\phi) = \varepsilon a(\varphi, \phi) + \frac{1}{\varepsilon} m(F'(\varphi), \phi),\label{weakTNSCH3}
\end{gather}
for all $\phi \in H^1(\Gamma(t)), \boldsymbol{\phi} \in \divfree{t}$ for almost all $t \in [0,T]$, and such that $\varphi(0) = \varphi_0$, $\ut(0) = \mbf{u}_{T,0}$ almost everywhere on $\Gamma_0$.\\

\begin{theorem}\label{smooth existence}
	Let $\Gamma(t)$ be a $C^3$ evolving surface, $F$ a potential function satisfying the assumptions at the beginning of the section, and $\varphi_0 \in H^1(\Gamma(0)), \mbf{u}_{T,0} \in \Hdivfree{0}$ be initial data.
    Then there exists a solution triple $(\varphi,\mu,\ut)$ on $[0,T]$ such that
	\begin{gather*}
		\varphi \in L^\infty_{H^1}\cap L^2_{H^2} \cap H^1_{H^{-1}},\\
		\mu \in L^2_{H^1},\\
		\ut \in L^\infty_{\mbf{L}^2} \cap H^1_{\mbf{V}_{\sigma}'},
	\end{gather*}
	and solving \eqref{weakTNSCH1}-\eqref{weakTNSCH3} for all $\phi \in H^1(\Gamma(t)), \boldsymbol{\phi} \in \divfree{t}$ for almost all $t \in [0,T]$, and such that $\varphi(0) = \varphi_0$, $\ut(0) = \mbf{u}_{T,0}$ almost everywhere on $\Gamma_0$.
\end{theorem}

\subsection{Logarithmic potential}

In this section we consider the well-posedness theory for the singular logarithmic potential,
\[ F(r) = \frac{\theta}{2} \left((1+r)\log(1+r) - (1-r)\log(1-r)\right) +\frac{1-r^2}{2} =: \frac{\theta}{2} F_{\log}(r) + \frac{1-r^2}{2},\]
for $\theta \in (0,1)$.
Here $\theta$ can be understood as a temperature in the system, which we have scaled for notational simplicity.
As the equations \eqref{TNSCH1}-\eqref{TNSCH4} consider the derivative of $F$ we introduce some shorthand notation,
\[ f(r) := (F_{\log}(r))' = \log\left( \frac{1+r}{1-r}\right).\]
The corresponding version of \eqref{weakTNSCH1}-\eqref{weakTNSCH3} for the logarithmic potential is as follows.
Given initial data $\varphi_0 \in \mathcal{I}_0, \mbf{u}_{T,0} \in \Hdivfree{0}$, find $\varphi \in H^{1}_{H^{-1}} \cap L^2_{H^1}, \mu \in L^2_{H^1}, \ut \in H^1_{\mbf{V}_\sigma'} \cap L^2_{\mbf{V}_\sigma}$ such that
\begin{gather}
	\begin{split}
		\mbf{m}_*(\normdev \ut, \boldsymbol{\phi}) + \hat{\mbf{a}}(\eta(\varphi), \ut, \boldsymbol{\phi}) + \mbf{c}_1(\ut,\ut, \boldsymbol{\phi}) + \mbf{l}(\ut, \boldsymbol{\phi}) + \mbf{d}_1(\ut, \boldsymbol{\phi})+\mbf{d}_2(\eta(\varphi), \boldsymbol{\phi})\\
		= \mbf{m}(\mbf{B}, \boldsymbol{\phi}) +  \mbf{c}_2(\mu,\varphi,\boldsymbol{\phi})
	\end{split},\label{logweakTNSCH1}\\
	m_*(\normdev \varphi, \phi) + a(\mu, \phi) + \mbf{c}_2(\phi,\varphi,\ut) + \mbf{c}_2(\phi,\varphi, \widetilde{\ut}) = 0,\label{logweakTNSCH2}\\
	m(\mu,\phi) = \varepsilon a(\varphi, \phi) + \frac{\theta}{2\varepsilon} m(f(\varphi), \phi) - \frac{1}{\varepsilon} m(\varphi, \phi),\label{logweakTNSCH3}
\end{gather}
for all $\phi \in H^1(\Gamma(t)), \boldsymbol{\phi} \in \divfree{t}$ for almost all $t \in [0,T]$, and such that $\varphi(0) = \varphi_0$, $\ut(0) = \mbf{u}_{T,0}$ almost everywhere on $\Gamma_0$.
Here $\mathcal{I}_0$ denotes the set of admissible initial conditions, given by
\[\mathcal{I}_0 := \left\{ \eta \in H^1(\Gamma_0) \mid \Ech[\eta;0] < \infty, \ \left| \ \mval{\eta}{\Gamma_0}\right| < 1 \right\},\]
where $\Ech$ is the Ginzburg-Landau functional
\begin{align}
	\Ech[\varphi;t] := \int_{\Gamma(t)} \frac{\varepsilon}{2}|\gradg \varphi|^2 + \frac{1}{\varepsilon}F(\varphi). \label{ginzburglandau}
\end{align}
We note there is no modification to the choice of initial velocity, so we still choose $\mbf{u}_{T,0} \in \Hdivfree{0}$ - as in the case of a polynomial potential.\\

An advantage of this singular potential is that it necessarily has ``physical solutions'', that is $|\varphi(t)| < 1 $ almost everywhere on $\Gamma(t)$ for almost all $t \in [0,T]$, due to the logarithmic nonlinearity.
As in \cite{caetano2021cahn,caetano2023regularization} there is a constraint on the initial conditions which allows us to find such a solution.
As we assume $|\Gamma(t)| = |\Gamma_0|$ we obtain a condition purely about information at $t=0$.

\begin{theorem}\label{log existence}
	Let $\Gamma(t)$ be a $C^3$ evolving surface such that $|\Gamma(t)| = |\Gamma_0|$ for all $t \in [0,T]$, $F$ be the logarithmic  potential, and $\varphi_0 \in \mathcal{I}_0, \mbf{u}_{T,0} \in \Hdivfree{0}$ be initial data.
	Then there exists a solution triple $(\varphi,\mu,\ut)$ such that
	\begin{gather*}
		\varphi \in L^\infty_{H^1}\cap L^2_{H^2} \cap H^1_{H^{-1}},\\
		\mu \in L^2_{H^1},\\
		\ut \in L^\infty_{\mbf{L}^2} \cap H^1_{\mbf{V}_{\sigma}'},
	\end{gather*}
	 solving \eqref{logweakTNSCH1}-\eqref{logweakTNSCH3}
	for all $\phi \in H^1(\Gamma(t)), \boldsymbol{\phi} \in \divfree{t}$ for almost all $t \in [0,T]$, and such that $\varphi(0) = \varphi_0$, $\ut(0) = \mbf{u}_{T,0}$ almost everywhere on $\Gamma_0$.
\end{theorem}
\section{Proof of existence}\label{Exist}
\subsection{The regular potential}
\subsubsection{Galerkin approximation}
We now show existence of a solution triple to \eqref{weakTNSCH1}-\eqref{weakTNSCH3} via the Faedo-Galerkin method.
For $M \in \mbb{N}$ we define the Galerkin approximations
\begin{gather*}
	\varphi^M(t) = \sum_{i=1}^M \alpha_i^M(t) \Phi_t^n\psi_i, \qquad \mu^M(t)= \sum_{i=1}^M \beta_i^M(t) \Phi_t^n\psi_i,\\
	\ut^M(t) = \sum_{i=1}^M \gamma_i^M(t) \mathcal{P}_t \boldsymbol{\chi}_i,
\end{gather*}
where $(\psi_i)_{i=1,...\infty}$ form a countable basis of $H^1(\Gamma_0)$, and $(\boldsymbol{\chi}_i)_{i=1,...\infty}$ form a countable basis of $\divfree{0}$, and hence pushforward onto countable bases of $H^1(\Gamma(t))$ and $\divfree{t}$ respectively.
For example, by Hilbert-Schmidt theory, one can choose $(\psi_i)_{i=1,...\infty}$ to be the eigenfunctions of the Laplace-Beltrami operator on $\Gamma_0$, and $(\boldsymbol{\chi}_i)_{i=1,...\infty}$ to be the eigenfunctions of the surface Stokes operator on $\Gamma_0$.
In particular, we choose $\psi_1 =1$ to retain the mass conservation property of $\varphi$ in the Galerkin approximation.
We define the following spaces:
\begin{gather*}
	V^M(t) := \text{span}\{ \Phi_t^n \psi_i \mid  i =1,..., M\} \subset H^1(\Gamma(t)),\\
	\mbf{V}_\sigma^M(t) := \text{span}\{ \mathcal{P}_t \boldsymbol{\chi}_i \mid  i =1,..., M\} \subset \divfree{t},
\end{gather*}
and the corresponding projections as $P^M_V(t):  H^1(\Gamma(t)) \rightarrow V^M(t), P^M_\mbf{V}(t): \Hdivfree{t} \rightarrow \mbf{V}_\sigma^M(t)$, which are defined by
\begin{gather*}
	(P^M_V(t) \phi, \psi)_{H^1(\Gamma(t))} = (\phi, \psi)_{H^1(\Gamma(t))},\ \forall \psi \in V^M(t),\\
	(P^M_\mbf{V}(t) \boldsymbol{\phi}, \boldsymbol{\psi})_{\mbf{L}^2(\Gamma(t))} = (\boldsymbol{\phi}, \boldsymbol{\psi})_{\mbf{L}^2(\Gamma(t))},\ \forall \boldsymbol{\psi} \in \mbf{V}_\sigma^M(t).
\end{gather*}
We emphasise that these are $H^1$ and $\mbf{L}^2$ projections respectively, as it is sufficient to choose initial data such that $\mbf{u}_{T,0} \in \Hdivfree{0}$, rather than $\divfree{0}$, but require $\varphi_0 \in H^1(\Gamma_0)$.

\begin{lemma}
	\label{galerkin existence}
	There exists a solution triple $(\varphi^M, \mu^M, \ut^M)$ on $[0,t^*)$, for some $t^* \leq T$ depending on $M$, solving
	\begin{gather}
		\begin{split}
			\mbf{m}(\normdev \ut^M, \boldsymbol{\phi}) + \hat{\mbf{a}}(\eta(\varphi^M), \ut^M, \boldsymbol{\phi}) + \mbf{c}_1(\ut^M,\ut^M, \boldsymbol{\phi}) +  \mbf{l}(\ut^M, \boldsymbol{\phi}) + \mbf{d}_1(\ut^M, \boldsymbol{\phi})+\mbf{d}_2(\eta(\varphi^M), \boldsymbol{\phi})\\
			= \mbf{m}(\mbf{B}, \boldsymbol{\phi}) +  \mbf{c}_2(\mu^M,\varphi^M,\boldsymbol{\phi}),
		\end{split}\label{galerkinTNSCH1}\\
		m(\normdev \varphi^M, \phi) + a(\mu^M, \phi) + \mbf{c}_2(\phi,\varphi^M,\ut^M) + \mbf{c}_2(\phi,\varphi^M, \widetilde{\ut}) = 0,\label{galerkinTNSCH2}\\
		m(\mu^M,\phi) = \varepsilon a(\varphi^M, \phi) + \frac{1}{\varepsilon} m(F'(\varphi^M), \phi),\label{galerkinTNSCH3}
	\end{gather}
	for all $\phi \in V^M(t), \boldsymbol{\phi} \in V^M_\sigma(t)$ for almost all $t \in [0,t^*)$, and such that $\varphi^M(0) = P_V^M \varphi_0$, $\ut(0) = P_\mbf{V}^M \mbf{u}_{T,0}$ almost everywhere on $\Gamma_0$.
\end{lemma}
\begin{proof}
	To begin we show the system \eqref{galerkinTNSCH1}-\eqref{galerkinTNSCH3} is equivalent to an ODE for the coefficients $\alpha, \beta, \gamma$.
	Firstly, by \eqref{derivatives2} we note that we may write
	\[\mbf{m}(\normdev \ut^M, \boldsymbol{\phi}) = \mbf{m}(\pioladev \ut^M, \boldsymbol{\phi}) + \mbf{m}(\bar{\mbb{A}} \ut^M, \boldsymbol{\phi}).\]
	Now we note that by definition of the strong derivative one finds
	\[\normdev \Phi_t^n \psi_i = 0, \qquad \pioladev \mathcal{P}_t \boldsymbol{\chi}_i = 0,\]
	and hence
	\begin{gather*}
		\normdev \varphi^M(t) = \sum_{i=1}^M (\alpha_i^M)'(t) \Phi_t^n\psi_i, \\
		\pioladev \ut^M(t) = \sum_{i=1}^M (\gamma_i^M)'(t) \mathcal{P}_t \boldsymbol{\chi}_i.
	\end{gather*}
	With this in hand, it is clear that testing \eqref{galerkinTNSCH1} with $\mathcal{P}_t \boldsymbol{\chi}_k$ yields
	\begin{multline*}
		\sum_{i=1}^M (\gamma_i^M)' \mbf{m}(\mathcal{P}_t \boldsymbol{\chi}_i, \mathcal{P}_t \boldsymbol{\chi}_k) + \sum_{i=1}^M \gamma_i^M \mbf{m}(\bar{\mbb{A}}\mathcal{P}_t \boldsymbol{\chi}_i, \mathcal{P}_t \boldsymbol{\chi}_k)\\
  + \sum_{i=1}^M \gamma_i^M \hat{\mbf{a}}\left(\eta\left( \sum_{j=1}^M \alpha_j^M \Phi_t^n\psi_j \right)\mathcal{P}_t \boldsymbol{\chi}_i, \mathcal{P}_t \boldsymbol{\chi}_k\right)
		+ \sum_{i=1}^M \sum_{j=1}^M \gamma_i^M \gamma_j^M \mbf{c}_1(\mathcal{P}_t\boldsymbol{\chi}_i,\mathcal{P}_t\boldsymbol{\chi}_j, \mathcal{P}_t\boldsymbol{\chi}_k)\\
  +\sum_{i=1}^M \gamma_i^M\mbf{l}(\mathcal{P}_t\boldsymbol{\chi}_i, \mathcal{P}_t\boldsymbol{\chi}_k) + \sum_{i=1}^M \gamma_i^M \mbf{d}_1(\mathcal{P}_t\boldsymbol{\chi}_i, \mathcal{P}_t\boldsymbol{\chi}_k)+\mbf{d}_2\left(\eta\left( \sum_{j=1}^M \alpha_j^M \Phi_t^n\psi_j \right), \mathcal{P}_t\boldsymbol{\chi}_k\right)\\
  = \mbf{m}(\mbf{B}, \mathcal{P}_t\boldsymbol{\chi}_k) + \sum_{i=1}^M \sum_{j=1}^M \alpha_j^M \beta_i^M \mbf{c}_2(\Phi_t^n \psi_i,\Phi_t^n \psi_j,\mathcal{P}_t\boldsymbol{\chi}_k).
	\end{multline*}
	Next testing \eqref{galerkinTNSCH2} and \eqref{galerkinTNSCH3} with $\Phi_t^n \psi_k$ we obtain
	\begin{multline*}
		\sum_{i=1}^M (\alpha_i^M)'m(\Phi_t^n \psi_i, \Phi_t^n \psi_k) + \sum_{i=1}^M \beta_i^Ma(\Phi_t^n \psi_i, \Phi_t^n \psi_i) + \sum_{i=1}^M \sum_{j=1}^M \alpha_i^M \gamma_j^M \mbf{c}_2(\Phi_t^n \psi_k,\Phi_t^n \psi_i,\mathcal{P}_t \boldsymbol{\chi}_j)\\
		+ \sum_{i=1}^M \alpha_i^M \mbf{c}_2(\Phi_t^n \psi_k,\Phi_t^n \psi_i, \widetilde{\ut}) = 0,
	\end{multline*}
	and
	\begin{gather*}
		\sum_{i=1}^M \beta_i^M m(\Phi_t^n \psi_i,\Phi_t^n \psi_k) = \varepsilon \sum_{i=1}^M \alpha_i^M a(\Phi_t^n \psi_i, \Phi_t^n \psi_k) + \frac{1}{\varepsilon} m\left(F'\left(\sum_{i=1}^M \alpha_i^M \Phi_t^n \psi_i\right), \Phi_t^n \psi_k\right).
	\end{gather*}
	By considering the system this generates for $k=1,...,M$ one obtain an ODE system for the vectors $\alpha^M(t), \beta^M(t), \gamma^M(t) \in \mbb{R}^M$,
    It is straightforward to see that the nonlinearities are locally Lipschitz, but we omit these details. Applying standard ODE theory one obtains the short time existence of a solution triple $(\alpha^M,\beta^M, \gamma^M)$.
\end{proof}
Next we establish existence on a time interval independent of $M$ by use of energy estimates.
To do this we recall the Ginzburg-Landau functional, \eqref{ginzburglandau}, and require the following assumptions.

\begin{assumption}
\begin{enumerate}
    \item We assume the basis $(\psi_i)_i$ of $H^1(\Gamma_0)$ is such that $\psi_1$ is constant on $\Gamma_0$.
    This guarantees that $1 \in V^M(t)$ for all $M \in \mbb{N}, t \in [0,T]$.
    \item We define $P_2^M(t): L^2(\Gamma(t)) \rightarrow V^M(t)$ to be the $L^2$ projection defined by
    \[(P^M_V(t) \phi, \psi)_{L^2(\Gamma(t))} = (\phi, \psi)_{L^2(\Gamma(t))},\ \forall \psi \in V^M(t).\]
    We assume that for $\eta \in H^1(\Gamma_0)$ that
    \[ \|P_2^M(0) \eta \|_{H^1(\Gamma_0)} \leq C \|\eta\|_{H^1(\Gamma_0)},\] 
    and given $\gamma > 0$ there exists $M^* \in \mbb{N}$ such that for $M > M^*$,
    \[\|P_2^M(0)\eta - \eta\|_{L^2(\Gamma_0)} \leq \gamma \|\eta\|_{H^1(\Gamma_0)}. \]
\end{enumerate}
\end{assumption}
These assumptions hold when we choose $(\psi_i)_i$ to be the eigenfunctions of the Laplace-Beltrami operator on $\Gamma_0$.
We notice that this second assumption implies that for $\eta \in H^1(\Gamma(t))$
\begin{align*}
   \|P_2^M(t)\eta - \eta\|_{L^2(\Gamma(t))} &\leq \|\Phi^n_{t}P_2^M(0) \Phi_{-t}^n\eta - \eta\|_{L^2(\Gamma(t))}\\
   &\leq C \|P_2^M(0) \Phi^n_{-t}\eta - \Phi^n_{-t} \eta\|_{L^2(\Gamma_0)}\\
   &\leq C \gamma \|\Phi_{-t}^n \eta\|_{H^1(\Gamma_0)} \leq C \gamma \|\eta\|_{H^1(\Gamma(t))},
\end{align*}
where the first inequality follows from the fact that $P_2^M(t)$ minimises the $L^2$ distance by definition.
Moreover, it is straightforward to see that
\[ \|P_2^M(t) \eta \|_{H^1(\Gamma(t))} \leq C \|\eta\|_{H^1(\Gamma(t))} \]

\begin{lemma}
    Given $t \in [0,T]$ such that a solution to \eqref{galerkinTNSCH1}-\eqref{galerkinTNSCH3} exists, then one has
    \[ \mval{\varphi^M}{\Gamma(t)} = \mval{\varphi_0}{\Gamma_0}\]
\end{lemma}
\begin{proof}
    We chose our basis such that $1 \in V^M(t)$, so we may test \eqref{galerkinTNSCH2} with $1$ for
    \[ m(\normdev \varphi^M, 1) + \mbf{c}_2(1, \varphi^M, \ut^M) + \mbf{c}_2(1,\varphi^M, \widetilde{\ut}) - 0. \]
    By integration by parts it is clear that
    \begin{gather*}
        \mbf{c}_2(1, \varphi^M, \ut^M) = - \mbf{c}_2(\varphi^M,1, \ut^M) = 0,\\
        \mbf{c}_2(1,\varphi^M, \widetilde{\ut}) = m(HV_N, \varphi^M)-\mbf{c}_2(\varphi^M,1, \widetilde{\ut}),
    \end{gather*}
    and so one finds 
    \[ \frac{d}{dt} m(\varphi^M, 1) =m(\normdev \varphi^M, 1) + m(\varphi^M, HV_N) = 0. \]
    Hence we have shown that
    \[ \int_{\Gamma(t)} \varphi^M = \int_{\Gamma_0} P_V^M \varphi_0 = (\varphi_0, 1)_{H^1(\Gamma_0)} = \int_{\Gamma_0} \varphi_0 \]
    where we have again used the definition of $P_V^M \varphi_0$, and the fact that $1 \in V_M(0)$.
    The equality for the mean values then follows since $|\Gamma(t)| = |\Gamma_0|$.
\end{proof}
The same logic applies to the solution of \eqref{weakTNSCH1}-\eqref{weakTNSCH3}.

\begin{lemma} \label{galerkin energy}
    For sufficiently large $M$, the solution triple $(\varphi^M, \mu^M, \ut^M)$ satisfies
	\begin{multline}
		\label{energyestimate}
		\sup_{t \in [0,T]} \left(\frac{1}{2}\| \ut^M\|_{\mbf{L}^2(\Gamma(t))}^2 + \Ech[\varphi^M;t] \right) + \frac{1}{2} \int_0^T \left( \eta_*\| \mbb{E}(\ut^M)\|_{\mbf{L}^2(\Gamma(t))}^2 + \|\gradg \mu^M \|_{L^2(\Gamma(t))}^2\right) \\
		\leq C,
	\end{multline}
	for a constant $C$ independent of $M$.
\end{lemma}
\begin{proof}
	To begin, we test \eqref{galerkinTNSCH1} with $\ut^M$ for
	\begin{multline*}
		\mbf{m}(\normdev \ut^M, \ut^M) + \hat{\mbf{a}}(\eta(\varphi^M), \ut^M, \ut^M)  = \mbf{m}(\mbf{B}, \ut^M) +  \mbf{c}_2(\mu^M,\varphi^M,\ut^M) - \mbf{l}(\ut^M, \ut^M)\\
  - \mbf{d}_1(\ut^M, \ut^M) -\mbf{d}_2(\eta(\varphi^M), \ut^M),
	\end{multline*}
	where we have used $\mbf{c}_1(\ut^M,\ut^M, \ut^M) = 0$.
	Next we notice that from \eqref{galerkinTNSCH2} that
	\begin{align*}
		\mbf{c}_2(\mu^M,\varphi^M,\ut^M) = -m(\normdev \varphi^M, \mu^M) - a(\mu^M, \mu^M) - \mbf{c}_2(\mu^M,\varphi^M, \widetilde{\ut}),
	\end{align*}
	and from \eqref{galerkinTNSCH3} we find
	\[m(\normdev \varphi^M, \mu^M) = \varepsilon a(\varphi^M, \normdev \varphi^M) + \frac{1}{\varepsilon} m(F'(\varphi^M), \normdev \varphi^M).\]
	Hence combining these three equalities it is clear that
	\begin{multline}
		\mbf{m}(\normdev \ut^M, \ut^M) + \hat{\mbf{a}}(\eta(\varphi^M), \ut^M, \ut^M) + \varepsilon a(\varphi^M, \normdev \varphi^M) + \frac{1}{\varepsilon} m(F'(\varphi^M), \normdev \varphi^M) + a(\mu^M, \mu^M) \\
		= \mbf{m}(\mbf{B}, \ut^M) -\mbf{c}_2(\mu^M,\varphi^M, \widetilde{\ut}) - \mbf{l}(\ut^M, \ut^M) - \mbf{d}_1(\ut^M, \ut^M) -\mbf{d}_2(\eta(\varphi^M), \ut^M). \label{energypf1}
	\end{multline}
	Recalling Lemma \ref{transport theorem} we find that
	\begin{gather*}
		\mbf{m}(\normdev \ut^M, \ut^M) = \frac{1}{2} \frac{d}{dt}\mbf{m}(\ut^M, \ut^M) - \frac{1}{2}\mbf{m}(\ut^M, H V_N \ut^M),\\
		a(\varphi^M, \normdev \varphi^M) = \frac{1}{2} \frac{d}{dt} a(\varphi^M,\varphi^M) - \frac{1}{2} b(\varphi^M, \varphi^M),\\
		m(F'(\varphi^M),\normdev \varphi^M) = \frac{d}{dt} m(F(\varphi^M),1) - m(F(\varphi^M), HV_N).
	\end{gather*}
	Next, by using the bound $\eta_* \leq \eta(\cdot)$ and the above we see
	\begin{multline}
		\frac{1}{2} \frac{d}{dt}\mbf{m}(\ut^M, \ut^M) + \frac{d}{dt}\Ech[\varphi^M;t] + \eta_*\mbf{a}(\ut^M, \ut^M) + a(\mu^M, \mu^M) = \mbf{m}(\mbf{B}, \ut^M)\\
		+ \frac{1}{2}\mbf{m}(\ut^M, H V_N \ut^M) + \frac{\varepsilon}{2}b(\varphi^M,\varphi^M)+ \frac{1}{\varepsilon}m(F(\varphi^M), HV_N) -\mbf{c}_2(\mu^M,\varphi^M, \widetilde{\ut})\\
  - \mbf{l}(\ut^M, \ut^M) - \mbf{d}_1(\ut^M, \ut^M) -\mbf{d}_2(\eta(\varphi^M), \ut^M). \label{energypf2}
	\end{multline}
	The focus now is bounding these terms on the right hand side.
	Firstly we find
	\begin{multline}
		\mbf{m}(\mbf{B}, \ut^M) + \frac{1}{2}\mbf{m}(\ut^M, H V_N \ut^M) + \mbf{l}(\ut^M,\ut^M) \leq \frac{1}{2} \|\mbf{B}\|_{\mbf{L}^2(\Gamma(t))}^2\\
		+\left( \frac{1}{2} + \frac{1}{2}\|HV_N\|_{L^\infty(\Gamma(t))}  + \|V_N \mbb{H}\|_{\mbf{L}^{\infty}(\Gamma(t))}\right) \|\ut^M\|_{\mbf{L}^2(\Gamma(t))}^2, \label{energypf3}
	\end{multline}
	where we note that $H,V_N \in C^1(\mathcal{G}_T), \mbb{H} \in (C^1(\mathcal{G}_T))^{3 \times 3} $ by our assumptions.
	Similarly, this smoothness assumption on $HV_N$ allows us to bound
	\begin{align}
		\frac{\varepsilon}{2}b(\varphi^M,\varphi^M) + \frac{1}{\varepsilon}m(F(\varphi^M), HV_N) \leq C \Ech[\varphi^M;t], \label{energypf4}
	\end{align}
	for a constant independent of $t$ and $M$.
	We now look at the terms introduced by the influence of $\widetilde{\ut}$.
	It is straightforward to see that
	\begin{align*}
		|\mbf{d}_1(\ut^M, \ut^M)|  \leq \|\widetilde{\ut}\|_{\mbf{L}^\infty(\Gamma(t))} \|\ut^M\|_{\mbf{H}^1(\Gamma(t))} \|\ut^M\|_{\mbf{L}^2(\Gamma(t))} + \|\widetilde{\ut}\|_{\mbf{H}^{1,\infty}(\Gamma(t))} \|\ut^M\|_{\mbf{L}^2(\Gamma(t))}^2
	\end{align*}
	where we have used the regularity result from Appendix \ref{evolvinglaplace} to bound $\widetilde{\ut}$.
	We then use \eqref{ladyzhenskaya2} to see that
	\[\|\ut^M\|_{\mbf{H}^1(\Gamma(t))} \|\ut^M\|_{\mbf{L}^2(\Gamma(t))} \leq C \|\ut^M\|_{\mbf{L}^2(\Gamma(t))}^{\frac{3}{2}}\|\mbb{E}(\ut^M)\|_{\mbf{L}^2(\Gamma(t))}^{\frac{1}{2}},\]
	and hence from Young's inequality we obtain
	\[\|\widetilde{\ut}\|_{\mbf{L}^\infty(\Gamma(t))} \|\ut^M\|_{\mbf{H}^1(\Gamma(t))} \|\ut^M\|_{\mbf{L}^2(\Gamma(t))} \leq \frac{\eta_*}{4}\|\mbb{E}(\ut^M)\|_{\mbf{L}^2(\Gamma(t))}^2 + C \|\widetilde{\ut}\|_{\mbf{L}^\infty(\Gamma(t))}^{\frac{4}{3}} \|\ut^M\|_{\mbf{L}^2(\Gamma(t))}^2.\]
	All in all, this yields a bound on $\mbf{d}_1(\ut^M, \ut^M)$ given by
	\begin{align}
		|\mbf{d}_1(\ut^M, \ut^M)| \leq \frac{\eta_*}{4}\|\mbb{E}(\ut^M)\|_{\mbf{L}^2(\Gamma(t))}^2 + \left( \|\widetilde{\ut}\|_{\mbf{H}^{1,\infty}(\Gamma(t))} +  C \|\widetilde{\ut}\|_{\mbf{L}^\infty(\Gamma(t))}^{\frac{4}{3}} \right)\|\ut^M\|_{\mbf{L}^2(\Gamma(t))}^2. \label{energypf5}
	\end{align}
	
	Likewise it is straightforward to see that
	\begin{align}
		|\mbf{d}_2(\eta(\varphi^M), \ut^M)| \leq \frac{\eta_*}{4} \|\mbb{E}(\ut^M)\|_{\mbf{L}^2(\Gamma(t))}^2 + \frac{2(\eta^*)^2}{\eta_*} \|\mbb{E}(\widetilde{\ut})\|_{\mbf{L}^2(\Gamma(t))}^2 \label{energypf6}.
	\end{align}
	We lastly consider the $\mbf{c}_2$ term and see that
	\[\mbf{c}_2(\mu^M,\varphi^M, \widetilde{\ut}) = \int_{\Gamma(t)} \mu^M \gradg \varphi^M \cdot \widetilde{\ut} = -\int_{\Gamma(t)}\varphi^M \gradg \mu^M \cdot \widetilde{\ut} - \int_{\Gamma(t)} \varphi^M \mu^M \divg \widetilde{\ut},\]
	which follows from integration by parts.
	By construction of $\widetilde{\ut}$ we now find that
	\[-\mbf{c}_2(\mu^M,\varphi^M, \widetilde{\ut}) = \mbf{c}_2(\varphi^M,\mu^M, \widetilde{\ut}) - m(\mu^M, HV_N \varphi^M).\]
	The difficulty now is in bounding $|m(\mu^M, HV_N \varphi^M)|$, for which we argue as in \cite{caetano2021cahn}.
    By definition we find that
    \[m(\mu^M, HV_N \varphi^M) = m(\mu^M, P_2^M(t) (HV_N\varphi^M)),\]
    where we observe that we can now test \eqref{galerkinTNSCH3} with $P_2^M(t) (HV_N\phi^M)$.
    This yields
    \[m(\mu^M, P_2^M(t) (HV_N\varphi^M)) = \varepsilon a(\varphi^M,P_2^M(t) (HV_N\varphi^M)) + \frac{1}{\varepsilon}m(F'(\varphi^M), P_2^M(t) (HV_N\varphi^M)),\]
    and hence we find that
    \begin{multline*}
    |m(\mu^M, HV_N \varphi^M)| \leq C \|\varphi^M\|_{H^1(\Gamma(t)}^2 + \frac{1}{\varepsilon}m(F'(\varphi^M),HV_N\varphi^M)\\
    + \frac{1}{\varepsilon}m(F'(\varphi^M), P_2^M(t) (HV_N\varphi^M) - HV_N \varphi^M).
    \end{multline*}
    Now recalling that assumptions on $F_1', F_2'$ and the assumptions on $P_2^M$ one finds that for sufficiently large $M$,
    \[ |m(\mu^M, HV_N \varphi^M)| \leq C + C \Ech[\varphi^M;t] + C \gamma \|\gradg \varphi^M\|_{L^2(\Gamma(t))}^{q+1}, \]
    where we have also used the Sobolev embedding $L^q(\Gamma(t)) \hookrightarrow H^1(\Gamma(t))$.
    We note that $\gamma \rightarrow 0$ as $M \rightarrow \infty$.
	From this bound one readily finds that
	\begin{align}
		|\mbf{c}_2(\mu^M,\varphi^M, \widetilde{\ut})| \leq C + \frac{1}{2}\|\gradg \mu^M\|_{L^2(\Gamma(t))}^2 + C \Ech[\varphi^M;t] + C \gamma \Ech[\varphi^M;t]^{\frac{q+1}{2}}, \label{energypf7}
	\end{align}
	for constants $C$ independent of $M$, and some small $\gamma > 0$.\\
	
	Finally, by using the estimates \eqref{energypf3}-\eqref{energypf7} in \eqref{energypf2} and integrating over $[0,t]$ we obtain an inequality of the form
	\begin{multline}
		\frac{1}{2}\| \ut^M\|_{\mbf{L}^2(\Gamma(t))}^2 + \Ech[\varphi^M;t] + \frac{\eta_*}{2} \int_0^t \| \mbb{E}(\ut^M)\|_{\mbf{L}^2(\Gamma(t))}^2 + \frac{1}{2}\int_0^t \|\gradg \mu^M \|_{L^2(\Gamma(t))}^2 \leq k\\
		+\int_0^t K(s) \left(\frac{1}{2}\| \ut^M\|_{\mbf{L}^2(\Gamma(s))}^2 + \Ech[\varphi^M;s] +  \gamma\left(\frac{1}{2}\| \ut^M\|_{\mbf{L}^2(\Gamma(s))}^2 + \Ech[\varphi^M;s]\right)^{\frac{q+1}{2}}\right) \, ds, \label{energypf8}
	\end{multline}
	where
	\begin{multline*}
		k = C + \frac{1}{2}\int_0^T \|\mbf{B}\|_{\mbf{L}^2(\Gamma(s))}^2 \, ds + \frac{2(\eta^*)^2}{\eta_*} \int_0^T \|\mbb{E}(\widetilde{\ut})\|_{\mbf{L}^2(\Gamma(s))}^2 \, ds + \frac{1}{2}\| P_\mbf{V}^M \mbf{u}_{T,0}\|_{\mbf{L}^2(\Gamma(0))}^2\\
		+ \Ech[P_V^M \varphi_0;0],
	\end{multline*}
	and
	\begin{align*}
		K(s) =  C + \frac{1}{2}\|HV_N\|_{L^\infty(\Gamma(s))}  + \|V_N \mbb{H}\|_{\mbf{L}^{\infty}(\Gamma(s))} +  \|\widetilde{\ut}\|_{\mbf{H}^{1,\infty}(\Gamma(s))} +  C \|\widetilde{\ut}\|_{\mbf{L}^\infty(\Gamma(s))}^{\frac{4}{3}}
	\end{align*}
	We can bound $k$ independently of $M$ by noting that
	\[ \|P_V^M \varphi_0\|_{H^1(\Gamma(0))} \leq \|\varphi_0\|_{H^1(\Gamma(0))}, \qquad \| P^M_\mbf{V} \mbf{u}_{T,0}\|_{\mbf{L}^2(\Gamma(0))} \leq \| \mbf{u}_{T,0}\|_{\mbf{L}^2(\Gamma(0))},\]
	and as above we can bound the potential term in $\Ech$ by using Sobolev embeddings.\\
	
	Lastly, we note that while $\Ech[\varphi^M;t]$ is not necessarily non-negative, it is bounded below.
	Hence we add some sufficiently large constant to \eqref{energypf8} so that the analogous inequality holds for the modified energy,
	\[\widetilde{\Ech}[\varphi;t] := \Ech[\varphi;t] + \tilde{\beta} \geq \frac{\varepsilon}{2}\|\gradg \varphi\|_{L^2(\Gamma(t))}^2.\]
	From our assumptions it is clear such a constant exists and depends only on the evolution of $\Gamma(t)$ and choice of $F$.
	Thus one obtains \eqref{energyestimate} by using Lemma \ref{biharilasalle}, with $\omega(s) = s + \gamma s^{\frac{q+1}{2}}$, where one can verify that for $q > 1$
    \[ \Omega(y) = \frac{2}{q-1} \log\left(\frac{\gamma + y_0^{\frac{1-q}{2}}}{\gamma + y^{\frac{1-q}{2}}} \right), \]
    for a suitable choice of $y_0 > 0$.
    In order to apply Lemma \ref{biharilasalle} we need that 
    \[\Omega(k) + \int_0^T K(s) \, ds \in \mathrm{dom}(\Omega^{-1}) = R(\Omega) = \left(-\infty, \frac{2}{q-1} \log\left( \frac{\gamma + y_0^{\frac{1-q}{2}}}{\gamma} \right)\right),\]
    which in turn follows if
    \[ -\frac{2}{q-1} \log\left({\gamma + k^{\frac{1-q}{2}}} \right) + \int_0^T K(s) \, ds \in \left( -\infty, \frac{-2}{q-1} \log(\gamma) \right).\]
    Taking $M$ sufficiently large, and hence $\gamma$ sufficiently small, we may now apply Lemma \ref{biharilasalle} to show \eqref{energyestimate}.
    In the case $q=1$ we may apply the usual Gr\"onwall inequality instead.
\end{proof}

From \eqref{energyestimate}, the growth conditions on $F'$, and the Sobolev embedding $H^1(\Gamma(t)) \hookrightarrow L^{2q}(\Gamma(t))$ one finds that $F'(\varphi^M) \in L^2_{L^2}$, and so it is clear that one obtains uniform $L^2_{H^1}$ bounds on $\mu^M$.
Likewise, using \eqref{korn1}, one can establish uniform $L^2_{\mbf{H}^1}$ bounds for $\ut^M$.

\subsubsection{Passage to the limit}
We have established uniform bounds for $\ut^M$ in $ L^\infty_{\mbf{H}_\sigma}$ and $ L^2_{\mbf{H^1}}$, for $\varphi^M$ in $L^\infty_{H^1}$, and for $\mu^M$ in $L^2_{H^1}$.
Thus, there exist limiting functions $\ut, \varphi, \mu$ such that
\begin{gather*}
	\ut^M \rightharpoonup \ut, \text{ weakly in } L^2_{\mbf{V}_\sigma},\\
	\ut^M \overset{*}{\rightharpoonup} \ut, \text{ weak-}* \text{ in } L^\infty_{\mbf{H}_\sigma},\\
	\varphi^M \overset{*}{\rightharpoonup} \varphi, \text{ weak-}* \text{ in } L^\infty_{H^1},\\
	\mu^M \rightharpoonup \mu, \text{ weakly in } L^2_{H^1}.
\end{gather*}
Moreover, arguing as in \cite{caetano2021cahn}, one can show that $\varphi^M \rightarrow \varphi$ strongly in $L^2_{L^2}$.
This is useful as one cannot use (a variant of) the Aubin-Lions theorem to show strong convergence of $\varphi, \ut$ because we do not have uniform estimates for $\varphi^M$ in $H^1_{H^{-1}}$ and $\ut^M$ in $H^1_{\mbf{V}_\sigma'}$.
Now by proceeding as in the proof of \cite{caetano2021cahn}, Proposition 4.10, one shows the existence of $\normdev \varphi$.
We do not do this in detail for the sake of brevity, but we outline the argument.
Firstly, one considers a sufficiently smooth test function so that one can pass the derivative onto this test function.
By careful choice of test function, one can pass to the limit, using the weak convergence of $\varphi^M, \mu^M$, and obtain an equation which characterises the weak material time derivative.
For explicit details we refer to \cite{caetano2021cahn} Proposition 4.10.
We compute a bound for $\normdev \varphi$ in $L^2_{H^{-1}}$ in Lemma \ref{smooth derivatives bound} below.\\

Recalling the compact embeddings
\[ {H^1}(\Gamma(t)) \hookrightarrow L^p(\Gamma(t)),\]
for all $p \in (1, \infty)$, by using \cite{AlpCaeDju23}, Theorem 6.2,  one obtains strongly convergent subsequences such that
\begin{gather*}
	\varphi^M \rightarrow \varphi, \text{ strongly in } L^2_{L^p}.
\end{gather*}
The case $p=4$ will be useful when we discuss the passage to the limit of $\ut^M$.\\

We now discuss passage to the limit in the nonlinear terms.
Firstly, it is shown in \cite{caetano2021cahn} that
\[F'(\varphi^M) \rightharpoonup F'(\varphi), \text{ weakly in } L^2_{L^2},\]
where the authors use a generalisation of the dominated convergence theorem.
To see that we can pass to the limit in the $\mbf{c}_1$ term one argues as in \cite{temam2001navier} Lemma 3.2.
Similarly, to see the convergence in the $\mbf{c}_2(\mu^M, \varphi^M, \boldsymbol{\phi})$ term in \eqref{galerkinTNSCH1} we write
\[\mbf{c}_2(\mu^M, \varphi^M, \boldsymbol{\phi}) = -\mbf{c}_2(\varphi^M, \mu^M, \boldsymbol{\phi}) = \mbf{c}_2(\varphi - \varphi^M, \mu^M, \boldsymbol{\phi}) - \mbf{c}_2(\varphi, \mu^M, \boldsymbol{\phi}).\]
Firstly, we notice that $\int_0^T \mbf{c}_2(\varphi, \cdot, \boldsymbol{\phi}) $ is an element of $(L^2_{H^1})'$ from the established bounds, and hence the weak convergence of $\mu^M$ in $L^2_{H^1}$ yields
\[\int_0^T \mbf{c}_2(\varphi, \mu^M, \boldsymbol{\phi}) \rightarrow \int_0^T \mbf{c}_2(\varphi, \mu, \boldsymbol{\phi})\]
as $M\rightarrow \infty$.
For the other term, we use the strong convergence $\varphi^M \rightarrow \varphi$ in $L^2_{L^4}$, and write
\[\left| \int_0^T \mbf{c}_2(\varphi - \varphi^M, \mu^M, \boldsymbol{\phi}) \right| \leq \int_0^T \|\varphi-\varphi^M\|_{L^4(\Gamma(t))} \|\gradg \mu^M \|_{L^2(\Gamma(t))} \|\boldsymbol{\phi}\|_{\mbf{L}^4(\Gamma(t))} \rightarrow 0.\]
Thus we find
\[\int_0^T \mbf{c}_2(\mu^M, \varphi^M, \boldsymbol{\phi}) \rightarrow -\int_0^T\mbf{c}_2(\varphi, \mu, \boldsymbol{\phi}) = \int_0^T \mbf{c}_2(\mu, \varphi, \boldsymbol{\phi}),\]
and by using standard localisation arguments one concludes that $\mbf{c}_2(\mu^M, \varphi^M, \boldsymbol{\phi}) \rightarrow \mbf{c}_2(\mu, \varphi, \boldsymbol{\phi})$ for almost all $t \in [0,T]$ and all $\boldsymbol{\phi} \in \divfree{t}$.
$\mbf{c}_2(\phi, \varphi^M, \ut^M)$ converges by similar logic.\\

The final nonlinear term to consider the convergence of is $\hat{\mbf{a}}(\eta(\varphi^M), \ut^M, \boldsymbol{\phi})$.
This follows essentially the same calculations as above, owing to the assumption that we have a Lipschitz continuous viscosity $\eta(\cdot)$.
One writes
\[\hat{\mbf{a}}(\eta(\varphi^M), \ut^M, \boldsymbol{\phi}) = \hat{\mbf{a}}(\eta(\varphi), \ut^M, \boldsymbol{\phi}) + \hat{\mbf{a}}(\eta(\varphi^M) - \eta(\varphi), \ut^M, \boldsymbol{\phi}),\]
and as above we find $\int_0^T \hat{\mbf{a}}(\eta(\varphi), \cdot, \boldsymbol{\phi})$ is an element of $L^2_{\mbf{V}_\sigma'}$.
Likewise, the integral of the second term over $[0,T]$ vanishes\footnote{Here one considers $\boldsymbol{\phi}$ sufficiently smooth so that $\mbb{E}(\boldsymbol{\phi}) \in L^\infty_{\mbf{L}^4}$ and extends to $\boldsymbol{\phi} \in L^2_{\divfree{t}}$ by density.} as $M\rightarrow \infty$ by using the Lipschitz property of $\eta(\cdot)$ and the strong convergence of $\varphi^M$ in $L^2_{L^4}$.\\

With these considerations, one argues as in \cite{OlsReuZhi22} to see that $\normdev \ut$ exists.
This argument is similar to the aforementioned proof that $\normdev \varphi$ exists, and consists of passing the time derivative onto some sufficiently smooth test function and carefully taking the limit $M \rightarrow \infty$.
We omit the details here as the main point of interest is in the nonlinear terms.
Now we know these derivatives exist, we may show the following bounds.
\begin{lemma}
	\label{smooth derivatives bound}
	We have $\normdev \varphi \in L^2_{H^{-1}}$ and $\normdev \ut \in L^2_{\mbf{V}_\sigma'}$.
\end{lemma}
\begin{proof}
	To see the bound for $\normdev \varphi^M$ observe from \eqref{weakTNSCH1} that
	\[|m(\normdev \varphi, \phi)| \leq |a(\mu, \phi)| + |\mbf{c}_2(\phi,\varphi,\ut)| + |\mbf{c}_2(\phi,\varphi, \widetilde{\ut})|,\]
	and we recall
	\begin{gather*}
		\mbf{c}_2(\phi,\varphi,\ut) = -\mbf{c}_2(\varphi,\phi,\ut),\\
		\mbf{c}_2(\phi,\varphi, \widetilde{\ut}) =  m(\phi, HV_N \varphi) - \mbf{c}_2(\varphi,\phi, \widetilde{\ut}).
	\end{gather*}
	From this it is clear to see that
	\begin{multline*}
		\frac{|m(\normdev \varphi, \phi)|}{\|\phi\|_{H^1(\Gamma(t))}} \leq \|\gradg \mu\|_{L^2(\Gamma(t))} + \|\varphi\|_{L^4(\Gamma(t))} \|\ut\|_{\mbf{L}^4(\Gamma(t))}+ \|H V_N\|_{L^\infty(\Gamma(t))} \|\varphi\|_{L^2(\Gamma(t))}\\
		+ \|\varphi\|_{L^4(\Gamma(t))} \|\widetilde{\ut}\|_{\mbf{L}^4(\Gamma(t))},
	\end{multline*}
	and hence using \eqref{ladyzhenskaya1}, \eqref{ladyzhenskaya2}, and the uniform bounds established by \eqref{energyestimate} we obtain the $L^2_{H^{-1}}$ bound.\\
	
	Similarly to bound $\normdev \ut$, we see from \eqref{galerkinTNSCH1} that
	\begin{multline*}
		|\mbf{m}(\normdev \ut, \boldsymbol{\phi})| \leq  |\hat{\mbf{a}}(\eta(\varphi), \ut, \boldsymbol{\phi})| + |\mbf{c}_1(\ut,\ut, \boldsymbol{\phi})| +  |\mbf{l}(\ut, \boldsymbol{\phi})| + |\mbf{d}_1(\ut, \boldsymbol{\phi})|\\
  +|\mbf{d}_2(\eta(\varphi), \boldsymbol{\phi})| + |\mbf{m}(\mbf{B}, \boldsymbol{\phi})| +  |\mbf{c}_2(\mu,\varphi,\boldsymbol{\phi})|.
	\end{multline*}
	Recall from properties of $\mbf{c}_1, \mbf{c}_2$ that
	\begin{gather*}
		|\mbf{c}_1(\ut,\ut, \boldsymbol{\phi})| = |\mbf{c}_1( \boldsymbol{\phi},\ut,\ut)| \leq \|\boldsymbol{\phi}\|_{\mbf{H}^1(\Gamma(t))} \|\ut\|_{\mbf{L}^4(\Gamma(t))}^2,\\
		|\mbf{c}_2(\mu,\varphi,\boldsymbol{\phi})| = |\mbf{c}_2(\varphi,\mu,\boldsymbol{\phi})|\leq C \|\varphi\|_{L^4(\Gamma(t))}\|\gradg \mu \|_{L^2(\Gamma(t))} \|\boldsymbol{\phi}\|_{\mbf{H}^1(\Gamma(t))},
	\end{gather*}
	where we have used the Sobolev embedding $\mbf{H}^1(\Gamma(t)) \hookrightarrow \mbf{L}^4(\Gamma(t))$ in the second inequality.
	The only other problematic term here is $|\mbf{d}_1(\ut, \boldsymbol{\phi})|$, which we bound as
	\begin{align*}
		|\mbf{d}_1(\ut, \boldsymbol{\phi})| &\leq |\mbf{c}_1(\ut, \widetilde{\ut}, \boldsymbol{\phi})| + |\mbf{c}_1(\widetilde{\ut},\ut, \boldsymbol{\phi})|\\
		&\leq |\mbf{c}_1(\boldsymbol{\phi}, \widetilde{\ut}, \ut)| + |\mbf{c}_1(\boldsymbol{\phi},\ut,\widetilde{\ut})| + 2|\mbf{m}(\ut HV_N, \boldsymbol{\phi})|\\
		&\leq C\left( \|\widetilde{\ut}\|_{\mbf{L}^4(\Gamma(t))} + 1 \right)\| \ut\|_{\mbf{L}^4(\Gamma(t))} \|\boldsymbol{\phi}\|_{\mbf{H}^1(\Gamma(t))}.
	\end{align*}
	From these inequalities it is straightforward to see that
	\begin{multline*}
		\frac{|\mbf{m}(\normdev \ut, \boldsymbol{\phi})|}{\|\boldsymbol{\phi}\|_{\mbf{H}^1(\Gamma(t))}} \leq \eta^* \|\mbb{E}(\ut)\|_{\mbf{L}^2(\Gamma(t))} + \|\ut\|_{\mbf{L}^4(\Gamma(t))}^2 + C \|\varphi\|_{L^4(\Gamma(t))}\|\gradg \mu \|_{L^2(\Gamma(t))}\\
		+ \| V_N \mbb{H} \|_{\mbf{L}^\infty(\Gamma(t))} \| \ut\|_{\mbf{L}^2(\Gamma(t))}+C \|\ut\|_{\mbf{L}^2(\Gamma(t))}+ C\| \ut\|_{\mbf{L}^4(\Gamma(t))} \|\widetilde{\ut}\|_{\mbf{L}^4(\Gamma(t))}\\
		+ \eta^*\|\mbb{E}(\widetilde{\ut})\|_{\mbf{L}^2(\Gamma(t))} + \| \mbf{B} \|_{\mbf{L}^2(\Gamma(t))}.
	\end{multline*}
	From this inequality it is straightforward to see how one obtains the $L^2_{\mbf{V}_\sigma'}$ bound by using \eqref{ladyzhenskaya2}, \eqref{energyestimate}, and Sobolev embeddings where necessary.
\end{proof}

Lastly it remains to discuss the initial conditions.
We note that $\ut^M(0) = P^M_\mbf{V} \mbf{u}_{T,0}$, and $\varphi^M(0) = P^M_V \varphi_0$.
These projections are such that $P^M_\mbf{V} \mbf{u}_{T,0} \rightarrow \mbf{u}_{T,0}$ strongly in $\Hdivfree{0}$, and $P^M_V \varphi_0 \rightarrow \varphi_0$ strongly in $H^1(\Gamma(0))$.
By using standard arguments (again we refer to \cite{caetano2021cahn,OlsReuZhi22}) one can then verify that $\ut(0) = \mbf{u}_{T,0}$, and $\varphi(0) = \varphi_0$.
All in all we have shown Theorem \ref{smooth existence}.

\subsection{The logarithmic potential}
As is common in the literature, see for example \cite{caetano2021cahn, elliott1991generalized}, we consider a regularised version of this potential and use the preceding theory to show existence.
To this end, we choose $\delta \in (0,1)$ and define a regularised function
\[ F_{\log}^\delta(r) = 
\begin{cases}
	(1-r)\log(\delta) + (1+r)\log(2-\delta) + \frac{(1-r)^2}{2 \delta} + \frac{(1+r)^2}{2(2-\delta)}, & r \geq 1 - \delta,\\
	F_{\log}(r), & r \in (-1+\delta, 1- \delta),\\
	(1+r)\log(\delta) + (1-r)\log(2-\delta) + \frac{(1+r)^2}{2 \delta} + \frac{(1-r)^2}{2(2-\delta)}, & r \leq -1 + \delta.
\end{cases}
\]
It is a straightforward calculation to see that $F_{\log}^{\delta} \in C^2(\mbb{R})$.
As before, we also introduce shorthand notation $\fd(r) = (F_{\log}^\delta)'(r)$.\\

We now focus on the following version of \eqref{logweakTNSCH1}-\eqref{logweakTNSCH3}.
We want to find $\varphi^\delta \in H^1_{H^{-1}} \cap L^\infty_{H^1}$, $\mu^\delta \in L^2_{H^1}$, and $\ut^\delta \in H^1_{\mbf{V}_\sigma'} \cap L^2_{\mbf{V}_\sigma}$ such that
\begin{gather}
	\begin{split}
		\mbf{m}_*(\normdev \ut^\delta, \boldsymbol{\phi}) + \hat{\mbf{a}}(\eta(\varphi^\delta), \ut^\delta, \boldsymbol{\phi}) + \mbf{c}_1(\ut^\delta,\ut^\delta, \boldsymbol{\phi}) + \mbf{l}(\ut^\delta, \boldsymbol{\phi}) + \mbf{d}_1(\ut^\delta, \boldsymbol{\phi})+\mbf{d}_2(\eta(\varphi^\delta), \boldsymbol{\phi})\\
		= \mbf{m}(\mbf{B}, \boldsymbol{\phi}) +  \mbf{c}_2(\mu^\delta,\varphi^\delta,\boldsymbol{\phi})
	\end{split},\label{regweakTNSCH1}\\
	m_*(\normdev \varphi^\delta, \phi) + a(\mu^\delta, \phi) + \mbf{c}_2(\phi,\varphi^\delta,\ut^\delta) + \mbf{c}_2(\phi,\varphi^\delta, \widetilde{\ut}),\label{regweakTNSCH2}\\
	m(\mu^\delta,\phi) = \varepsilon a(\varphi^\delta, \phi) + \frac{\theta}{2\varepsilon} m(\fd(\varphi^\delta), \phi) - \frac{1}{\varepsilon}m(\varphi^\delta, \phi),\label{regweakTNSCH3}
\end{gather}
for all $\phi \in H^1(\Gamma(t)), \boldsymbol{\phi} \in \divfree{t}$ and almost all $t \in [0,T]$, such that $\varphi^\delta(0) = \varphi_0 \in \mathcal{I}_0$ and $\ut^\delta(0) = \mbf{u}_{T,0} \in \Hdivfree{0}$.\\

Global existence then follows from Theorem \ref{smooth existence}, since the nonlinear term $\fd(\cdot)$ has linear growth ($q=1$) by construction.
However, as our polynomial conditions depend on our regularisation parameter, $\delta$, we now have to establish new energy estimates before passing to the limit $\delta \rightarrow 0$.
\begin{lemma}
	The solution triple $(\varphi^\delta, \mu^\delta, \ut^\delta)$ is such that
	\begin{multline}
		\label{regenergyestimate}
		\sup_{t \in [0,T]} \left(\frac{1}{2}\| \ut^\delta\|_{\mbf{L}^2(\Gamma(t))}^2 + \Echd[\varphi^\delta;t] \right) + \frac{\eta_*}{2} \int_0^T \| \mbb{E}(\ut^\delta)\|_{\mbf{L}^2(\Gamma(t))}^2 + \frac{1}{2}\int_0^T \|\gradg \mu^\delta \|_{L^2(\Gamma(t))}^2\\
		\leq C,
	\end{multline}
	for a constant $C$ independent of $\delta$.
\end{lemma}
\begin{proof}
	For the interest of brevity we try to recycle as much of the proof of the previous energy estimate.
	Examining the proof of Lemma \ref{galerkin energy} one finds that the analogues of \eqref{energypf1}-\eqref{energypf6} still hold.
	Continuing in the same way, we find that
	\[-\mbf{c}_2(\mu^\delta,\varphi^\delta, \widetilde{\ut}) = \mbf{c}_2(\varphi^\delta,\mu^\delta, \widetilde{\ut}) - m(\mu^\delta, HV_N \varphi^\delta).\]
	In particular, all we need to establish is an $L^2$ bound on $\mu^\delta$.
	To do this we test \eqref{regweakTNSCH3} against $\varphi^\delta H V_N \in H^1(\Gamma(t))$ for
	\begin{align}
		m(\mu^\delta, HV_N \varphi^\delta) = \varepsilon a(\varphi^\delta, HV_N \varphi^\delta) + \frac{\theta}{2 \varepsilon} m(\fd(\varphi^\delta), \varphi^\delta HV_N) - \frac{1}{\varepsilon} m(\varphi^\delta, \varphi^\delta HV_N). \label{energypf9}
	\end{align}
	We then notice that $r\fd(r) \geq 0$, and so
	\begin{align*}
		\frac{\theta}{2 \varepsilon} m(\fd(\varphi^\delta), \varphi^\delta HV_N) &\leq \frac{\theta}{2 \varepsilon}\|HV_N\|_{L^{\infty}(\Gamma(t))} \|\fd(\varphi^\delta)\varphi^\delta\|_{L^1(\Gamma(t))}\\
		&= \frac{\theta}{2 \varepsilon} \|HV_N\|_{L^{\infty}(\Gamma(t))} m(\fd(\varphi^\delta), \varphi^\delta),
	\end{align*}
	and using \eqref{regweakTNSCH3} we see
	\begin{align}
		\frac{\theta}{2 \varepsilon} m(\fd(\varphi^\delta), \varphi^\delta) = m(\mu^\delta, \varphi^\delta) - \varepsilon a(\varphi^\delta, \varphi^\delta) + \frac{1}{\varepsilon} m(\varphi^\delta, \varphi^\delta). \label{energypf10}
	\end{align}
	Now by using the inverse Laplacian, we see
	\begin{align*}
		m(\mu^\delta, \varphi^\delta) &= m\left(\mu^\delta, \varphi^\delta - \mval{\varphi^\delta}{\Gamma(t)} \right) + m\left(\mu^\delta, \mval{\varphi^\delta}{\Gamma(t)}\right)\\
		&= a\left(\mu^\delta, \mathcal{G} \left(\varphi^\delta - \mval{\varphi^\delta}{\Gamma(t)} \right)\right) + m\left(\mu^\delta, \mval{\varphi^\delta}{\Gamma(t)}\right).
	\end{align*}
	Now by using the definition of the inverse Laplacian, and Poincar\'e's inequality one finds
	\[ a\left(\mu^\delta, \mathcal{G} \left(\varphi^\delta - \mval{\varphi^\delta}{\Gamma(t)} \right)\right) \leq C \| \gradg \mu^\delta\|_{L^2(\Gamma(t))} \| \gradg \varphi^\delta\|_{L^2(\Gamma(t))}.\]
	Also, as we chose $\varphi_0 \in \mathcal{I}_0$ there exists some constant $\alpha < 1$ such that
	\[ m\left(\mu^\delta, \mval{\varphi^\delta}{\Gamma(t)}\right) = m(\mu^\delta, 1)\mval{\varphi^\delta}{\Gamma(t)} \leq |m(\mu^\delta, 1)| \left| \ \mval{\varphi_0}{\Gamma_0} \right| \leq \alpha  |m(\mu^\delta, 1)|,  \]
	and hence we find
	\begin{align}
		|m(\mu^\delta, \varphi^\delta)| \leq C \| \gradg \mu^\delta\|_{L^2(\Gamma(t))} \| \gradg \varphi^\delta\|_{L^2(\Gamma(t))} + \alpha  |m(\mu^\delta, 1)|. \label{energypf11}
	\end{align}
	Testing \eqref{regweakTNSCH2} against $\phi \equiv 1$ and noting that $\fd(r) \leq r \fd(r) + 1$,
	\[ m(\mu^\delta, 1) = \frac{\theta}{2 \varepsilon}m(\fd(\varphi^\delta),1) - \frac{1}{\varepsilon}m(\varphi^\delta, 1) \leq \frac{\theta}{2 \varepsilon} |\Gamma(t)| +  \frac{\theta}{2 \varepsilon}m(\fd(\varphi^\delta),\varphi^\delta) - \frac{1}{\varepsilon}m(\varphi^\delta, 1),\]
	thus using \eqref{energypf10} we see
	\begin{align}
		|m(\mu^\delta, 1)| \leq C + C \| \gradg \varphi^\delta \|_{L^2(\Gamma(t)}^2 + |m(\varphi^\delta, \mu^\delta)|, \label{energypf12}
	\end{align}
	where we have used Young's and Poincar\'e's inequalities as appropriate.
	Now by combining \eqref{energypf11}, \eqref{energypf12} we see
	\[|m(\mu^\delta, \varphi^\delta)| \leq C + C \|\gradg \varphi^\delta \|_{L^2(\Gamma(t))}^2 + C \| \gradg \mu^\delta\|_{L^2(\Gamma(t))} \| \gradg \varphi^\delta\|_{L^2(\Gamma(t))},\]
	for constants which depend on $\alpha$.
	Now using this and \eqref{energypf10} in \eqref{energypf9} one can show that
	\[|m(\mu^\delta, HV_N \varphi^\delta)| \leq  C + C \|\gradg \varphi^\delta \|_{L^2(\Gamma(t))}^2 + C \| \gradg \mu^\delta\|_{L^2(\Gamma(t))} \| \gradg \varphi^\delta\|_{L^2(\Gamma(t))},\]
	for constants which depend on $\sup_{t \in [0,T]}\|HV_N\|_{H^{1,\infty}(\Gamma(t))}$.
	Now by using Young's inequality it then follows that
	\[|\mbf{c}_2(\mu^\delta,\varphi^\delta, \widetilde{\ut})|\leq C + C \| \gradg \varphi^\delta\|_{L^2(\Gamma(t))}^2 + \frac{1}{2} \| \gradg \mu^\delta\|_{L^2(\Gamma(t))}^2,\]
	which is the analogue of the bound \eqref{energypf7}.
	Combining this with the analogues of \eqref{energypf1}-\eqref{energypf6} it is straightforward to see that
	\begin{multline}
		\frac{1}{2}\| \ut^M\|_{\mbf{L}^2(\Gamma(t))}^2 + \Echd[\varphi^M;t] + \frac{\eta_*}{2} \int_0^t \| \mbb{E}(\ut^M)\|_{\mbf{L}^2(\Gamma(t))}^2 + \frac{1}{2}\int_0^t \|\gradg \mu^M \|_{L^2(\Gamma(t))}^2\\
  \leq k +\int_0^t K(s) \left(\frac{1}{2}\| \ut^M\|_{\mbf{L}^2(\Gamma(s))}^2 + \Echd[\varphi^M;s]\right) \, ds, \label{energypf13}
	\end{multline}
 for
	\begin{multline*}
	    k = C + \frac{1}{2}\int_0^T \|\mbf{B}\|_{\mbf{L}^2(\Gamma(s))}^2 \, ds + \frac{2(\eta^*)^2}{\eta_*} \int_0^T \|\mbb{E}(\widetilde{\ut})\|_{\mbf{L}^2(\Gamma(s))}^2 \, ds + \frac{1}{2}\| \mbf{u}_{T,0}\|_{\mbf{L}^2(\Gamma(0))}^2\\
     + \Echd[\varphi_0;0],
	\end{multline*}
	and
	\begin{align*}
		K(s) = C +  \frac{1}{2}\|HV_N\|_{L^\infty(\Gamma(s))}  + \|V_N \mbb{H}\|_{\mbf{L}^{\infty}(\Gamma(s))} +  \|\widetilde{\ut}\|_{\mbf{H}^{1,\infty}(\Gamma(s))} +  C \|\widetilde{\ut}\|_{\mbf{L}^\infty(\Gamma(s))}^{\frac{4}{3}}.
	\end{align*}
	The claim then follows from Gr\"onwall's inequality once we establish that $\Echd[\varphi_0;0]$ is independent of $\delta$.
	To see this we find that
	\begin{multline*}
		\Echd[\varphi_0;0] = \int_{\Gamma(t)} \frac{\varepsilon|\gradg \varphi_0|^2}{2} + \frac{\theta}{2\varepsilon} F_{\log}^\delta(\varphi_0)\\
		= \frac{\varepsilon}{2}\|\varphi_0\|_{H^1(\Gamma_0)}^2 + \frac{1}{2 \varepsilon} \|1 - \varphi_0\|_{L^2(\Gamma_0)}^2+ \frac{\theta}{2\varepsilon} \int_{\{|\varphi_0|< 1 - \delta\}} F_{\log}(\varphi_0)\\
		+ \frac{\theta}{2\varepsilon} \int_{\{\varphi_0 \geq 1- \delta\}} F_{\log}^\delta(\varphi_0)+ \frac{\theta}{2\varepsilon} \int_{\{\varphi_0 \leq -1 + \delta\}} F_{\log}^\delta(\varphi_0), 
	\end{multline*}
	and so it is only necessary to bound the last three terms.
	Firstly, from the definition of $F_{\log}(\cdot)$, it is straightforward to see that
	\[ \frac{\theta}{2\varepsilon} \int_{\{|\varphi_0|< 1 - \delta\}} F_{\log}(\varphi_0) \leq \frac{\theta\log(2)|\Gamma_0|}{\varepsilon}.  \]
	For the terms involving $F_{\log}^\delta(\cdot)$ we use the set they're integrating over to bound them, for example
	\begin{align*}
		\int_{ \{ \varphi_0 \geq 1-\delta \} }  (1-\varphi_0)\log(\delta) + (1+\varphi_0)\log(2-\delta) + \frac{(1-\varphi_0)^2}{2 \delta} + \frac{(1+\varphi_0)^2}{2(2-\delta)}\\
		\leq \int_{\Gamma_0} -\delta \log(\delta) + 2\log(2) + \frac{\delta}{2} + 2 \leq \left( \frac{1}{e} + 2\log(2) + \frac{5}{2} \right)|\Gamma_0|,
	\end{align*}
	and the other term is bounded similarly.
	Thus we see $\Echd[\varphi_0;0]$ is bounded independent of $\delta$ and we obtain \eqref{regenergyestimate}.
\end{proof}
Notice that this proof also gave us uniform bounds for $\mu^\delta$ in $L^2_{L^2}$.
We now use this energy estimate to obtain some further $\delta$-independent bounds.
\begin{lemma}
	There exists a constant $C$ independent of $\delta$ such that:
	\begin{enumerate}
		\item \begin{align*}
			\int_0^T \| \normdev \ut^\delta \|_{\divfree{t}'}^2 + \int_0^T \| \normdev \varphi^\delta\|_{H^{-1}(\Gamma(t))}^2 \leq C,
		\end{align*}
		\item \begin{align*}
			\int_0^T \| \fd(\varphi^\delta)\|_{L^2(\Gamma(t))}^2 \leq C.
		\end{align*}
	\end{enumerate}
	
\end{lemma}
\begin{proof}
	The proof for the bounds on the time derivatives is identical to that of Lemma \ref{smooth derivatives bound}, and is hence omitted.
	The bound for the regularised potential term was not required in the setting of a smooth potential as we used the polynomial growth conditions and Sobolev embeddings to bound this term.
	Testing \eqref{regweakTNSCH3} against $\fd(\varphi^\delta) \in H^1(\Gamma(t))$ yields
	\[ \|\fd(\varphi^\delta)\|_{L^2(\Gamma(t)}^2 = \frac{2\varepsilon}{\theta} m(\mu^\delta, \fd(\varphi^\delta)) - \frac{2 \varepsilon^2}{\theta} a(\varphi^\delta, \fd(\varphi^\delta)) + \frac{2}{\theta} m(\varphi^\delta, \fd(\varphi^\delta)). \]
	Now we notice that as $(\fd)'(\cdot) \geq 0$ one has
	\[ -a(\varphi^\delta, \fd(\varphi^\delta)) = -\int_{\Gamma(t)} (\fd)'(\varphi^\delta) \gradg \varphi^\delta \cdot \gradg \varphi^\delta \leq 0, \]
	hence by using Young's inequality it is clear that
	\[\|\fd(\varphi^\delta)\|_{L^2(\Gamma(t)}^2 \leq \frac{8}{\theta^2}\left( \|\varphi^\delta\|_{L^2(\Gamma(t))}^2 + \varepsilon^2 \|\mu^\delta\|_{L^2(\Gamma(t))}^2 \right).\]
	The bound then follows from integrating in time and using the uniform $L^2_{L^2}$ bounds.
\end{proof}
\subsubsection{Passage to the limit}
Now we want to pass to the limit as $\delta \rightarrow 0$.
We have established uniform bounds for $\ut^\delta$ in $ L^\infty_{\mbf{L}^2}, L^2_{\mbf{H^1}}$ and $H^1_{\mbf{V}_\sigma'}$, hence there is a limiting function $\ut$ such that
\begin{gather*}
	\ut^\delta \rightharpoonup \ut, \text{ weakly in } L^2_{\mbf{V}_\sigma},\\
	\ut^\delta \overset{*}{\rightharpoonup} \ut, \text{ weak-}* \text{ in } L^\infty_{\mbf{L}^2},\\
	\ut^\delta \rightharpoonup \ut, \text{ weakly in } H^1_{\mbf{V}_\sigma'}.
\end{gather*}
Likewise $\varphi^\delta$ is uniformly bounded in $L^\infty_{H^1}$ and $H^1_{H^{-1}}$, and $\mu^\delta$ is uniformly bounded in $L^2_{H^1}$.
Thus, there exist limiting functions $\varphi, \mu$ such that
\begin{gather*}
	\varphi^\delta \overset{*}{\rightharpoonup} \varphi, \text{ weak-}* \text{ in } L^\infty_{H^1},\\
	\varphi^\delta \rightharpoonup \varphi, \text{ weakly in } H^1_{H^{-1}},\\
	\mu^\delta \rightharpoonup \mu, \text{ weakly in } L^2_{H^1}.
\end{gather*}
Lastly, the uniform bounds for $\fd(\varphi^\delta)$ in $L^2_{L^2}$ imply the existence of some $\tilde{f} \in L^2_{L^2}$ such that
\begin{gather*}
	\fd(\varphi^\delta) \rightharpoonup \tilde{f}, \text{ weakly in } L^2_{L^2}.
\end{gather*}
One also has the same strong convergence properties as seen for smooth potentials, which follow from the relevant compact embeddings.\\

It remains for us to show that the limiting functions are such that $|\varphi(t)|< 1$ almost everywhere on $\Gamma(t)$ for almost all $t \in [0,T]$, and that $\tilde{f} = f(\varphi)$.
From the piecewise definition of $\fd$ it is clear that these two issues are related.
We firstly recall a result from \cite{caetano2021cahn}.
\begin{lemma}[\cite{caetano2021cahn}, Lemma 5.8]
	There exist constants $C_1,C_2$ independent of $\delta, t$ such that
	\[ \int_{\Gamma(t)} [\varphi^\delta(t) - 1]_+ + \int_{\Gamma(t)} [-\varphi^\delta(t)-1]_+ \leq C_1 \delta + \frac{C_2}{|\log(\delta)|},  \]
	for almost all $t \in [0,T]$.
\end{lemma}
The proof of this result is unchanged with the coupling with the Navier-Stokes equations, and hence follows without any adaptations.
In particular, in the limit $\delta \rightarrow 0$, we find that $|\varphi(t)| \leq 1$ almost everywhere on $\Gamma(t)$ for almost all $t \in [0,T]$.
There is still the issue of the set of values such that $|\varphi(t)| = 1$, which we discuss using the same arguments as in \cite{caetano2021cahn}.
\begin{lemma}
	For almost all $t \in [0,T]$ the set $\{ x \in \Gamma(t) \mid |\varphi(x,t)| = 1 \}$ has $\mathcal{H}^2$ measure 0.
\end{lemma}
\begin{proof}
	To begin, we claim that up to a subsequence of $\delta \rightarrow 0$ we have for almost all $t \in [0,T]$ that
	\[ \lim_{\delta \rightarrow 0} \fd(\varphi^\delta(t)) = \begin{cases}
		f(\varphi(t)), & \text{ if } |\varphi(t)| < 1 \text{ almost everywhere on } \Gamma(t),\\
		\infty, & \text{ otherwise.}
	\end{cases} \]
	This is in fact exactly the content of \cite{caetano2021cahn}, Lemma 5.10 - and so we do not prove this.
	Then we have that
	\[ \lim_{\delta \rightarrow 0} \varphi^\delta(t)\fd(\varphi^\delta(t)) = \begin{cases}
		\varphi(t) f(\varphi(t)), & \text{ if } |\varphi(t)| < 1 \text{ almost everywhere on } \Gamma(t),\\
		\infty, & \text{ otherwise,}
	\end{cases} \]
	pointwise, almost everywhere on $\Gamma(t)$ for almost all $t \in [0,T]$.
	Now by testing \eqref{regweakTNSCH3} against $\varphi^\delta$, integrating in time, and using the established uniform bounds for $\varphi^\delta$ and $\mu^\delta$ one finds
	\[ \int_0^T \int_{\Gamma(t)} \varphi^\delta \fd(\varphi^\delta)  \leq C,\]
	for a constant $C$ independent of $\delta$.
	Now noting that $r\fd(r) \geq 0$, by Fatou's lemma one has
	\[ \int_0^T \int_{\Gamma(t)} \liminf_{\delta \rightarrow 0} \varphi^\delta \fd(\varphi^\delta)  \leq \liminf_{\delta \rightarrow 0} \int_0^T \int_{\Gamma(t)} \varphi^\delta \fd(\varphi^\delta)  \leq C.  \]
	From the claim it is now evident that one necessarily has that for almost all $t \in [0,T]$ the set $\{ x \in \Gamma(t) \mid |\varphi(x,t)| = 1 \}$ has measure 0.
\end{proof}
In particular, we see that $\fd(\varphi^\delta(x,t)) \rightarrow f(\varphi(x,t))$ for almost all $x \in \Gamma(t), t \in [0,T]$.
Hence, again by using Fatou's lemma, one finds
\[ \int_0^T \|f(\varphi)\|_{L^2(\Gamma(t))}^2 = \int_0^T \int_{\Gamma(t)} \liminf_{\delta \rightarrow 0} |\fd(\varphi^\delta)|^2  \leq \liminf_{\delta \rightarrow 0} \int_0^T \|\fd(\varphi^\delta)\|_{L^2(\Gamma(t))}^2\leq C,\]
and hence $f(\varphi) \in L^2_{L^2}$.
Now by the uniqueness of weak limits and a suitable variant of the dominated convergence theorem for evolving surfaces (see \cite{caetano2021cahn}, Theorem B.2), one finds that $f(\varphi) = \tilde{f}$.
\section{Proof of uniqueness}
\label{Unique}

\subsection{Uniqueness for the regular potential}In this section we prove the uniqueness of solutions to \eqref{weakTNSCH1}-\eqref{weakTNSCH3}.
As a preliminary result, we note that by elliptic regularity theory one has that
\begin{align}
	\int_0^T \|\varphi\|_{H^2(\Gamma(t))}^2 \leq C\int_0^T \left( \|\mu\|_{L^2(\Gamma(t))}^2 + \|F'(\varphi)\|_{L^2(\Gamma(t))}^2 \right).
	\label{extra regularity}
\end{align}
This $L^2_{H^2}$  regularity of $\varphi$ is invaluable for proving the uniqueness of solutions, as it allows one to eliminate $\mu$ from \eqref{weakTNSCH1}.
To see this we notice that for almost all $t \in [0,T]$, and all $\boldsymbol{\phi} \in \divfree{t}$
\[ \mbf{c}_2(\mu,\varphi, \boldsymbol{\phi}) = \int_{\Gamma(t)} \mu \gradg \varphi \cdot \boldsymbol{\phi} = \int_{\Gamma(t)} \left(-\varepsilon \lapg \varphi + \frac{1}{\varepsilon} F'(\varphi)\right) \gradg \varphi \cdot \boldsymbol{\phi}.\]
Now, as observed in the derivation, one formally calculates that
\[ -\varepsilon \lapg \varphi \gradg \varphi = -\varepsilon \gradg \cdot (\gradg \varphi \otimes \gradg \varphi)  + \frac{\varepsilon}{2} \gradg |\gradg \varphi|^2 - \varepsilon\left( \gradg \varphi \cdot \mbb{H} \gradg \varphi \right) \boldsymbol{\nu},\]
almost everywhere on $\Gamma(t)$.
Thus one finds
\[ \mbf{c}_2(\mu,\varphi, \boldsymbol{\phi}) = \int_{\Gamma(t)}  \gradg \left( \frac{\varepsilon}{2} |\gradg \varphi|^2 + \frac{1}{\varepsilon} F(\varphi)\right) \cdot \boldsymbol{\phi} -\varepsilon \gradg \cdot (\gradg \varphi \otimes \gradg \varphi) \cdot \boldsymbol{\phi}  ,\]
where the normal term has vanished as $\boldsymbol{\phi}$ is tangential.
Now using integration by parts, and the fact that $\boldsymbol{\phi}$ is solenoidal, it is clear that
\begin{multline*}
	\int_{\Gamma(t)}  \gradg \left( \frac{\varepsilon}{2} |\gradg \varphi|^2 + \frac{1}{\varepsilon} F(\varphi)\right) \cdot \boldsymbol{\phi}= \int_{\Gamma(t)}  \divg \left( \frac{\varepsilon}{2} |\gradg \varphi|^2 \boldsymbol{\phi} + \frac{1}{\varepsilon} F(\varphi)\boldsymbol{\phi} \right)\\
	- \int_{\Gamma(t)}  \left( \frac{\varepsilon}{2} |\gradg \varphi|^2 + \frac{1}{\varepsilon} F(\varphi)\right) \divg \boldsymbol{\phi} = 0,
\end{multline*}
and
\begin{align*}
	-\varepsilon\int_{\Gamma(t)} \gradg \cdot (\gradg \varphi \otimes \gradg \varphi)) \cdot \boldsymbol{\phi} = \varepsilon\int_{\Gamma(t)} (\gradg \varphi \otimes \gradg \varphi): \gradg \boldsymbol{\phi}.
\end{align*}
Hence we find
\[ \mbf{c}_2(\mu,\varphi, \boldsymbol{\phi})= \varepsilon \mbf{c}_3(\varphi,\varphi,\boldsymbol{\phi}),\]
where the trilinear form $\mbf{c}_3$ is defined as
\[ \mbf{c}_3(t; \phi, \psi, \boldsymbol{\chi}) := \int_{\Gamma(t)} (\gradg \phi \otimes \gradg \psi): \gradg \boldsymbol{\chi} .\]
\\
The structure of this proof is similar to that in \cite{giorgini2019uniqueness}, with relevant modifications for an evolving surface - as discussed in Appendix \ref{inversestokes}.
\begin{theorem}\label{smooth uniqueness}
	Let $\Gamma(t), F$ be such that the assumptions in Theorem \ref{smooth existence} hold.
	Moreover, assume $F_2'$ is Lipschitz continuous.
	Then the solution triple, $(\varphi,\mu,\ut)$ solving \eqref{weakTNSCH1}-\eqref{weakTNSCH3} is unique.
\end{theorem}
    
    The first step is to observe that if we have two solution triples, $(\varphi^i, \mu^i, \ut^i)$, $i = 1,2$, with the same initial data and defining
    \begin{gather*}
		\bar{\varphi} := \varphi^1 - \varphi^2,\\
		\bar{\mu} := \mu^1 - \mu^2,\\
		\bar{\ut} := \ut^1 - \ut^2,
	\end{gather*}
    then these solve the system
	\begin{gather}
		\begin{split}
			\mbf{m}_*\left(\normdev {\bar{\ut}}, \boldsymbol{\phi} \right) + \hat{\mbf{a}}(\eta(\varphi^1), \ut^1, \boldsymbol{\phi}) - \hat{\mbf{a}}(\eta(\varphi^2), \ut^2, \boldsymbol{\phi}) + \mbf{c}_1(\ut^1,\ut^1, \boldsymbol{\phi}) - \mbf{c}_1(\ut^2,\ut^2, \boldsymbol{\phi})\\
			+ \mbf{l}(\bar{\ut}, \boldsymbol{\phi}) + \mbf{d}_1(\bar{\ut}, \boldsymbol{\phi}) +\mbf{d}_2(\eta(\varphi^1) - \eta(\varphi^2), \boldsymbol{\phi})\\
			=  \varepsilon \mbf{c}_3(\varphi^1,\varphi^1,\boldsymbol{\phi}) - \varepsilon \mbf{c}_3(\varphi^2,\varphi^2,\boldsymbol{\phi})
		\end{split}, \label{uniqueness1}\\
		m_*\left(\normdev {\bar{\varphi}}, \phi\right) + a(\bar{\mu}, \phi) + \mbf{c}_2(\phi,\varphi^1,\ut^1) - \mbf{c}_2(\phi,\varphi^2,\ut^2) + \mbf{c}_2(\phi,\bar{\varphi},\widetilde{\ut}) = 0,\label{uniqueness2}\\
		m(\bar{\mu},\phi) = \varepsilon a(\bar{\varphi}, \phi) + \frac{1}{\varepsilon} m(F'(\varphi^1) - F'(\varphi^2), \phi), \label{uniqueness3}
	\end{gather}
	for almost all $t \in [0,T]$, and all $\phi \in H^1(\Gamma(t)), \boldsymbol{\phi} \in \divfree{t}$.
	This proof firstly requires obtaining bounds for $\bar{\varphi}$ and $\bar{\ut}$ in appropriate norms.
    We refer to Appendix \ref{inversestokes} for the definition of one of the norms we will use.\\
	
	The proof relies on proving the following differential inequalities.
	\begin{lemma}\label{uniqueineq}
	$\bar{\ut}$ is such that
	 \begin{multline}
		\frac{1}{2} \frac{d}{dt} \| \bar{\ut} \|_{\mathcal{S}}^2 + \frac{\eta_*}{2} \| \bar{\ut} \|_{\mbf{L}^2(\Gamma(t))}^2 \leq \frac{\varepsilon}{4} \| \gradg \bar{\varphi} \|_{L^2(\Gamma(t))}+ K_1(t) \| \bar{\ut}\|_{\mathcal{S}}^2+ C_{\log}(\bar{\varphi}, \ut^1)\| \bar{\ut}\|_{\mathcal{S}},\label{uniquenessu}
	\end{multline}
	where
	\begin{multline*}
		K_1(t) = C \left( 1 +  \|\mbb{E}(\ut^1)\|_{\mbf{L}^2(\Gamma(t))}^2 + \| \ut^1 \|_{\mbf{L}^4(\Gamma(t))}^4 + \| \ut^2 \|_{\mbf{L}^4(\Gamma(t))}^4 \right)\\
		+ C\left(\| \varphi^1\|_{H^{1,4}(\Gamma(t))}^4 + \| \varphi^2\|_{H^{1,4}(\Gamma(t))}^4\right),
	\end{multline*}
	and
	\[C_{\log}(\bar{\varphi}, \ut^1) = C\log\left( \frac{C_2 (\| \gradg \bar{\varphi} \|_{L^2(\Gamma(t))} +  \|\bar{\varphi}\|_{H^2(\Gamma(t))})}{\|\gradg \bar{\varphi}\|_{L^2(\Gamma(t))}^2} \right)^\frac{1}{2} \| \gradg \bar{\varphi} \|_{L^2(\Gamma(t))}\|\mbb{E}(\ut^1)\|_{\mbf{L}^2(\Gamma(t))}. \]
	
	Likewise,	$\bar{\varphi}$ is such that
	\begin{multline}
			\frac{1}{2} \frac{d}{dt} \| \bar{\varphi} \|_{-1}^2  + \varepsilon \| \gradg \bar{\varphi} \|_{L^2(\Gamma(t))}^2 \leq \frac{\varepsilon}{4} \| \gradg \bar{\varphi} \|_{L^2(\Gamma(t))}^2 + \frac{\eta_*}{2} \| \bar{\ut}\|_{\mbf{L}^2(\Gamma(t))}^2+ K_2(t) \| \bar{\varphi} \|_{-1}^2 , \label{uniquenessphi}
	\end{multline}
	where 
	
	\[K_2(t) = C \left(1 + \|\widetilde{\ut}\|_{\mbf{L}^\infty(\Gamma(t))}^2 + \| \ut^1 \|_{\mbf{L}^4(\Gamma(t))}^2 + \|\varphi^2\|_{L^\infty(\Gamma(t))}^2 \right).\]
\end{lemma}
		
\begin{proof}[Proof of uniqueness]
With these bounds we are now in a position to show uniqueness.
	Taking the sum of \eqref{uniquenessu} \eqref{uniquenessphi} one finds
	\begin{multline*}
		\frac{1}{2} \frac{d}{dt} \left(\| \bar{\ut}\|_{\mathcal{S}}^2 + \| \bar{\varphi} \|_{-1}^2 \right) + \frac{\varepsilon}{2} \| \gradg \bar{\varphi} \|_{L^2(\Gamma(t))}^2
		\leq K(t)\left(\| \bar{\ut}\|_{\mathcal{S}}^2 + \| \bar{\varphi} \|_{-1}^2 \right)\\
		+ C_{\log}(\bar{\varphi}, \ut^1)\left(\| \bar{\ut}\|_{\mathcal{S}}^2 + \| \bar{\varphi} \|_{-1}^2 \right)^{\frac{1}{2}},
	\end{multline*}
	where
	\[ K(t) = K_1(t) + K_2(t). \]
	Now recall that we have $\ut^i \in L^\infty_{\mbf{L}^2} \cap L^2_{\mbf{H}^1}, \varphi^i \in L^\infty_{H^1}\cap L^2_{H^2}$, and in particular this implies $\varphi^i \in L^4_{H^{1,4}}$ as
	\begin{align*}
		\int_0^T \|\varphi^i\|_{H^{1,4}(\Gamma(t))}^4 &\leq C \int_0^T \|\varphi^i\|_{H^{1}(\Gamma(t))}^2 \|\varphi^i\|_{H^2(\Gamma(t))}^2\\
		&\leq C \left(\sup_{t \in [0,T]} \| \varphi^i\|_{H^1(\Gamma(t))}^2\right)\int_0^T \|\varphi^i\|_{H^2(\Gamma(t))}^2 < \infty.
	\end{align*}
	Moreover as $\| \bar{\ut}(0)\|_{\mathcal{S}}^2 + \| \bar{\varphi}(0) \|_{-1}^2 = 0$ by definition, we see that we may use Lemma \ref{gronwalltype} to see that $ \| \bar{\ut}\|_{\mathcal{S}}$ and $\| \bar{\varphi} \|_{-1}$ vanish on $[0,T]$.
	Hence it follows that $\bar{\ut}, \bar{\varphi}$ vanish for almost all $t$, and from this one can readily show that $\bar{\mu} = 0$ a.e. on $[0,T]$ and hence determine uniqueness of weak solutions.
\end{proof}
		
\subsubsection{Proof of Lemma \ref{uniqueineq}}
\begin{proof}
	We begin by showing \eqref{uniquenessu}.
	Testing \eqref{uniqueness1} with $\mathcal{S} \bar{\ut}$, as defined in Appendix \ref{inversestokes}, and rewriting terms in a suitable way we find that
	\begin{multline}
		\mbf{m}_*\left(\normdev {\bar{\ut}}, \mathcal{S} \bar{\ut} \right) + \hat{\mbf{a}}(\eta(\varphi^1) - \eta(\varphi^2), \ut^1, \mathcal{S} \bar{\ut}) + \hat{\mbf{a}}(\eta(\varphi^2), \bar{\ut}, \mathcal{S} \bar{\ut})\\
		+ \mbf{c}_1(\bar{\ut},\ut^1, \mathcal{S} \bar{\ut}) + \mbf{c}_1(\ut^2,\bar{\ut}, \mathcal{S} \bar{\ut})+ \mbf{l}(\bar{\ut}, \mathcal{S} \bar{\ut}) + \mbf{d}_1(\bar{\ut}, \mathcal{S} \bar{\ut})\\
		+\mbf{d}_2(\eta(\varphi^1) - \eta(\varphi^2), \mathcal{S} \bar{\ut})
		=  \varepsilon \mbf{c}_3(\bar{\varphi},\varphi^1,\mathcal{S} \bar{\ut}) + \varepsilon \mbf{c}_3(\varphi^2,\bar{\varphi},\mathcal{S} \bar{\ut}). \label{uniqueness4}
	\end{multline}
	Firstly, we claim that
	\begin{multline}
		\mbf{m}_*\left(\normdev {\bar{\ut}}, \mathcal{S} \bar{\ut} \right) = \frac{1}{2} \frac{d}{dt} \| \bar{\ut} \|_{\mathcal{S}}^2 + \frac{1}{2} \mbf{m}(\mathcal{S}\bar{\ut}, \mathcal{S}\bar{\ut}) + \frac{1}{2}\mbf{b}(\mathcal{S}\bar{\ut},\mathcal{S}\bar{\ut}) - \mbf{m}(\bar{\ut} HV_N, \mathcal{S}\bar{\ut}), \label{uniqueness5}
	\end{multline}
	where we are using the notation from Appendix \ref{inversestokes}.
	To see this we write
	\[ \mbf{m}_*\left(\normdev {\bar{\ut}}, \mathcal{S} \bar{\ut} \right) = \frac{d}{dt} \mbf{m}(\bar{\ut}, \mathcal{S} \bar{\ut}) - \mbf{m}\left(\bar{\ut}, \normdev{ \mathcal{S} \bar{\ut}}\right) - \mbf{m}\left(\bar{\ut}HV_N, \mathcal{S} \bar{\ut}\right).\]
	By the definition of $\mathcal{S}$ we have
	\[ \frac{d}{dt} \mbf{m}(\bar{\ut}, \mathcal{S} \bar{\ut}) = \frac{d}{dt}\|\bar{\ut}\|_{\mathcal{S}}^2, \]
	and it remains to rewrite $\mbf{m}\left(\bar{\ut}, \normdev \mathcal{S} \bar{\ut}\right)$.
	For this we see that
	\begin{align*}
		\mbf{m}\left(\bar{\ut}, \normdev \mathcal{S} \bar{\ut}\right) &= \mbf{m}\left(\mathcal{S}\bar{\ut}, \normdev \mathcal{S} \bar{\ut}\right) +\mbf{a}\left(\mathcal{S}\bar{\ut}, \normdev \mathcal{S} \bar{\ut}\right)\\
		&= \frac{1}{2} \frac{d}{dt} \|\bar{\ut}\|_{\mathcal{S}}^2 - \frac{1}{2} \mbf{m}(\mathcal{S}\bar{\ut}, \mathcal{S}\bar{\ut} H V_N) - \frac{1}{2} \mbf{b}(\mathcal{S}\bar{\ut},\mathcal{S}\bar{\ut}),
	\end{align*}
	from which we see the claim holds.\\
	
	Next we rewrite the second $\hat{\mbf{a}}$ term in \eqref{uniqueness4} by using integration by parts.
	To do this we note that $\mathcal{S}\bar{\ut} \in L^2_{H^2}$ and we have the bound \eqref{stokes regularity}.
	Integration by parts yields
	\begin{multline*}
		\hat{\mbf{a}}(\eta(\varphi^2), \bar{\ut}, \mathcal{S}\bar{\ut}) = \int_{\Gamma(t)} \eta(\varphi^2) \mbb{E}(\bar{\ut}) : \mbb{E}(\mathcal{S}\bar{\ut}) = -\int_{\Gamma(t)} \bar{\ut} \cdot \mbb{P} \gradg  \cdot\left( \eta(\varphi^2) \mbb{E}(\mathcal{S}\bar{\ut}) \right)\\
		= -\int_{\Gamma(t)}\eta'(\varphi^2) \bar{\ut} \cdot  \mbb{E}(\mathcal{S}\bar{\ut}) \gradg \varphi^2 + \int_{\Gamma(t)} \eta(\varphi^2) \bar{\ut} \cdot (\bar{\ut} - \mathcal{S}\bar{\ut}),
	\end{multline*}
	where we have used the fact that $\mathcal{S}\bar{\ut}-\mbb{P} \divg \mbb{E}(\mathcal{S}\bar{\ut}) = \bar{\ut}$ a.e. on $\Gamma(t)$.
	Hence one obtains
	\begin{multline}
		\hat{\mbf{a}}(\eta(\varphi^2), \bar{\ut}, \mathcal{S}\bar{\ut}) = -\int_{\Gamma(t)}\eta'(\varphi^2) \bar{\ut} \cdot  \mbb{E}(\mathcal{S}\bar{\ut}) \gradg \varphi^2 + \int_{\Gamma(t)} \eta(\varphi^2) \bar{\ut} \cdot \bar{\ut}\\
		- \int_{\Gamma(t)} \eta(\varphi^2) \bar{\ut} \cdot \mathcal{S}\bar{\ut}. \label{uniqueness6}
	\end{multline}
	Next we bound the other $\hat{\mbf{a}}$ term.
	To do this we recall that $\eta(\cdot)$ is Lipschitz continuous, and so we obtain the bound
	\[\hat{\mbf{a}}(\eta(\varphi^1) - \eta(\varphi^2), \ut^1, \mathcal{S}\bar{\ut}) \leq C \| \bar{\varphi} \|_{L^\infty(\Gamma(t))} \|\mbb{E}(\ut^1)\|_{\mbf{L}^2(\Gamma(t))} \| \bar{\ut}\|_{\mathcal{S}} . \]
	Now we use Lemma \ref{brezisgallouet}, and Poincar\'e's inequality, to see that
	\[ \| \bar{\varphi}\|_{L^\infty(\Gamma(t))} \leq C \| \gradg \bar{\varphi} \|_{L^2(\Gamma(t))} \left( 1 + \log\left( 1 + \frac{C \|\bar{\varphi}\|_{H^2(\Gamma(t))}}{\|\gradg \bar{\varphi}\|_{L^2(\Gamma(t))}} \right)^\frac{1}{2}\right),\]
	where we note from \eqref{extra regularity} that $\bar{\varphi}$ has sufficient regularity.
	Now from the $L^\infty_{H^1}$ bounds for $\varphi_1, \varphi_2$ we see that there is a constant, $C_1$, such that $\|\gradg \bar{\varphi} \|_{L^2(\Gamma(t))} \leq C_1$ for almost all $t \in [0,T]$.
	Hence one can find a sufficiently large constant $C_2$ so that
	\[ \log\left( 1 + \frac{C \|\bar{\varphi}\|_{H^2(\Gamma(t))}}{\|\gradg \bar{\varphi}\|_{L^2(\Gamma(t))}} \right) \leq \log\left( \frac{C_2 (\| \gradg \bar{\varphi} \|_{L^2(\Gamma)} +  \|\bar{\varphi}\|_{H^2(\Gamma(t))})}{\|\gradg \bar{\varphi}\|_{L^2(\Gamma(t))}^2} \right),\]
	such that this logarithmic term is positive, and we will ultimately be able to apply Lemma \ref{gronwalltype}.
	All in all this gives us the bound
	\begin{multline}
		|\hat{\mbf{a}}(\eta(\varphi^1) - \eta(\varphi^2), \ut^1, \mathcal{S}\bar{\ut})| \leq \frac{\varepsilon}{12}\| \gradg \bar{\varphi} \|_{L^2(\Gamma(t))}^2 + C\|\mbb{E}(\ut^1)\|_{\mbf{L}^2(\Gamma(t))}^2 \| \bar{\ut}\|_{\mathcal{S}}^2\\
		+ C\log\left( \frac{C_2 (\| \gradg \bar{\varphi} \|_{L^2(\Gamma(t))} +  \|\bar{\varphi}\|_{H^2(\Gamma(t))})}{\|\gradg \bar{\varphi}\|_{L^2(\Gamma(t))}^2} \right)^\frac{1}{2} \| \gradg \bar{\varphi} \|_{L^2(\Gamma(t))}\|\mbb{E}(\ut^1)\|_{\mbf{L}^2(\Gamma(t))} \| \bar{\ut}\|_\mathcal{S}, \label{uniqueness7}
	\end{multline}
	where we have used Young's inequality where appropriate.\\
	
	Now we bound the $\mbf{c}_1$ terms.
	Firstly we note that
	\[ |\mbf{c}_1(\bar{\ut}, \ut^1, \mathcal{S}\bar{\ut})| = |\mbf{c}_1(\mathcal{S}\bar{\ut}, \ut^1, \bar{\ut})| \leq \|\mathcal{S}\bar{\ut}\|_{\mbf{H}^{1,4}(\Gamma(t))} \| \ut^1 \|_{\mbf{L}^4(\Gamma(t))} \| \bar{\ut}\|_{\mbf{L}^2(\Gamma(t))} .\]
	We recall the interpolation inequality\footnote{See \cite{Aub82}, and note that a $C^3$ surface is sufficiently smooth for this to hold.
	This can be extended to evolving surfaces with a time independent constant as in \cite{OlsReuZhi22} Lemma 3.4.},
	\[\|\mathcal{S}\bar{\ut}\|_{\mbf{H}^{1,4}(\Gamma(t))} \leq C \|\mathcal{S}\bar{\ut}\|_{\mbf{H}^1(\Gamma(t))}^{\frac{1}{2}} \| \mathcal{S}\bar{\ut} \|_{\mbf{H}^2(\Gamma(t))}^{\frac{1}{2}},\]
	and use \eqref{korn1}, \eqref{stokes regularity} to see that
	\begin{align}
		\|\mathcal{S}\bar{\ut}\|_{\mbf{H}^{1,4}(\Gamma(t))} \leq C \| \bar{\ut}\|_{\mathcal{S}}^{\frac{1}{2}} \| \bar{\ut} \|_{\mbf{L}^2(\Gamma(t))}^{\frac{1}{2}}. \label{interpolation inequality}
	\end{align}
	Hence we observe that
	\[ |\mbf{c}_1(\bar{\ut}, \ut^1, \mathcal{S}\bar{\ut})| \leq C \| \ut^1 \|_{\mbf{L}^4(\Gamma(t))} \| \bar{\ut}\|_{\mathcal{S}}^{\frac{1}{2}} \| \bar{\ut} \|_{\mbf{L}^2(\Gamma(t))}^{\frac{3}{2}},  \]
	and hence Young's inequality yields
	\begin{align}
		|\mbf{c}_1(\bar{\ut}, \ut^1, \mathcal{S}\bar{\ut})| \leq C \| \ut^1 \|_{\mbf{L}^4(\Gamma(t))}^4 \| \bar{\ut}\|_{\mathcal{S}}^2 + \frac{\eta_*}{12}\| \bar{\ut} \|_{\mbf{L}^2(\Gamma(t))}^2.\label{uniqueness8}
	\end{align}
	An identical argument yields 
	\begin{align}
		|\mbf{c}_1(\ut^2,\bar{\ut}, \mathcal{S}\bar{\ut})| \leq C \| \ut^2 \|_{\mbf{L}^4(\Gamma(t))}^4 \|\bar{\ut}\|_{\mathcal{S}}^2 + \frac{\eta_*}{12}\| \bar{\ut} \|_{\mbf{L}^2(\Gamma(t))}^2.\label{uniqueness9}
	\end{align}
	We now turn to the contributions from the evolution of the surface, that is the terms involving $\mbf{l}, \mbf{d}_1, \mbf{d}_2$, which would vanish for a stationary surface.
	The simplest of these terms is
	\begin{align}
		\mbf{l}(\bar{\ut}, \mathcal{S}\bar{\ut}) \leq C \| \bar{\ut}\|_{\mbf{L}^2(\Gamma(t))}\| \mathcal{S}\bar{\ut}\|_{\mbf{L}^2(\Gamma(t))}\leq \frac{\eta_*}{12}\|\bar{\ut}\|_{\mbf{L}^2(\Gamma(t))}^2 + C \| \bar{\ut} \|_{\mathcal{S}}^2,
		\label{uniqueness10}
	\end{align}
	which follows from Young's inequality.
	Next we look at
	\begin{align}
		\mbf{d}_2(\eta(\varphi^1) - \eta(\varphi^2), \mathcal{S}\bar{\ut}) \leq C \|\bar{\varphi}\|_{L^2(\Gamma(t))} \| \bar{\ut}\|_{\mathcal{S}} \leq \frac{\varepsilon}{12} \|\gradg \bar{\varphi} \|_{L^2(\Gamma(t))}^2 +C \| \bar{\ut} \|_{\mathcal{S}}^2, \label{uniqueness11}
	\end{align}
	which follows similarly to the above inequality, but we have also used the Lipschitz continuity of $\eta(\cdot)$, Poincar\'e's inequality, and the uniform bounds on $\widetilde{\ut}$.
	Finally to bound the $\mbf{d}_1$ term we see
	\begin{align*}
		\mbf{d}_1(\bar{\ut}, \mathcal{S}\bar{\ut}) &= \mbf{c}_1(\bar{\ut}, \widetilde{\ut}, \mathcal{S}\bar{\ut}) + \mbf{c}_1(\widetilde{\ut},\bar{\ut}, \mathcal{S}\bar{\ut})\\
		&=-\mbf{c}_1(\mathcal{S}\bar{\ut}, \widetilde{\ut}, \bar{\ut}) - \mbf{c}_1(\mathcal{S}\bar{\ut},\bar{\ut}, \widetilde{\ut}) + 2\mbf{m}(\bar{\ut} HV_N, \mathcal{S}\bar{\ut}),
	\end{align*}
	where we have used the antisymmetry of $\mbf{c}_1$ (for solenoidal functions), and the extra term comes from the fact that $\divg \widetilde{\ut} = -HV_N$.
	From this one readily sees that
	\[ |\mbf{d}_1(\bar{\ut}, \mathcal{S}\bar{\ut})| \leq C\| \bar{\ut}\|_{\mathcal{S}} \|\bar{\ut}\|_{\mbf{L}^2(\Gamma(t))} \|\widetilde{\ut}\|_{\mbf{L}^\infty(\Gamma(t))} + C \|\bar{\ut}\|_{\mbf{L}^2(\Gamma(t))} \| \bar{\ut} \|_{\mathcal{S}}. \]
	This clearly yields
	\begin{align}
		|\mbf{d}_1(\bar{\ut}, \mathcal{S}\bar{\ut})| \leq C\left(1 + \|\widetilde{\ut}\|_{\mbf{L}^\infty(\Gamma(t))}^2\right) \| \bar{\ut}\|_{\mathcal{S}}^2 + \frac{\eta_*}{12} \|\bar{\ut}\|_{\mbf{L}^2(\Gamma(t))}^2. \label{uniqueness12} 
	\end{align}
	
	The last terms for us to bound are the $\mbf{c}_3$ contributions, from which one readily sees that Young's inequality gives us
	\begin{multline*}
		\varepsilon |\mbf{c}_3(\bar{\varphi},\varphi^1,\mathcal{S} \bar{\ut})| + \varepsilon |\mbf{c}_3(\varphi^2,\bar{\varphi},\mathcal{S} \bar{\ut})| \leq \frac{\varepsilon}{12} \| \gradg \bar{\varphi} \|_{L^2(\Gamma(t))}^2\\
		+ C\left(\| \varphi^1\|_{H^{1,4}(\Gamma(t))}^2 + \| \varphi^2\|_{H^{1,4}(\Gamma(t))}^2\right)\|\mathcal{S} \bar{\ut}\|_{\mbf{H}^{1,4}(\Gamma(t))}^2.
	\end{multline*}
	We then recall \eqref{interpolation inequality} to see that the above yields
	\begin{multline}
		\varepsilon |\mbf{c}_3(\bar{\varphi},\varphi^1,\mathcal{S} \bar{\ut})| + \varepsilon |\mbf{c}_3(\varphi^2,\bar{\varphi},\mathcal{S} \bar{\ut})| \leq \frac{\varepsilon}{12} \| \gradg \bar{\varphi} \|_{L^2(\Gamma(t))}^2 + \frac{\eta_*}{12} \|\bar{\ut}\|_{\mbf{L}^2(\Gamma(t))}^2\\
		+ C\left(\| \varphi^1\|_{H^{1,4}(\Gamma(t))}^4 + \| \varphi^2\|_{H^{1,4}(\Gamma(t))}^4\right)\| \bar{\ut}\|_{\mathcal{S}}^2,
		\label{uniqueness13}
	\end{multline}
	where we note that $\varphi^i \in L^4_{H^{1,4}}$ as shown above.\\
	
	Now we use \eqref{uniqueness5}-\eqref{uniqueness13} in \eqref{uniqueness4} to see that
	\begin{multline*}
		\frac{1}{2} \frac{d}{dt} \| \bar{\ut} \|_{\mathcal{S}}^2 + \eta_* \|\bar{\ut} \|_{\mbf{L}^2(\Gamma(t))}^2 \leq \frac{\varepsilon}{4}\| \gradg \bar{\varphi} \|_{L^2(\Gamma(t))}^2 + \frac{5\eta_*}{12} \|\bar{\ut}\|_{\mbf{L}^2(\Gamma(t))}^2+\frac{1}{2} \mbf{m}(\mathcal{S}\bar{\ut}, \mathcal{S}\bar{\ut})\\
		+ \frac{1}{2}|\mbf{b}(\mathcal{S}\bar{\ut},\mathcal{S}\bar{\ut})| + |\mbf{m}(\bar{\ut} HV_N, \mathcal{S}\bar{\ut})| + \left|\int_{\Gamma(t)}\eta'(\varphi^2) \bar{\ut} \cdot  \mbb{E}(\mathcal{S}\bar{\ut}) \gradg \varphi^2\right|\\
		+ \left| \int_{\Gamma(t)} \eta(\varphi^2) \bar{\ut} \cdot \mathcal{S}\bar{\ut} \right| + K_1(t)\| \bar{\ut}\|_{\mathcal{S}}^2+ C_{\log}(\bar{\varphi}, \ut^1) \| \bar{\ut}\|_\mathcal{S}.
	\end{multline*}
	\eqref{uniquenessu} then follows by noting the bound
	\begin{multline*}
		\frac{1}{2} |\mbf{m}(\mathcal{S}\bar{\ut}, \mathcal{S}\bar{\ut} HV_N)| + \frac{1}{2}|\mbf{b}(\mathcal{S}\bar{\ut},\mathcal{S}\bar{\ut})| + |\mbf{m}(\bar{\ut} HV_N, \mathcal{S}\bar{\ut})|\\
        + \left|\int_{\Gamma(t)}\eta'(\varphi^2) \bar{\ut} \cdot  \mbb{E}(\mathcal{S}\bar{\ut}) \gradg \varphi^2\right|
		+ \left| \int_{\Gamma(t)} \eta(\varphi^2) \bar{\ut} \cdot \mathcal{S}\bar{\ut} \right|\\
  \leq \frac{\eta_*}{12}\| \bar{\ut}\|_{\mbf{L}^2(\Gamma(t))}^2
  + C(1 + \|\varphi^2\|_{H^{1,4}(\Gamma(t))}^4) \| \bar{\ut}\|_{\mathcal{S}}^2,
	\end{multline*}
	where we have used \eqref{interpolation inequality} so that
	\[ \left|\int_{\Gamma(t)}\eta'(\varphi^2) \bar{\ut} \cdot  \mbb{E}(\mathcal{S}\bar{\ut}) \gradg \varphi^2\right| \leq C \|\varphi^2\|_{H^{1,4}(\Gamma(t))} \|\bar{\ut}\|_{\mbf{L}^2(\Gamma(t))}^{\frac{3}{2}}\|\bar{\ut}\|_{\mathcal{S}}^{\frac{1}{2}}. \]

	It remains to establish \eqref{uniquenessphi}.
	To do this we test \eqref{uniqueness2} with $\mathcal{G} \bar{\varphi}$, which we note is well defined as $\mval{\bar{\varphi}}{\Gamma(t)} = 0$ for almost all $t \in [0,T]$, which yields
	\begin{align}
		m_*\left(\normdev {\bar{\varphi}}, \mathcal{G} \bar{\varphi}\right) + a(\bar{\mu}, \mathcal{G} \bar{\varphi}) + \mbf{c}_2(\mathcal{G} \bar{\varphi},\bar{\varphi},\ut^1) + \mbf{c}_2(\mathcal{G} \bar{\varphi},\varphi^2,\bar{\ut}) + \mbf{c}_2(\mathcal{G} \bar{\varphi},\bar{\varphi},\widetilde{\ut})= 0. \label{uniqueness20}
	\end{align}
	The first term can be expressed as
	\begin{align}
		m_*\left( \normdev{\bar{\varphi}}, \mathcal{G} \bar{\varphi} \right) = \frac{1}{2} \frac{d}{dt} \| \bar{\varphi} \|_{-1}^2 - m(\bar{\varphi} HV_N, \mathcal{G} \bar{\varphi}) + \frac{1}{2} b(\mathcal{G} \bar{\varphi}, \mathcal{G} \bar{\varphi}). \label{uniqueness21}
	\end{align}
	To see this we express this term as
	\[ m_*\left( \normdev{\bar{\varphi}}, \mathcal{G} \bar{\varphi} \right) = \frac{d}{dt} m\left( \bar{\varphi}, \mathcal{G} \bar{\varphi} \right) - m\left( \bar{\varphi}, \normdev{\mathcal{G} \bar{\varphi}} \right) - m(\bar{\varphi} HV_N, \mathcal{G} \bar{\varphi}),\]
	and note $m(\bar{\varphi}, \mathcal{G} \bar{\varphi}) = \| \bar{\varphi}\|_{-1}^2$, and that
	\[ m\left( \bar{\varphi}, \normdev{\mathcal{G} \bar{\varphi}} \right) = a\left(\mathcal{G} \bar{\varphi}, \normdev{\mathcal{G} \bar{\varphi}} \right) = \frac{1}{2}\frac{d}{dt} \|\bar{\varphi}\|_{-1}^2 - \frac{1}{2} b(\mathcal{G} \bar{\varphi}, \mathcal{G} \bar{\varphi}). \]
	
	To bound the second term, we see from the definition of the inverse Laplacian that $a(\bar{\mu}, \mathcal{G} \bar{\varphi}) = m(\bar{\mu}, \bar{\varphi})$, and hence testing \eqref{uniqueness2} with $\bar{\varphi}$ we see
	\[a(\bar{\mu}, \mathcal{G} \bar{\varphi}) = \varepsilon a(\bar{\varphi}, \bar{\varphi}) + \frac{1}{\varepsilon}m(F'(\varphi^1)-F'(\varphi^2), \bar{\varphi}).\]
	We recall that $F = F_1 + F_2$ where $F_1$ is convex, so that
	\[m(F'(\varphi^1)-F'(\varphi^2), \bar{\varphi}) \geq m(F_2'(\varphi^1)-F_2'(\varphi^2), \bar{\varphi}).\]
	By using the Lipschitz continuity of $F_2'$, and the definition of the inverse Laplacian, one readily sees that
	\begin{align}
		|m(F_2'(\varphi^1)-F_2'(\varphi^2), \bar{\varphi})| \leq C \| \bar{\varphi}\|_{L^2(\Gamma(t))}^2 = C a(\bar{\varphi}, \mathcal{G}\bar{\varphi}) \leq \frac{\varepsilon}{16} \| \gradg \bar{\varphi} \|_{L^2(\Gamma(t))} + C \| \bar{\varphi} \|_{-1}^2. \label{uniqueness22}
	\end{align}
	It remains to bound the various $\mbf{c}_2$ terms.
	Firstly we find that
	\[|\mbf{c}_2(\mathcal{G} \bar{\varphi}, \bar{\varphi}, \ut^1)| \leq \| \mathcal{G} \bar{\varphi}\|_{L^4(\Gamma(t))} \|\gradg\bar{\varphi}\|_{L^2(\Gamma(t))} \| \ut^1 \|_{\mbf{L}^4(\Gamma(t))},\]
	and by using the embedding $H^1(\Gamma(t)) \hookrightarrow L^4(\Gamma(t))$, Poincar\'e's inequality and Young's inequality we find that
	\begin{align}
		|\mbf{c}_2(\mathcal{G} \bar{\varphi}, \bar{\varphi}, \ut^1)| \leq \frac{\varepsilon}{16} \| \gradg \bar{\varphi} \|_{L^2(\Gamma(t))}^2 + C  \| \ut^1 \|_{\mbf{L}^4(\Gamma(t))}^2 \| \bar{\varphi} \|_{-1}^2. \label{uniqueness23}
	\end{align}
	The other $\mbf{c}_2$ term is similar, but now we use the antisymmetry of $\mbf{c}_2$ in the first two arguments so that
	\[ |\mbf{c}_2(\mathcal{G} \bar{\varphi}, \varphi^2, \bar{\ut})| = |\mbf{c}_2( \varphi^2, \mathcal{G} \bar{\varphi},\bar{\ut})| \leq \| \varphi^2 \|_{L^\infty(\Gamma(t))} \| \bar{\varphi} \|_{-1} \| \bar{\ut} \|_{\mbf{L}^2(\Gamma(t))}.\]
	It is then clear from Young's inequality that
	\begin{align}
		|\mbf{c}_2(\mathcal{G} \bar{\varphi}, \varphi^2, \bar{\ut})| \leq \frac{\eta_*}{4} \| \bar{\ut} \|_{\mbf{L}^2(\Gamma(t))}^2 + C\|\varphi^2\|_{L^\infty(\Gamma(t))}^2 \| \bar{\varphi} \|_{-1}^2. \label{uniqueness24}
	\end{align}
	Finally we bound the term involving $\widetilde{\ut}$.
	To do this, we observe that
	\[ |\mbf{c}_2(\mathcal{G} \bar{\varphi},\bar{\varphi},\widetilde{\ut})| \leq |\mbf{c}_2( \bar{\varphi},\mathcal{G}\bar{\varphi},\widetilde{\ut})| + |m(\bar{\varphi} HV_N, \mathcal{G} \bar{\varphi})|, \]
	so that by similar arguments to the above
	\begin{align}
		|\mbf{c}_2(\mathcal{G} \bar{\varphi},\bar{\varphi},\widetilde{\ut})| \leq \frac{\varepsilon}{16} \| \gradg \bar{\varphi} \|_{L^2(\Gamma(t))}^2 + C \left(1 + \|\widetilde{\ut}\|_{\mbf{L}^\infty(\Gamma(t))}^2\right) \| \bar{\varphi} \|_{-1}^2. \label{uniqueness25}
	\end{align}
	Hence using \eqref{uniqueness21}-\eqref{uniqueness25} in \eqref{uniqueness20} one obtains
	\begin{multline*}
		\frac{1}{2} \frac{d}{dt} \| \bar{\varphi} \|_{-1}^2 \leq |m(\bar{\varphi} HV_N, \mathcal{G} \bar{\varphi})| + \frac{1}{2} |b(\mathcal{G} \bar{\varphi}, \mathcal{G} \bar{\varphi})| +  \frac{3\varepsilon}{16} \| \gradg \bar{\varphi} \|_{L^2(\Gamma(t))} + \frac{\eta_*}{4} \| \bar{\ut} \|_{\mbf{L}^2(\Gamma(t))}^2 \\ 
		+ C \left(1+ \| \ut^1 \|_{\mbf{L}^4(\Gamma(t))}^2 +  \|\varphi^2\|_{L^\infty(\Gamma(t))}^2 + \|\widetilde{\ut}\|_{\mbf{L}^\infty(\Gamma(t))}^2\right)\| \bar{\varphi} \|_{-1}^2.
	\end{multline*}
    \eqref{uniquenessphi} then follows from the bound
	\[ |m(\bar{\varphi} HV_N, \mathcal{G} \bar{\varphi})| + |b(\mathcal{G} \bar{\varphi}, \mathcal{G} \bar{\varphi})| \leq \frac{\varepsilon}{16}\| \gradg \bar{\varphi} \|_{L^2(\Gamma(t))}^2 + C \|\bar{\varphi}\|_{-1}^2. \]
\end{proof}

Next we show a stability result for the case of constant viscosity.
This also provides a simpler proof for uniqueness in this special case, where we no longer require Lemma \ref{brezisgallouet}.
\begin{proposition}
	\label{smooth stability}
	Let $(\varphi^i, \mu^i, \ut^i)$ denote the solution triple corresponding to some choice of initial data $\varphi_0^i \in H^1(\Gamma_0), \mbf{u}_{T,0}^i \in \mbf{H}_\sigma$, for $i=1,2$, where $\mval{\varphi_0^1}{\Gamma_0} = \mval{\varphi_2^1}{\Gamma_0}$.
	Then, under the same assumptions as the preceding theorem, we have
	\begin{multline}
		\| \ut^1(t) - \ut^2(t)\|_{\mathcal{S}}^2 + \| \varphi^1(t) - \varphi^2(t) \|_{-1}^2 \leq C\left(\| \mbf{u}_{T,0}^1 - \mbf{u}_{T,0}^2\|_{\mathcal{S}}^2 + \| \varphi_0^1 - \varphi_0^2 \|_{-1}^2\right), \label{stability bound}
	\end{multline}
	for a constant $C$ which depends on $t, \Gamma$ and the initial data.
\end{proposition}
\begin{proof}
	This proof is largely the same as that of the previous theorem, but now there are simpler terms regarding the viscosity.
	We use the same notation as before, except now we denote $(\varphi^i, \mu^i, \ut^i)$ as the solution corresponding to some choice of initial data $\varphi_0^i \in H^1(\Gamma), \mbf{u}_{T,0}^i \in \mbf{H}_\sigma$.
	We define $(\bar{\varphi}, \bar{\mu}, \bar{\ut})$ as before, and note that instead of \eqref{uniqueness1} we find that $(\bar{\varphi}, \bar{\mu}, \bar{\ut})$ solves
	\begin{multline}
		\mbf{m}_*\left(\normdev {\bar{\ut}}, \boldsymbol{\phi} \right) + \eta\mbf{a}(\bar{\ut}, \boldsymbol{\phi})+ \mbf{c}_1(\ut^1,\ut^1, \boldsymbol{\phi}) - \mbf{c}_1(\ut^2,\ut^2, \boldsymbol{\phi}) + \mbf{l}(\bar{\ut}, \boldsymbol{\phi}) + \mbf{d}_1(\bar{\ut}, \boldsymbol{\phi})\\
		=  \varepsilon \mbf{c}_3(\varphi^1,\varphi^1,\boldsymbol{\phi}) - \varepsilon \mbf{c}_3(\varphi^2,\varphi^2,\boldsymbol{\phi}),\label{stability1}
	\end{multline}
	but \eqref{uniqueness2} and \eqref{uniqueness3} are still satisfied.
	As before we test \eqref{stability1} with $\mathcal{S} \bar{\ut}$ to see
	\begin{multline*}
		\frac{1}{2} \frac{d}{dt} \| \bar{\ut} \|_{\mathcal{S}}^2 + \eta \| \bar{\ut} \|_{\mbf{L}^2(\Gamma(t))}^2 + \mbf{c}_1(\bar{\ut},\ut^1, \mathcal{S} \bar{\ut}) + \mbf{c}_1(\ut^2,\bar{\ut}, \mathcal{S} \bar{\ut}) + \mbf{l}(\bar{\ut}, \mathcal{S} \bar{\ut})\\
		+ \mbf{d}_1(\bar{\ut}, \mathcal{S} \bar{\ut})
		=  \varepsilon \mbf{c}_3(\bar{\varphi},\varphi^1,\mathcal{S} \bar{\ut}) + \varepsilon \mbf{c}_3(\varphi^2,\bar{\varphi},\mathcal{S} \bar{\ut}) - \frac{1}{2} \mbf{m}(\mathcal{S}\bar{\ut}, \mathcal{S}\bar{\ut})\\
		- \frac{1}{2}\mbf{b}(\mathcal{S}\bar{\ut},\mathcal{S}\bar{\ut}) + \mbf{m}(\bar{\ut} HV_N, \mathcal{S}\bar{\ut}) + \eta \mbf{m}(\bar{\ut}, \mathcal{S} \bar{\ut})
	\end{multline*}
	where we have used \eqref{uniqueness5} and the definition of $\mathcal{S}$.
	Now by arguing as we did for \eqref{uniqueness8}-\eqref{uniqueness13} it is straightforward to see that
	\begin{align}
		\frac{1}{2} \frac{d}{dt} \| \bar{\ut}\|_{\mathcal{S}}^2 + \eta \| \bar{\ut} \|_{\mbf{L}^2(\Gamma(t))}^2 \leq \frac{\eta}{2}\| \bar{\ut} \|_{\mbf{L}^2(\Gamma(t))}^2 + \frac{\varepsilon}{4} \| \gradg \bar{\varphi} \|_{L^2(\Gamma(t))}^2 + K_1(t)\| \bar{\ut}\|_{\mathcal{S}}^2, \label{stability2}
	\end{align}
	where $K_1$ is as before.
	Notice that the requirement $\mval{\varphi_0^1}{\Gamma} = \mval{\varphi_0^2}{\Gamma}$ allows us to define $\mathcal{G} \bar{\varphi}$, and so related calculations from the preceding theorem still hold.
	By summing \eqref{uniquenessphi} and \eqref{stability2} one finds
	\[\frac{1}{2} \frac{d}{dt} \left(\| \bar{\ut}\|_{\mathcal{S}}^2 + \| \bar{\varphi} \|_{-1}^2 \right) \leq K(t) \left(\| \bar{\ut}\|_{\mathcal{S}}^2 + \| \bar{\varphi} \|_{-1}^2 \right),\]
	where $K \in L^1([0,T])$.
	An application of Gr\"onwall's inequality then yields \eqref{stability bound}.
\end{proof}

\begin{remark}
	\label{stability remark}
	One can improve this result on a stationary surface as follows.
	Firstly we note that depending on the surface, $\Gamma$, there may be a nontrivial, finite-dimensional kernel
	\[ \mathcal{K} := \left\{ \boldsymbol{\phi} \in \mbf{V}_\sigma \mid \mbb{E}(\boldsymbol{\phi}) = 0 \right\}, \]
	consisting of the Killing vectors of $\Gamma$.
	We then define the subspace $\mathcal{K}^\perp$ such that $\mbf{H}_\sigma = \mathcal{K} \oplus \mathcal{K}^\perp$, and for $\boldsymbol{\phi} \in \mathcal{K}^\perp$ we define $\Sperp \boldsymbol{\phi} \in \mathcal{K}^\perp \cap \mbf{V}_\sigma$ to be the unique solution of
	\[ \mbf{a}(\Sperp \boldsymbol{\phi}, \boldsymbol{\psi}) = \mbf{m}(\boldsymbol{\phi}, \boldsymbol{\psi}), \]
	for all $\boldsymbol{\psi} \in \mbf{V}_\sigma$.
	It is shown in \cite{JanOlsReu18} that this is well-defined.\\
	
	One can then decompose $\bar{\ut} = \PK \bar{\ut} + \PKperp \bar{\ut}$, where $ \PK \bar{\ut} \in \mathcal{K}, \PKperp \bar{\ut} \in \mathcal{K}^\perp$, for which one can now show stability for $ \PK \bar{\ut}$ in $\mbf{L}^2(\Gamma)$, and $ \PKperp \bar{\ut}$ in $\mbf{V}_\sigma'$ - we omit the calculations here.
	An important note is that the dimension of $\mathcal{K}$ is not a topological invariant, and so under arbitrary (but area conserving) normal evolution the dimension of $\mathcal{K}$ can vary.
	In particular, this means there is not necessarily an isomorphism $\mathcal{K}(t) \rightarrow \mathcal{K}(s)$ for $t, s \in [0,T]$.
	Hence this argument does not hold for evolving surfaces without some extra assumptions based on the dimension of $\mathcal{K}(t)$ - which can be understood as geometric constraint based on the symmetries of $\Gamma(t)$.
\end{remark}

\subsection{Uniqueness for the logarithmic potential}
It is clear that \eqref{extra regularity} still holds, that is $\varphi \in L^2_{H^2}$, since we know $\mu, f(\varphi) \in L^2_{L^2},$ and so we may use the $\mbf{c}_3$ bilinear form as before.
With this at hand, the proofs of Theorem \ref{smooth uniqueness} and Proposition \ref{smooth stability} follow.
That is, for $\Gamma(t)$ a $C^3$ evolving surface, and initial data $\varphi_0 \in \mathcal{I}_0, \mbf{u}_{T,0} \in \Hdivfree{0}$ the solution triple $(\varphi,\mu,\ut)$ solving \eqref{logweakTNSCH1}-\eqref{logweakTNSCH3} is unique.
Moreover, for a constant viscosity, and $\varphi_0^i \in \mathcal{I}_0$ such that where $\mval{\varphi_0^1}{\Gamma_0} = \mval{\varphi_0^2}{\Gamma_0}$, then we have stability bound similar to \eqref{stability bound}.
We note that our proof does not require $\gradg \varphi \in L^2_{L^\infty}$, as is assumed in \cite{giorgini2019uniqueness}.

\begin{remark}
    The results for the logarithmic potential extend to more general singular potentials of the form $F(r) = F_1(r) - \frac{\theta}{2} r^2$, with $F_1 \in C^2((a,b)) \cap C^0([a,b])$ for some $a,b \in \mbb{R}$ under some necessary assumptions we do not expand upon.
    Potentials of this form are treated on a Euclidean domain in \cite{abels2009diffuse,giorgini2019uniqueness}, but here we have only covered the thermodynamically relevant logarithmic potential - which is still illustrative of the general case.
\end{remark}

\section{Reintroducing the surface pressure}\label{Pressure}
We end our discussion by reintroducing the surface pressure and the correct divergence condition.
We now consider the mixed formulation, with a regular potential, where one finds a solution $(\varphi,\mu,\ut, p)$, with $\varphi \in H^{1}_{H^{-1}} \cap L^2_{H^1}, \mu \in L^2_{H^1}, \ut \in H^1_{\mbf{H}^{-1}} \cap L^2_{\mbf{H}^1}, p \in L^2_{L^2}$, solving
\begin{gather}
	\begin{split}
		\langle \normdev \ut, \boldsymbol{\phi}\rangle_{\mbf{H}^{-1}(\Gamma(t)), \mbf{H}^1(\Gamma(t))} + \hat{\mbf{a}}(\eta(\varphi), \ut, \boldsymbol{\phi}) + \mbf{c}_1(\ut,\ut, \boldsymbol{\phi}) + \mbf{l}(\ut, \boldsymbol{\phi})+ \mbf{d}_1(\ut, \boldsymbol{\phi})\\
		+\mbf{d}_2(\eta(\varphi), \boldsymbol{\phi}) = m(p, \divg \boldsymbol{\phi}) + \mbf{m}(\mbf{B}, \boldsymbol{\phi}) +  \mbf{c}_2(\mu,\varphi,\boldsymbol{\phi})
	\end{split},\label{mixedTNSCH1}\\
	m(q,\divg \ut) = 0,\label{mixedTNSCH2}\\
	m_*(\normdev \varphi, \phi) + a(\mu, \phi) + \mbf{c}_2(\phi,\varphi,\ut) + \mbf{c}_2(\phi,\varphi, \widetilde{\ut}) = 0 ,\label{mixedTNSCH3}\\
	m(\mu,\phi) = \varepsilon a(\varphi, \phi) + \frac{1}{\varepsilon} m(F'(\varphi), \phi),\label{mixedTNSCH4}
\end{gather}
for all $q \in L^2(\Gamma(t)), \phi \in H^1(\Gamma(t)), \boldsymbol{\phi} \in \mbf{H}^1(\Gamma(t))$ for almost all $t \in [0,T]$.
Here the initial data is $\varphi_0 \in H^1(\Gamma_0), \mbf{u}_{T,0} \in \Hdivfree{0}$, so one has $\varphi(0) = \varphi_0, \ut(0) = \mbf{u}_{T,0}$ almost everywhere on $\Gamma_0$.\\

Before proving the existence and uniqueness of this system, we recall the uniform inf-sup condition of \cite{OlsReuZhi22}.
\begin{lemma}[\cite{OlsReuZhi22}, Lemma 3.3]
	There exists a constant, $C$, independent of time such that for all $q \in L_0^2(\Gamma(t)):=\left\{ \phi \in L^2(\Gamma(t)) \mid \mval{\phi}{\Gamma(t)} = 0 \right\}$,
	\begin{align}
		\| \gradg q \|_{\mbf{H}^{-1}(\Gamma(t))} := \sup_{\boldsymbol{\phi} \in \mbf{H}^1(\Gamma(t)) \setminus \{ 0 \}} \frac{\int_{\Gamma(t)} q \divg \boldsymbol{\phi}}{\| \boldsymbol{\phi}\|_{\mbf{H}^1(\Gamma(t))} } \geq C \|q \|_{L^2(\Gamma(t))}. \label{infsup condition}
	\end{align}
\end{lemma}

\begin{theorem}
	\label{mixed form theorem 1}
	There exists a unique solution, $(\varphi,\mu,\ut, p)$, with $\varphi \in H^{1}_{H^{-1}} \cap L^2_{H^1}, \mu \in L^2_{H^1}, \ut \in H^1_{\mbf{H}^{-1}} \cap L^2_{\mbf{H}^1}, p \in L^2_{L^2}$, of \eqref{mixedTNSCH1}-\eqref{mixedTNSCH4}.
\end{theorem}
\begin{proof}
	To begin, let $(\varphi,\mu,\ut)$ be the unique solution of \eqref{weakTNSCH1}-\eqref{weakTNSCH3}, and define $\mathcal{F}(t)$ by
	\begin{multline*}
		\langle \mathcal{F}(t), \boldsymbol{\phi} \rangle_{\mbf{H}^{-1}(\Gamma(t)), \mbf{H}^1(\Gamma(t))}:= \langle \normdev \ut, \boldsymbol{\phi}\rangle_{\mbf{H}^{-1}(\Gamma(t)), \mbf{H}^1(\Gamma(t))} + \hat{\mbf{a}}(\eta(\varphi), \ut, \boldsymbol{\phi}) + \mbf{c}_1(\ut,\ut, \boldsymbol{\phi})\\
		+ \mbf{l}(\ut, \boldsymbol{\phi})
		+ \mbf{d}_1(\ut, \boldsymbol{\phi})+\mbf{d}_2(\eta(\varphi), \boldsymbol{\phi})- \mbf{m}(\mbf{B}, \boldsymbol{\phi}) - \mbf{c}_2(\mu,\varphi,\boldsymbol{\phi}),
	\end{multline*}
	for $\boldsymbol{\phi} \in \mbf{H}^1(\Gamma(t))$.
	We claim that $\mathcal{F} \in L^2_{\mbf{H}^{-1}}$.
	To see this, we recall the equivalence,
	\[ \normdev \boldsymbol{\phi} \in L^2_{\mbf{V}_\sigma'} \Leftrightarrow \normdev \boldsymbol{\phi} \in L^2_{\mbf{H}^{-1}}, \ \boldsymbol{\phi} \in L^2_{\mbf{V}_\sigma} \]
	from \cite{OlsReuZhi22}, and repeat various estimates we have used throughout.
	We elaborate on the estimates for $\mbf{c}_1(\ut,\ut, \boldsymbol{\phi}),\mbf{c}_2(\mu,\varphi,\boldsymbol{\phi})$, but skip further calculations.
	As $\boldsymbol{\phi} \in \mbf{H}^1(\Gamma(t))$, and not necessarily $\divfree{t}$, we cannot use the properties, $\mbf{c}_1(\ut, \ut, \boldsymbol{\phi}) = -\mbf{c}_1(\boldsymbol{\phi}, \ut, \ut)$ and $\mbf{c}_2(\mu,\varphi,\boldsymbol{\phi}) = -\mbf{c}_2(\varphi,\mu,\boldsymbol{\phi})$.
	However, by using the divergence theorem and the fact that $\partial\Gamma(t) = \emptyset$, one finds
	\begin{gather*}
		c_1(\ut,\ut,\boldsymbol{\phi}) = - c_1(\boldsymbol{\phi}, \ut, \ut) - \int_{\Gamma(t)} (\ut \cdot \ut) \divg \boldsymbol{\phi},\\
		c_2(\mu,\varphi,\boldsymbol{\phi}) = -c_2(\varphi,\mu,\boldsymbol{\phi})- \int_{\Gamma(t)} \varphi \mu \divg \boldsymbol{\phi},
	\end{gather*}
	and hence
	\begin{gather*}
		\int_0^T |c_1(\ut,\ut, \boldsymbol{\phi})| \leq C \sup_{t \in [0,T]} \| \ut \|_{\mbf{L}^2(\Gamma(t))} \left( \int_0^T \|\ut\|_{\mbf{H}^1(\Gamma(t))} \right)^{\frac{1}{2}} \left( \int_0^T \|\boldsymbol{\phi}\|_{\mbf{H}^1(\Gamma(t))} \right)^{\frac{1}{2}},\\
		\int_0^T |c_2(\mu,\varphi,\boldsymbol{\phi})|  \leq C\sup_{t \in [0,T]} \|\varphi\|_{H^1(\Gamma(t))} \left( \int_0^T \| \mu \|_{H^1(\Gamma(t))}^2 \right)^\frac{1}{2} \left( \int_0^T \| \boldsymbol{\phi} \|_{\mbf{H}^1(\Gamma(t))}^2 \right)^\frac{1}{2},
	\end{gather*}
	where we have used Sobolev embeddings and \eqref{ladyzhenskaya2} as appropriate.
	The other terms follow similar, but simpler, arguments.\\
	
	We now observe that from the inf-sup condition \eqref{infsup condition} that the distributional divergence, $\gradg :L_0^2(\Gamma(t)) \rightarrow \mbf{H}^{-1}(\Gamma(t)),$ has a closed range $R(\gradg) \subset \mbf{H}^{-1}(\Gamma(t))$.
    This follows from \eqref{infsup condition} and continuity of $\gradg$ as an operator.
    Now by the closed range theorem, see for example \cite{yosida2012functional} VII.5, we find that
	\[ R(\gradg) = \ker(\gradg^*)^\perp, \text{ where } \ker(\gradg^*) = \divfree{t}, \]
	where $\gradg^*$ is the adjoint of $\gradg$.\\
	
    Since $(\varphi,\mu,\ut)$ solves \eqref{weakTNSCH1}-\eqref{weakTNSCH3} for almost all $t \in [0,T]$, we see that \[\langle \mathcal{F}(t), \boldsymbol{\phi} \rangle_{\mbf{H}^{-1}(\Gamma(t)), \mbf{H}^1(\Gamma(t))} = 0\]
    for all $\boldsymbol{\phi} \in \divfree{t}$.
	Hence from the above we see that $\mathcal{F}(t) \in R(\gradg)$ for almost all $t \in [0,T]$.
	Thus there exists some $p \in L^2_0(\Gamma(t))$ such that $\gradg p = \mathcal{F}(t)$ in the distributional sense.
	The map $t \mapsto \| p\|_{L^2(\Gamma(t))}$ is measurable by the same logic as in the proof of \cite{OlsReuZhi22}, Theorem 4.2.
	Moreover, by using \eqref{infsup condition} we see that $p$ is unique and one has
	\[ \int_0^T \|p \|_{L^2(\Gamma(t))}^2 \leq C \int_0^T \|\mathcal{F} \|_{\mbf{H}^{-1}(\Gamma(t))}^2, \]
	where the latter term can be expressed in terms of $\varphi,\mu, \ut$.
\end{proof}

Lastly we want to return to the setting of non-solenoidal vectors.
Letting $(\varphi,\mu,\hatut, p)$ be the solution from the previous theorem, then by our construction of $\widetilde{\ut}$ and the bilinear forms $\mbf{d}_1, \mbf{d}_2$, it is clear that $\ut := \hatut - \widetilde{\ut}$ is such that
\begin{gather*}
	\begin{split}
		\langle \normdev \ut, \boldsymbol{\phi} \rangle_{\mbf{H}^{-1}(\Gamma(t)), \mbf{H}^1(\Gamma(t))} + \hat{\mbf{a}}(\eta(\varphi), \ut, \boldsymbol{\phi}) + \mbf{c}_1(\ut,\ut, \boldsymbol{\phi}) + \mbf{l}(\ut, \boldsymbol{\phi}) = \mbf{m}(\mbf{F}_T, \boldsymbol{\phi})\\
		+ m(p, \divg \boldsymbol{\phi}) +  \mbf{c}_2(\mu,\varphi,\boldsymbol{\phi}),
	\end{split}\\
	m(q,\divg \ut) = -m(q,HV_N),\\
	m_*(\normdev \varphi, \phi) + a(\mu, \phi) + \mbf{c}_2(\phi,\varphi,\ut)=0,\\
	m(\mu,\phi) = \varepsilon a(\varphi, \phi) + \frac{1}{\varepsilon} m(F'(\varphi), \phi),
\end{gather*}
for all $q \in L^2(\Gamma(t)), \phi \in H^1(\Gamma(t)), \boldsymbol{\phi} \in \mbf{H}^1(\Gamma(t))$ for almost all $t \in [0,T]$.
Moreover we find that $\varphi \in H^{1}_{H^{-1}} \cap L^2_{H^1}, \mu \in L^2_{H^1}, \ut \in H^1_{\mbf{H}^{-1}} \cap L^2_{\mbf{H}^1}, p \in L^2_{L^2}$.\\

The initial condition for $\varphi$ is unchanged, but the initial condition for $\ut$ is required to be such that $\ut(0) = \mbf{u}_{T,0} \in \widetilde{\ut} + \Hdivfree{0}$.
One deduces the appropriate regularity for $\ut$ from the regularity of $\hatut, \widetilde{\ut}$.
The above arguments also work for the logarithmic potential, but we omit further details.

\begin{remark}
	In this section we have not discriminated between the pressure, $p$, and the modified pressure, $\tilde{p}$, as it is largely beside the point - that is the existence of some Lagrange multiplier enforcing the divergence condition.
    The distinction between these two pressures is discussed in Section \ref{derivations}.
    Moreover, it is straightforward to establish that $p \in L^2_{L^2} \Leftrightarrow \tilde{p} \in L^2_{L^2}$.
\end{remark}

\section{Concluding remarks} \label{Conc}
We have derived a system coupling the Navier-Stokes equations with the Cahn-Hilliard equations on an evolving surface, and shown the well-posedness for a prescribed, sufficiently smooth normal evolution.
There is still much work to be done on this topic, which we expound upon here.\\

Firstly, for the (evolving surface) Cahn-Hilliard equations with a logarithmic potential one observes a ``separation from the pure phases'' where after some small time the solution, $\varphi$, is such that $|\varphi|<1-\xi$ for some small $\xi$ - as was shown in \cite{caetano2023regularization}.
This has been established for a Navier-Stokes-Cahn-Hilliard system on a stationary domain in \cite{giorgini2019uniqueness}, and so it seems reasonable it would extend to our setting.\\

If one does not prescribe the normal component of the velocity then the system \eqref{NSCH1}-\eqref{NSCH4} also contains a geometric evolution equation, \eqref{normalcomponent}, which one must solve.
Unlike more standard geometric evolution equations, for example mean curvature flow, this flow is essentially second order in time as one considers the material derivative of the normal velocity.
Indeed, even if one ignores the Cahn-Hilliard component of \eqref{NSCH1}-\eqref{NSCH4} there are, to the authors' knowledge, no results on the well-posedness of the evolving surface Navier-Stokes equations (with unknown normal component) as discussed in \cite{brandner2022derivations,OlsReuZhi22}.\\

Moreover, the model we have considered is a diffuse interface model - and depends strongly on the choice of the interface width, $\varepsilon$.
It is known that, in the sharp interface limit, $\varepsilon \rightarrow 0$, the zero-level set of the solution of the Cahn-Hilliard equation (with a constant mobility) converges in a suitably weak sense to the Mullins-Sekerka system, see \cite{alikakos1994convergence}.
Likewise, it is known that the analogous zero-level set from the Navier-Stokes-Cahn-Hilliard system converges to a coupled Navier-Stokes-Mullins-Sekerka system - see for instance \cite{abels2009existence,abels2013well}.
However such results, or even formal asymptotics, have not been obtained for the corresponding systems on an evolving surface - or even on a stationary surface, to our knowledge.
In particular, it would be interesting to study the sharp interface limit of \eqref{NSCH1}-\eqref{NSCH4}, as the limiting system should consist of a coupling been a Navier-Stokes type equation for the surface velocity coupled with the Mullins-Sekerka problem.\\

Lastly, there is interest in the numerical simulation of the system we have considered (with or without a prescribed normal velocity).
There has recently (see \cite{olshanskii2024eulerian}) been some numerical analysis of the tangential Navier-Stokes equations, where the authors discretise by using the TraceFEM method - but this has not yet been considered for the system \eqref{TNSCH1}-\eqref{TNSCH4}.
It therefore would be interesting to see how existing results for a stationary domain, for instance \cite{styles2008finite}, adapt to an evolving surface.

\subsection*{Acknowledgements}
The authors would like to thank Achilleas Mavrakis, Andrea Poiatti, and Arnold Reusken for discussions surrounding an earlier version of this paper.
Thomas Sales is supported by the Warwick Mathematics Institute Centre for Doctoral Training, and gratefully acknowledges funding from the University of Warwick and the UK Engineering and Physical Sciences Research Council (Grant number:EP/TS1794X/1).
For the purpose of open access, the author has applied a Creative Commons Attribution (CC BY) licence to any Author Accepted Manuscript version arising from this submission.

\appendix
\section{Laplace's equation on an evolving surface}
\label{evolvinglaplace}
In this appendix we consider the regularity of the solution of Laplace's equation on an evolving domain.
For $t \in [0,T]$ we define $\Psi(t)$ to be the unique weak solution of
\[-\lapg \Psi(t) = H(t)V_N(t),\]
on $\Gamma(t)$, subject to the constraint $\mval{\Psi}{\Gamma(t)}=0$.
We note that this is well defined since
\[\int_{\Gamma(t)} H(t) V_N(t) = \frac{d}{dt} |\Gamma(t)| = 0,\]
by assumption.\\

We recall the normal pushforward map as $\Phi_t^n : \Gamma_0 \rightarrow \Gamma(t),$ and $\Phi_{-t}^n$ denoting its inverse.
As these are $C^2$ diffeomorphisms the differentials $D \Phi_t^n(p) : T_p \Gamma_0 \rightarrow T_{\Phi_t^n(p)} \Gamma(t)$ are invertible.
We recall the notation $J(p,t) = \det( D \Phi_t^n(p))$, $J^{-1}(x,t) = \det( D \Phi_{-t}^n(x)) = J(t, \Phi_{-t}^n x)^{-1}$, $\mbb{D}(p,t) =  D \Phi_t^n(p) \mbb{P}(p,0)$, and $\mbb{D}^{-1}(x,t) =  D \Phi_{-t}^n(x) \mbb{P}(x,t)$.
These matrices are such that $\mbb{D}\mbb{D}^{-1} = \mbb{D}^{-1} \mbb{D} = \mbb{P}$.\\

\begin{lemma}
	Let $\Psi$ be as above, and $\Gamma(t)$ be a $C^3$ evolving surface with $|\Gamma(t)| = |\Gamma_0|$ for all $t \in [0,T]$.
 Then $\Psi \in C^0_{H^{3,p}} \cap C^1_{H^{1,p}},$ for all $p \in [1,\infty)$.
	
\end{lemma}
\begin{proof}
	Let $\chi \in H^1(\Gamma_0)$, then by the weak formulation of the above PDE and the compatibility of $(H^1(\Gamma(t)), \Phi_t^n)$ we find that
	\[ \int_{\Gamma(t)} \gradg \Psi(t) \cdot \gradg \Phi_t^n \chi = \int_{\Gamma{t}} H(t)V_N(t) \Phi_t^n \chi .\]
	Hence by pulling back the integrals onto $\Gamma_0$ we see that
	\[ \int_{\Gamma(t)} J(t) \mbb{D} \gradg \Phi_{-t}^n \Psi(t) \cdot \mbb{D} \gradg \chi = \int_{\Gamma_0} J(t)\Phi_{-t}^n(HV_N) \chi,\]
	where the operators now are $\nabla_{\Gamma_0}$.
	Similarly, the mean value condition transforms as
	\[0 = \int_{\Gamma(t)} \Psi(t) = \int_{\Gamma_0} J(t) \Phi_{-t}^n \Psi(t),\]
	and as such we focus on the function $\psi(t) := J(t) \Phi_{-t}^n \Psi(t)$, where we see $\psi(t) \in H^1(\Gamma_0)$ for all $t \in [0,T]$.
	We similarly write $f(t) := J(t) \Phi_{-t}^n (HV_n) \in H^1(\Gamma_0)$.
	It is then clear that $\psi(t)$ solves the PDE
	\begin{align}
		\int_{\Gamma_0} \tilde{\mbb{D}}(t) \gradg \psi(t) \cdot \gradg \chi  + \psi(t) \boldsymbol{\omega}(t) \cdot \gradg \chi = \int_{\Gamma_0} f(t) \chi, \label{gamma0pde}
	\end{align}
	for all $\chi \in H^1(\Gamma_0)$, where
	\[\tilde{\mbb{D}} = \mbb{D}^T\mbb{D}, \qquad \boldsymbol{\omega} = J(t)\tilde{\mbb{D}} \gradg (J(t)^{-1}).\]
	We note that clearly $\tilde{\mbb{D}}$ is positive definite, and the uniqueness of $\Psi$ implies uniqueness of $\psi$.\\
	
	Then our assumptions on $\Phi_t^n$ imply we have sufficient smoothness so that we may apply elliptic regularity theory to see that
	\[\| \psi(t) \|_{H^{3,p}(\Gamma_0)} \leq C \| f(t) \|_{H^{1,p}(\Gamma_0)},\]
    for $p \in [1,\infty)$ and $C$ depends on $p,\Gamma_0, \tilde{\mbb{D}}(t), \boldsymbol{\omega}(t)$.
	It is straightforward to see that by considering \eqref{gamma0pde} at two times $t,s \in [0,T]$, and noting that $\tilde{\mbb{D}}, \boldsymbol{\omega}$ are $C^2$ in $t$ and $f$ is $C^1$ in $t$, that the map $t \mapsto \| \psi(t) \|_{H^{3,p}(\Gamma_0)}$ is continuous on $[0,T]$.
	We omit further details on this calculation.\\
	
	Next we show that $\psi$ has a strong derivative.
	By considering \eqref{gamma0pde} at times $t \in [0,T)$ and $t+h$ for some small $h > 0$ so that $t+h \in (0,T)$ we find that
	\begin{multline*}
		\frac{1}{h}\int_{\Gamma_0} \left(\tilde{\mbb{D}}(t+h) \gradg \psi(t+h) - \tilde{\mbb{D}}(t) \gradg \psi(t)\right)\cdot \gradg \chi  + \left(\psi(t+h) \boldsymbol{\omega}(t+h) - \psi(t) \boldsymbol{\omega}(t) \right)\cdot \gradg \chi\\
		= \frac{1}{h}\int_{\Gamma_0} \left(f(t+h) - f(t)\right) \chi,  
	\end{multline*}
	for all $\chi \in H^1(\Gamma_0)$.
	We write this in terms of difference quotients as
	\begin{multline*}
		\int_{\Gamma_0} \left(\Delta_h \tilde{\mbb{D}}(t) \gradg \psi(t+h) + \tilde{\mbb{D}}(t) \gradg \Delta_h \psi(t)\right)\cdot \gradg \chi  + \left(\psi(t+h) \Delta_h\boldsymbol{\omega}(t) + \Delta_h\psi(t) \boldsymbol{\omega}(t) \right)\cdot \gradg \chi\\
		= \int_{\Gamma_0} \Delta_h f(t) \chi, 
	\end{multline*}
	where $\Delta_h X(t) = \frac{X(t+h) - X(t)}{h}$ for some quantity $X$.
    Now by letting $h, h' > 0$ be sufficiently small one readily finds that
	\begin{align*}
		\langle L(t)(\Delta_h \psi(t) - \Delta_{h'} \psi(t)), \Delta_h \psi(t) - \Delta_{h'} \psi(t) \rangle &\leq C\|\Delta_{h'} \tilde{\mbb{D}}(t) \gradg \psi(t+h') - \Delta_h \tilde{\mbb{D}}(t) \gradg \psi(t+h)\|_{L^2(\Gamma_0)}^2\\
		&+ C \|\psi(t+h') \Delta_{h'}\boldsymbol{\omega}(t) - \psi(t+h) \Delta_{h}\boldsymbol{\omega}(t)\|_{\mbf{L}^2(\Gamma_0)}^2\\
		&+ C \|\Delta_h f(t) - \Delta_{h'} f(t)\|_{L^2(\Gamma_0)}^2\\
		&+ \gamma \| \gradg (\Delta_h \psi(t) - \Delta_{h'} \psi(t))\|_{L^2(\Gamma_0)}^2,
	\end{align*}
	for some small $\gamma$ to be determined.
	Here $L(t) \in \mathcal{L}(H^1(\Gamma_0) \cap L_0^2(\Gamma_0),(H^1(\Gamma_0) \cap L_0^2(\Gamma_0))')$ is the operator defined so that
	\[\langle L(t)\zeta, \chi\rangle = \int_{\Gamma_0} \tilde{\mbb{D}}(t) \gradg \zeta \cdot \gradg \chi  + \zeta \boldsymbol{\omega}(t) \cdot \gradg \chi,\]
	where $L^2_0(\Gamma_0)$ is the subspace of $L^2(\Gamma_0)$ containing elements such that $\mval{\phi}{\Gamma_0} = 0$.
	By pushing the integral forward onto $\Gamma(t)$, in the reverse to the beginning of the proof, we can observe $L(t)$ is elliptic by the ellipticity of $-\Delta_{\Gamma(t)}$ on $H^1(\Gamma(t)) \cap L_0^2(\Gamma(t)),$ and moreover the ellipticity constant is independent of $t$.
	Thus there exists some constant $\kappa$ such that
	\[\langle L(t)(\Delta_h \psi(t) - \Delta_{h'} \psi(t)), \Delta_h \psi(t) - \Delta_{h'} \psi(t) \rangle \geq \kappa \| \gradg (\Delta_h \psi(t) - \Delta_{h'} \psi(t)) \|_{L^2(\Gamma_0)}^2,\]
	and hence choosing $\gamma = \frac{\kappa}{2}$, and using Poincar\'e's inequality on $\Gamma_0$ we see
	\begin{align*}
		\|\Delta_h \psi(t) - \Delta_{h'} \psi(t)\|_{H^1(\Gamma_0)}^2 &\leq C\|\Delta_{h'} \tilde{\mbb{D}}(t) \gradg \psi(t+h') - \Delta_h \tilde{\mbb{D}}(t) \gradg \psi(t+h)\|_{L^2(\Gamma_0)}^2\\
		&+ C \|\psi(t+h') \Delta_{h'}\boldsymbol{\omega}(t) - \psi(t+h) \Delta_{h}\boldsymbol{\omega}(t)\|_{\mbf{L}^2(\Gamma_0)}^2\\
		&+ C \|\Delta_h f(t) - \Delta_{h'} f(t)\|_{L^2(\Gamma_0)}^2,
	\end{align*}
	for some constants $C(t)$ depending on $\Gamma_0$, and the ellipticity of $L(t)$.
	Now, by the differentiability of $\tilde{\mbb{D}}, \boldsymbol{\omega}, f$, and the continuity of $\psi$ it is clear that by taking $h, h'$ sufficiently small that we can make $\|\Delta_h \psi(t) - \Delta_{h'} \psi(t)\|_{H^1(\Gamma_0)}$ arbitrarily small.
	Thus $\Delta_h \psi(t)$ is a Cauchy sequence in $H^1(\Gamma_0)$ and a right time derivative of $\psi$ exists at $t \in [0,T)$.
    A similar calculation verifies that a left time derivative exists too.\\
	
	Differentiating \eqref{gamma0pde} in time we find
	\begin{multline*}
		\int_{\Gamma_0} \ddt{\tilde{\mbb{D}}}(t) \gradg \psi(t) \cdot \gradg \chi + \tilde{\mbb{D}}(t) \gradg \ddt{\psi}(t) \cdot \gradg \chi  + \ddt{\psi}(t) \boldsymbol{\omega}(t) \cdot \gradg \chi + \psi(t) \ddt{\boldsymbol{\omega}}(t) \cdot \gradg \chi\\
		= \int_{\Gamma_0} \ddt{f}(t) \chi,
	\end{multline*}
	for all $\chi \in H^1(\Gamma_0)$, $t \in [0,T]$.
	As above, by noting that $\tilde{\mbb{D}}, \boldsymbol{\omega}$ are $C^2$ in $t$ and $f$ is $C^1$ in $t$, one can now readily observe that the map $t \mapsto \| \ddt{\psi}(t) \|_{H^1(\Gamma_0)}$ is continuous on $[0,T]$.
	Applying elliptic regularity theory we find that $\psi \in C^0([0,T];H^{3,p}(\Gamma_0)) \cap C^1([0,T]; H^{1,p}(\Gamma_0))$, and hence using the compatibility of $(H^{3,p}(\Gamma(t)), \Phi_t^n)$ (and uniform bounds on $J(t)$ where needed) it follows that $\Psi \in C^0_{H^{3,p}} \cap C^1_{H^{1,p}}$.
\end{proof}

\section{An inverse Stokes-type operator}
\label{inversestokes}
In this appendix we discuss a solution operator related to the surface Stokes equation.
We refer the reader to \cite{bonito2020divergence,JanOlsReu18} for further details.
For $t\in [0,T]$, and a given $\boldsymbol{\phi} \in \Hdivfree{t}$ we are interested in finding a solution $\mathcal{S}\boldsymbol{\phi} \in \divfree{t}$ solving 
\[ \mathcal{S} \boldsymbol{\phi} -\mbb{P} \gradg \cdot (2 \mbb{E}(\mathcal{S}\boldsymbol{\phi}) ) = \boldsymbol{\phi}, \text{ on } \Gamma(t),\]
in a weak sense.\\

For $\boldsymbol{\phi} \in \Hdivfree{t}$ we define $\mathcal{S} \boldsymbol{\phi} \in \divfree{t}$ to be the unique solution  to
\[\mbf{m}(\mathcal{S} \boldsymbol{\phi}, \boldsymbol{\psi}) + \mbf{a}(\mathcal{S} \boldsymbol{\phi},\boldsymbol{\psi} )= \mbf{m}(\boldsymbol{\phi},\boldsymbol{\psi}),\]
for all $\boldsymbol{\psi} \in \divfree{t}$.
This is clearly well-defined by \eqref{korn1}, and the Lax-Milgram theorem.
With this norm we define a norm on $\Hdivfree{t}$ by
\[ \| \boldsymbol{\phi}\|_{\mathcal{S}} := \left( \mbf{m}(\mathcal{S} \boldsymbol{\phi}, \mathcal{S} \boldsymbol{\phi}) + \mbf{a}(\mathcal{S} \boldsymbol{\phi},\mathcal{S} \boldsymbol{\phi} ) \right)^{\frac{1}{2}} = \mbf{m}( \boldsymbol{\phi},\mathcal{S} \boldsymbol{\phi})^{\frac{1}{2}}, \]
where it is straightforward to see that
\[ \| \boldsymbol{\phi}\|_{\mathcal{S}} \leq C \| \boldsymbol{\phi}\|_{\mbf{L}^2(\Gamma(t))}, \]
for a constant independent of $t$.
We now prove a result on the time-differentiability of this operator, analogous to \cite{elliott2015evolving}, Lemma 4.3.
\begin{lemma}
	If $\boldsymbol{\phi} \in L^2_{\mbf{H}_\sigma} \cap H^1_{\mbf{V}_\sigma'}$ then $\mathcal{S} \boldsymbol{\phi} \in H^1_{\mbf{H}^1}$, such that
	\begin{align}
		\int_0^T \|\normdev \mathcal{S} \boldsymbol{\phi}\|_{\mbf{H}^1(\Gamma(t))}^2 \leq C \left( \int_0^T \| \normdev \boldsymbol{\phi} \|_{\divfree{t}'}^2 + \| \boldsymbol{\phi} \|_{\mbf{L}^2(\Gamma(t))}^2 + \| \mathcal{S} \boldsymbol{\phi} \|_{\mbf{H}^1(\Gamma(t))}^2 \right).
	\end{align}
\end{lemma}
\begin{proof}
	To begin we formally differentiate the equation defining $\mathcal{S}$, with a test function $\boldsymbol{\psi} \in H^1_{\mbf{V}_\sigma}$, to obtain
	\begin{multline*}
		\mbf{m}(\normdev \mathcal{S} \boldsymbol{\phi}, \boldsymbol{\psi}) + \mbf{m}(\mathcal{S} \boldsymbol{\phi}, \normdev \boldsymbol{\psi}) + \mbf{g}( \mathcal{S} \boldsymbol{\phi}, \boldsymbol{\psi}) + \mbf{a}(\normdev \mathcal{S} \boldsymbol{\phi}, \boldsymbol{\psi}) + \mbf{a}( \mathcal{S} \boldsymbol{\phi}, \normdev \boldsymbol{\psi}) + \mbf{b}( \mathcal{S} \boldsymbol{\phi}, \boldsymbol{\psi})\\
		= \mbf{m}_*(\normdev \boldsymbol{\phi}, \boldsymbol{\psi}) + \mbf{m}(\boldsymbol{\phi}, \normdev \boldsymbol{\psi}) + \mbf{g}(\boldsymbol{\phi}, \boldsymbol{\psi}).
	\end{multline*}
	This then simplifies to
	\begin{align*}
		\mbf{m}(\normdev \mathcal{S} \boldsymbol{\phi}, \boldsymbol{\psi}) + \mbf{a}(\normdev \mathcal{S} \boldsymbol{\phi}, \boldsymbol{\psi})  = \mbf{m}_*(\normdev \boldsymbol{\phi}, \boldsymbol{\psi}) + \mbf{g}(\boldsymbol{\phi}, \boldsymbol{\psi}) - \mbf{g}( \mathcal{S} \boldsymbol{\phi}, \boldsymbol{\psi}) - \mbf{b}( \mathcal{S} \boldsymbol{\phi}, \boldsymbol{\psi}),
	\end{align*}
    or equivalently
    \begin{multline*}
		\mbf{m}(\pioladev \mathcal{S} \boldsymbol{\phi}, \boldsymbol{\psi}) + \mbf{a}(\pioladev \mathcal{S} \boldsymbol{\phi}, \boldsymbol{\psi})  = \mbf{m}_*(\normdev \boldsymbol{\phi}, \boldsymbol{\psi}) + \mbf{g}(\boldsymbol{\phi}, \boldsymbol{\psi}) - \mbf{g}( \mathcal{S} \boldsymbol{\phi}, \boldsymbol{\psi}) - \mbf{b}( \mathcal{S} \boldsymbol{\phi}, \boldsymbol{\psi})\\
  - \mbf{m}(\bar{\mbb{A}} \mathcal{S} \boldsymbol{\phi}, \boldsymbol{\psi}) - \mbf{a}(\bar{\mbb{A}} \mathcal{S} \boldsymbol{\phi}, \boldsymbol{\psi}),
	\end{multline*}
	which we see extends to $\boldsymbol{\psi} \in L^2_{\mbf{V}_\sigma}$.
 We use the formulation involving $\pioladev$, as it is not clear that one would have $\mbb{P}\normdev \mathcal{S} \boldsymbol{\phi} \in \divfree{t}$, but this is the case for $\pioladev \mathcal{S} \boldsymbol{\phi}$ by construction.
	From \eqref{korn1} and the Lax-Milgram theorem one finds that there exists a unique $\pioladev \mathcal{S} \boldsymbol{\phi} \in L^2_{\mbf{V}_\sigma}$, and hence $\mbb{P}\normdev \mathcal{S} \boldsymbol{\phi} \in L^2_{\mbf{H}^1}$ such that
	\begin{align*}
		\int_0^T \|\normdev \mathcal{S} \boldsymbol{\phi}\|_{\mbf{H}^1(\Gamma(t))}^2 \leq C \left( \int_0^T \| \normdev \boldsymbol{\phi} \|_{\divfree{t}'}^2 + \| \boldsymbol{\phi} \|_{\mbf{L}^2(\Gamma(t))}^2 + \| \mathcal{S} \boldsymbol{\phi} \|_{\mbf{H}^1(\Gamma(t))}^2 \right).
	\end{align*}
	Moreover it is straightforward to see that $\normdev \mathcal{S} \boldsymbol{\phi}$ is indeed the weak time derivative of $\mathcal{S} \boldsymbol{\phi}$.
\end{proof}

As in \cite{bonito2020divergence}, we have a sufficiently smooth surface, $\Gamma(t)$, so that one has improved regularity
\begin{align}
	\| \mathcal{S} \boldsymbol{\phi} \|_{\mbf{H}^2(\Gamma(t))} \leq C \|\boldsymbol{\phi}\|_{\mbf{L}^2(\Gamma(t)}, \label{stokes regularity}
\end{align}
and the constant $C$ is independent of $t$ by the usual arguments.

\bibliographystyle{acm}
\bibliography{LibraryBibdeskRefs_2023}

\end{document}